\documentclass[onefignum,onetabnum]{siamart190516}

\usepackage[utf8]{inputenc}

\usepackage{pgfplotstable}

\usepackage{pgfplots}
\pgfplotsset{compat=1.16}
\usepgfplotslibrary{groupplots,colorbrewer}
\usetikzlibrary{matrix,positioning,arrows,fit,patterns}
\usetikzlibrary{external}

\usepackage{tikz-3dplot}
\usepackage[caption=false]{subfig}
\usepackage{booktabs}
\usepackage{enumitem}
\usepackage{bbm}
\usepackage[normalem]{ulem}

\usepackage{lipsum}
\usepackage{amsfonts}
\usepackage{graphicx}
\usepackage{epstopdf}
\usepackage{algorithmic}
\ifpdf{}
  \DeclareGraphicsExtensions{.eps,.pdf,.png,.jpg}
\else
  \DeclareGraphicsExtensions{.eps}
\fi

\usepackage{silence}
\usepackage{silence}
\newsiamremark{remark}{Remark}
\newsiamremark{example}{Example}
\newsiamremark{hypothesis}{Hypothesis}
\crefname{hypothesis}{Hypothesis}{Hypotheses}
\crefname{lemma}{Lemma}{Lemmata}
\newsiamthm{claim}{Claim}

\usepackage{amsopn}
\DeclareMathOperator{\diag}{diag}

  {\end{smallmatrix}\right)} 

\usepackage{bm}

\newcommand{\dd}{\text{d}} 

\newcommand{\change}[1]{%
  \begingroup%
  \color{blue}#1%
  \endgroup%
}%
\headers{A Quasi-Optimal Periodic Schrödinger Preconditioner}{Benjamin Stamm and Lambert Theisen}

\title{A Quasi-Optimal Factorization Preconditioner for Periodic Schrödinger Eigenstates in Anisotropically Expanding Domains\thanks{\LaTeX{} compiled on \today.
\funding{This work was supported by the German Research Foundation (DFG) under project 411724963.}}}

\author{
Benjamin Stamm\thanks{
  Applied and Computational Mathematics, RWTH Aachen University, 52062 Aachen, Germany
  (\email{best@acom.rwth-aachen.de}).
}
\and
Lambert Theisen
\thanks{
  Corresponding author. Applied and Computational Mathematics, RWTH Aachen University, 52062 Aachen, Germany
  (\email{theisen@acom.rwth-aachen.de}).
}
}

\ifpdf{}
\hypersetup{
  pdftitle={A Quasi-Optimal Factorization Preconditioner for Periodic Schrödinger Eigenstates in Anisotropically Expanding Domains},
  pdfauthor={Benjamin Stamm, Lambert Theisen}
}

\begin{document}

\maketitle

\begin{abstract}
  This paper provides a provably quasi-optimal preconditioning strategy of the linear Schr\"odinger eigenvalue problem with periodic potentials for a possibly non-uniform spatial expansion of the domain. The quasi-optimality is achieved by having the iterative eigenvalue algorithms converge in a constant number of iterations for different domain sizes. In the analysis, we derive an analytic factorization of the spectrum and asymptotically describe it using concepts from the homogenization theory. This decomposition allows us to express the eigenpair as an easy-to-calculate cell problem solution combined with an asymptotically vanishing remainder. We then prove that the easy-to-calculate limit eigenvalue can be used in a shift-and-invert preconditioning strategy to bound the number of eigensolver iterations uniformly. Several numerical examples illustrate the effectiveness of this quasi-optimal preconditioning strategy.
\end{abstract}

\begin{keywords}
  Periodic Schrödinger Equation, Iterative Eigenvalue Solvers, Preconditioner, Asymptotic Eigenvalue Analysis, Factorization Principle, Directional Homogenization
\end{keywords}

\begin{AMS}
  65N25, 65F15, 65N30, 35B27, 35B40
\end{AMS}


\section{Introduction}\label{sec:introduction}
This paper considers the spectral problem for linear time-independent Schrödinger-type operators. Let us consider a parametrized family of \(d\)-dimensional boxes \(\Omega_L\) given by
\begin{align}
  \boldsymbol{z} &\in \Omega_L = {(0, L)}^{p} \times {(0, \ell)}^{q} = \Omega_{\boldsymbol{x}} \times \Omega_{\boldsymbol{y}} \subset \mathbb{R}^d
  ,
\end{align}
with coordinates \(\boldsymbol{z} := (\boldsymbol{x}, \boldsymbol{y}) = (x_1,\dots,x_p,y_1,\dots,y_q)\) and dimensions \(L, \ell \in \mathbb{R}\). Note that we provide an extension to arbitrary domains in \cref{ssec:experimentBarrierPotentialXDefects} and only keep the box shape for the early chapters of the analysis. We denote by $H^1_0(\Omega_L)$ the standard Sobolev space of index 1 with zero Dirichlet trace on \(\Omega_L\).
\par
Then we consider the eigenvalue problem: Find \((\phi_L, \lambda_L)\in (H^1_0(\Omega_L) \setminus \{0\})\times \mathbb R\), such that
\begin{align}\label{eq:schroedingerEquation}
  - \Delta \phi_L + V \phi_L &= \lambda_L \phi_L \quad \text{ in } \Omega_L
  .
\end{align}
Here, \(\phi_L\) and \(\lambda_L\) are the eigenfunctions and -values, respectively, while the function \(V\) encodes an external potential applied to the system. Typically, we are interested in computing some of the smallest eigenpairs of \cref{eq:schroedingerEquation}.
\par
For the analysis, we make the following assumptions about the potential:
\begin{description}
  \item[(A1)] The potential \(V\) is directional-periodic with a period of \(1\) in each expanding direction: \(V(\boldsymbol{x}, \boldsymbol{y}) = V\left(\boldsymbol{x} + \boldsymbol{i} , \boldsymbol{y}\right) \quad\forall (\boldsymbol{x},\boldsymbol{y}) \in \Omega_L, \boldsymbol{i} \in \mathbb{Z}^p\);
  \item[(A2)] The potential \(V\) is essentially bounded: \(V \in L^\infty(\Omega_L)\);
  \item[(A3)] The potential \(V\) is non-negative: \(V \ge 0 \text{ a.e.~in } \Omega_L\).
\end{description}
\par
Note that under the assumption (A2), we can always apply a constant spectral shift to the potential without affecting the eigenfunctions to fulfill (A3). Of course, (A1) is only chosen for simplicity and arbitrary periods are possible. Furthermore, we could extend the theory to general elliptic operators satisfying the properties of \cref{ssec:existence}. \cref{sfig:geometricSetup} presents the geometrical framework.
\par
We focus on the case of \(q\) fixed dimensions of length \(\ell\). In contrast, the other \(p\) dimensions expand with \(L \to \infty\). This geometric setup allows us to study chain-like (\(d=2, p=1\)) or plane-like (\(d=3, p=2\)) domains with \(L \to \infty\). These are the most common application cases. However, the setup is not limited to \(d \le 3\), and all results also hold in the general case.
\par
Suppose one aims to solve for the ground-state eigenpair (smallest eigenvalue). In that case, the convergence rate of typical numerical algorithms~\cite[p53]{baiTemplatesSolutionAlgebraic2000} depends on the fundamental ratio between the first and the second (\(\lambda_L^{(1)} < \lambda_L^{(2)}\)) eigenvalue
\begin{equation}
  r_L := |\lambda_L^{(1)}| / |\lambda_L^{(2)}| < 1
  .
\end{equation}
Our geometrical setup of \(\Omega_L\), with \({(0,L)}^p \to {(0,\infty)}^p\) and \({(0,\ell)}^q\) being fixed, can lead to a collapsing fundamental gap \(\lambda_L^{(1)} - \lambda_L^{(2)} \to 0\) and thus \(r_L \to 1\) as \(L \to \infty\). This will deteriorate the convergence rate in the limit. Therefore, the eigensolver routine needs more and more iterations to converge to a fixed tolerance as $L$ increases. To overcome this problem, we theoretically study the operator's spectrum in \cref{eq:schroedingerEquation} to construct a suitable shift-and-invert preconditioner~\cite[p193]{saadNumericalMethodsLarge2011}, such that the preconditioned fundamental gap \(r_L(\sigma) := |\lambda_L^{(1)} - \sigma| / |\lambda_L^{(2)} - \sigma|\) is uniformly bounded above by a constant \(C < 1\), for all \(L > L^*\).
For this strategy to work, we need to choose a shift \(\sigma\) based on the asymptotic behavior of the problem. As it turns out later, the quasi-optimal shift \(\sigma\) has to be the asymptotic eigenvalue \(\lambda_\infty := \lim_{L \to \infty} \lambda_L^{(1)}\).
Throughout this paper, we understand quasi-optimality in terms of eigensolver iterations, which belong to \(\mathcal{O}(1)\) for all \(L\). This complexity is optimal except for an \(L\)-independent multiplicative constant. Also, note that we follow~\cite[p193]{saadNumericalMethodsLarge2011} and understand preconditioning in the eigenvalue context as a mechanism to speed up the convergence of an iterative solver by applying a spectral transformation~\cite[p43]{baiTemplatesSolutionAlgebraic2000}.
%

\subsection{Motivation: Collapsing Fundamental Gap for the Laplace Eigenvalue Problem}\label{ssec:motivation}
Even for the simplest potential satisfying the assumptions (A1)--(A3), namely \(V(\boldsymbol{z})=0\), the fundamental gap \(r_L\) decreases if only a subset of directions in \(\Omega_L\) is increased. The simplicity of the Laplace eigenvalue problem allows us to highlight the main challenges by considering the explicitly known eigenvalues. For \(p=1\) and \(q=1\), the pure Laplace eigenvalue problem has eigenvalues \(\lambda_L^{(1)} = \pi^2/L^2 + \pi^2/\ell^2, \lambda_L^{(2)} = 4\pi^2/L^2 + \pi^2/\ell^2\). It is then evident that \(\lim_{L \to \infty} \lambda_L^{(1)} / \lambda_L^{(2)} = 1\) for \(L \to \infty\). 
Thus, this leads to a decreasing converge rate \(r_L\). The collapsing fundamental gap is visible in the eigenvalue lattice illustrated in \cref{sfig:laplaceSpectrum}. In such a representation, each point represents an eigenvalue \(\lambda_L^{(m)}\), and its distance to the origin corresponds to \(\text{sqrt}{(\lambda_L^{(m)}/\pi)}\). All eigenvalues will form continuous \(x\)-parallel lines for the asymptotic case of \(L \to \infty\). For this Laplace eigenvalue problem, a shift of \(\sigma = \lambda_\infty = \pi^2/l^2\) would lead to \(\lim_{L \to \infty} r_L(\sigma) = 1/4 < 1\).
\par
This simple example serves as a motivation. However, to give a systematic approach to choosing the asymptotic correct shift \(\sigma\) in the general case, we develop a framework for characterizing the asymptotic behavior of the spectral properties for Schrödinger operators with periodic potentials satisfying (A1)--(A3). Knowing the asymptotic behavior will allow solving the algebraic eigenvalue problem within only a constant number of eigensolver iterations. 
\begin{figure}[t]
  \centering
  \subfloat[
    Geometric setup of \(\Omega_L\).\label{sfig:geometricSetup}
  ]{%

\tdplotsetmaincoords{70}{120}%
\begin{tikzpicture}[tdplot_main_coords,line join=bevel, scale=0.55]%
\pgfmathsetmacro{\N}{5};%
\pgfmathsetmacro{\elll}{2};%
\pgfmathsetmacro{\eps}{0.5}; 
\filldraw[%
  fill=white%
]%
(-\eps+\eps-\eps,-\eps+\eps-\eps,\elll) -- (-\eps+\eps+\eps-\eps,-\eps+\eps-\eps,\elll) -- (-\eps+\eps+\eps-\eps,-\eps+\eps+\eps-\eps,\elll) -- (-\eps+\eps-\eps,-\eps+\eps+\eps-\eps,\elll) -- cycle;
\filldraw[
  fill=white
]
(-\eps+\eps-\eps,-\eps+\eps+\eps-\eps,0) -- (-\eps+\eps+\eps-\eps,-\eps+\eps+\eps-\eps,0) -- (-\eps+\eps+\eps-\eps,-\eps+\eps+\eps-\eps,\elll) -- (-\eps+\eps-\eps,-\eps+\eps+\eps-\eps,\elll) -- cycle;
\filldraw[
  fill=white
]
(-\eps+\eps+\eps-\eps,-\eps+\eps-\eps,0) -- (-\eps+\eps+\eps-\eps,-\eps+\eps-\eps,\elll) -- (-\eps+\eps+\eps-\eps,-\eps+\eps+\eps-\eps,\elll) -- (-\eps+\eps+\eps-\eps,-\eps+\eps+\eps-\eps,0) -- cycle;

\foreach \i in {1,...,\N}{
    \pgfmathsetmacro\iisint{abs(Mod(\i,1))<0.001}
    \pgfmathsetmacro\idist{\iisint*1+Mod(\i,1)}
    \pgfmathsetmacro\jisint{abs(Mod(\eps,1))<0.001}
    \pgfmathsetmacro\jdist{\jisint*1+Mod(\eps,1)}
    \pgfmathsetmacro\j{\eps}
    \filldraw[
      fill=white
    ]
    (0+\i-\idist,-\eps+\j-\jdist,\elll) -- (0+\idist+\i-\idist,-\eps+\j-\jdist,\elll) -- (0+\idist+\i-\idist,-\eps+\jdist+\j-\jdist,\elll) -- (0+\i-\idist,-\eps+\jdist+\j-\jdist,\elll) -- cycle;
    \filldraw[
      fill=white
    ]
    (0+\i-\idist,-\eps+\jdist+\j-\jdist,0) -- (0+\idist+\i-\idist,-\eps+\jdist+\j-\jdist,0) -- (0+\idist+\i-\idist,-\eps+\jdist+\j-\jdist,\elll) -- (0+\i-\idist,-\eps+\jdist+\j-\jdist,\elll) -- cycle;
    \filldraw[
      fill=white
    ]
    (0+\idist+\i-\idist,-\eps+\j-\jdist,0) -- (0+\idist+\i-\idist,-\eps+\j-\jdist,\elll) -- (0+\idist+\i-\idist,-\eps+\jdist+\j-\jdist,\elll) -- (0+\idist+\i-\idist,-\eps+\jdist+\j-\jdist,0) -- cycle;
}
  \foreach \j in {1,2,...,\N}{
    \pgfmathsetmacro\iisint{abs(Mod(\eps,1))<0.001}
    \pgfmathsetmacro\idist{\iisint*1+Mod(\eps,1)}
    \pgfmathsetmacro\jisint{abs(Mod(\j,1))<0.001}
    \pgfmathsetmacro\jdist{\jisint*1+Mod(\j,1)}
    \pgfmathsetmacro\i{\eps}
    \filldraw[
      fill=white
    ]
    (-\eps+\i-\idist,0+\j-\jdist,\elll) -- (-\eps+\idist+\i-\idist,0+\j-\jdist,\elll) -- (-\eps+\idist+\i-\idist,0+\jdist+\j-\jdist,\elll) -- (-\eps+\i-\idist,0+\jdist+\j-\jdist,\elll) -- cycle;
    \filldraw[
      fill=white
    ]
    (-\eps+\i-\idist,0+\jdist+\j-\jdist,0) -- (-\eps+\idist+\i-\idist,0+\jdist+\j-\jdist,0) -- (-\eps+\idist+\i-\idist,0+\jdist+\j-\jdist,\elll) -- (-\eps+\i-\idist,0+\jdist+\j-\jdist,\elll) -- cycle;
    \filldraw[
      fill=white
    ]
    (-\eps+\idist+\i-\idist,0+\j-\jdist,0) -- (-\eps+\idist+\i-\idist,0+\j-\jdist,\elll) -- (-\eps+\idist+\i-\idist,0+\jdist+\j-\jdist,\elll) -- (-\eps+\idist+\i-\idist,0+\jdist+\j-\jdist,0) -- cycle;
  }

\foreach \i in {0.5,1,2,...,\N}{
  \foreach \j in {1,2,...,\N}{
    \pgfmathsetmacro\iisint{abs(Mod(\i,1))<0.001}
    \pgfmathsetmacro\idist{\iisint*1+Mod(\i,1)}
    \pgfmathsetmacro\jisint{abs(Mod(\j,1))<0.001}
    \pgfmathsetmacro\jdist{\jisint*1+Mod(\j,1)}
    \filldraw[
      fill=white
    ]
    (0+\i-\idist,0+\j-\jdist,\elll) -- (0+\idist+\i-\idist,0+\j-\jdist,\elll) -- (0+\idist+\i-\idist,0+\jdist+\j-\jdist,\elll) -- (0+\i-\idist,0+\jdist+\j-\jdist,\elll) -- cycle;
    \filldraw[
      fill=white
    ]
    (0+\i-\idist,0+\jdist+\j-\jdist,0) -- (0+\idist+\i-\idist,0+\jdist+\j-\jdist,0) -- (0+\idist+\i-\idist,0+\jdist+\j-\jdist,\elll) -- (0+\i-\idist,0+\jdist+\j-\jdist,\elll) -- cycle;
    \filldraw[
      fill=white
    ]
    (0+\idist+\i-\idist,0+\j-\jdist,0) -- (0+\idist+\i-\idist,0+\j-\jdist,\elll) -- (0+\idist+\i-\idist,0+\jdist+\j-\jdist,\elll) -- (0+\idist+\i-\idist,0+\jdist+\j-\jdist,0) -- cycle;
  }
}

\filldraw[
  fill=blue!25
]
(0+\N-1,0+\N-1,\elll) -- (1+\N-1,0+\N-1,\elll) -- (1+\N-1,1+\N-1,\elll) -- (0+\N-1,1+\N-1,\elll) -- cycle;
\filldraw[
  fill=blue!25
]
(0+\N-1,1+\N-1,0) -- (1+\N-1,1+\N-1,0) -- (1+\N-1,1+\N-1,\elll) -- (0+\N-1,1+\N-1,\elll) -- cycle;
\filldraw[
  fill=blue!25
]
(1+\N-1,0+\N-1,0) -- (1+\N-1,0+\N-1,\elll) -- (1+\N-1,1+\N-1,\elll) -- (1+\N-1,1+\N-1,0) -- cycle;

\draw[
  red, thick,
] (-\eps,-\eps,\elll) -- (\N,-\eps,\elll) -- (\N,\N,\elll) -- (-\eps,\N,\elll) -- cycle;
\draw[
  red, thick,
] (\N,-\eps,0) -- (\N,-\eps,\elll) -- (\N,\N,\elll) -- (\N,\N,0) -- cycle;
\draw[
  red, thick,
] (\N,\N,0) -- (\N,\N,\elll) -- (-\eps,\N,\elll) -- (-\eps,\N,0) -- cycle;

\draw[
  blue
]
(0+\N-1,0+\N-1,\elll) -- (1+\N-1,0+\N-1,\elll) -- (1+\N-1,1+\N-1,\elll) -- (0+\N-1,1+\N-1,\elll) -- cycle;
\draw[
  blue
]
(0+\N-1,1+\N-1,0) -- (1+\N-1,1+\N-1,0) -- (1+\N-1,1+\N-1,\elll) -- (0+\N-1,1+\N-1,\elll) -- cycle;
\draw[
  blue
]
(1+\N-1,0+\N-1,0) -- (1+\N-1,0+\N-1,\elll) -- (1+\N-1,1+\N-1,\elll) -- (1+\N-1,1+\N-1,0) -- cycle;

\node[blue] at (\N,\N-0.5,\elll/2) {$\Omega_1$};
\node[red] at (\N,1,-0.5) {$\Omega_L$};

\draw[
  >=latex,
  <->
] (-\eps,\N+0.5,0) -- ++(0,0,\elll) node[anchor=west, midway]{$\ell$};
\draw[] (-\eps,\N+0.5-0.25,\elll) -- ++(0,0.5,0);
\draw[] (-\eps,\N+0.5-0.25,0) -- ++(0,0.5,0);
\draw[
  >=latex,
  <->
] (-\eps,-0.5-\eps,\elll) -- ++(\N+\eps,0,0) node[anchor=south east, midway]{$L$};
\draw[] (-\eps,-0.5-\eps-0.25,\elll) -- ++(0,0.5,0);
\draw[] (\N,-0.5-\eps-0.25,\elll) -- ++(0,0.5,0);
\draw[
  >=latex,
  <->
] (-0.5-\eps,-\eps,\elll) -- ++(0,\N+\eps,0) node[anchor=south, midway]{$L$};
\draw[] (-0.5-\eps-0.25,-\eps,\elll) -- ++(0.5,0,0);
\draw[] (-0.5-\eps-0.25,\N,\elll) -- ++(0.5,0,0);

\draw[->]
  (\N,\N,\elll) -- ++ (-\N-1.0-\eps,0,0) node[anchor=south]{$x_1$};
\draw[->]
  (\N,\N,\elll) -- ++(0,-\N-1.0-\eps,0) node[anchor=north]{$x_2$};
\draw[->]
  (\N,\N,\elll) -- ++(0,0,-\elll-1.0) node[anchor=west]{$y_1$};
\end{tikzpicture}

  }
  \subfloat[
    Eigenvalue lattice for \(\Omega_L \subset \mathbb{R}^2, V=0\).\label{sfig:laplaceSpectrum}
  ]{%


\begin{tikzpicture}[scale=0.85, rotate=0]

  \def\Lx{1.5};
  \def\Ly{1.0};
  \def\iMax{7}
  \def\jMax{3}
  \def\ptTh{2.5pt}
  \def\xMax{5}
  \def\yMax{3.5}

  \draw [<->,thick] (0,\yMax) node (yaxis) [above] {$y = \frac{j}{\ell}$}
  |- (\xMax,0) node (xaxis) [right] {$x = \frac{i}{L}$};

  \draw (\xMax,\yMax) node [draw,rectangle,above left] {$L=\Lx$, $\ell=\Ly$};

  \foreach \i in {0, ..., \iMax} {
    \foreach \j in {0, ..., \jMax} {
      \coordinate (\i-\j) at ({\i/(\Lx)},{\j/(\Ly)});
    }
  }

  \def\ext{20pt}
  \foreach \j in {1, ..., \jMax} {
    \draw[dashed] (0-\j) node[left] {} -- ($(\iMax-\j)+(\ext,0)$) node[below] {};
  }

  \foreach \i in {1, ..., \iMax} {
    \draw [thick] ($(\i-0)+(0,2pt)$) -- ($(\i-0)-(0,2pt)$) node [below] {$\frac{\i}{L}$};
  }
  \foreach \j in {1, ..., \jMax} {
    \draw [thick] ($(0-\j)+(2pt,0)$) -- ($(0-\j)-(2pt,0)$) node [left] {$\frac{\j}{\ell}$};
  }

  \draw[thick] (0-0) node[left] {} -- (3-2) node[pos=0.73, right] {$\sqrt{\lambda_L^{(m)}}/\pi={\Vert (x,y) \Vert}_2$};

  \foreach \i in {1, ..., \iMax} {
    \foreach \j in {1, ..., \jMax} {
      \filldraw[red, draw=black] (\i-\j) circle (\ptTh);
    }
  }

  \filldraw[black, draw=black] (0-1) circle (\ptTh+0.75) node[above right] {$\lambda_\infty$};

  \def\alen{20pt}
  \def\adis{7pt}
  \draw[thick,->,red] ($(\iMax-\jMax)+(0,0)-(0,\adis)$) -- ($(\iMax-\jMax)-(\alen,0)-(0,\adis)$) node [midway,below] {$L \to \infty$};

\end{tikzpicture}

  }
  \caption{
    \cref{sfig:geometricSetup}: Geometric setup with \(p=2\) expanding directions with length \(L=5.5\) and \(q=1\) fixed dimensions with length \(\ell=2\).
    \cref{sfig:laplaceSpectrum}: The Dirichlet Laplacian spectrum on a rectangle domain \(\Omega_L = (0,L) \times (0,l)\) mapped to an eigenvalue lattice.
  }
\end{figure}

\subsection{State-of-the-Art and Context}\label{ssec:stateoftheart}
We can embed our results into existing research for three different aspects \textendash\ the considered model equation, the geometrical setup with the present periodic potential, and other mathematical analyses for related equations.
\par
First, the Schrödinger equation \cref{eq:schroedingerEquation} describes the stationary states of the wave function \(\phi_L\) for a quantum-mechanical system influenced by an external potential~\(V\). Therefore, example applications naturally arise in computational chemistry and quantum mechanics. Since the present model is one of the simpler models, it is only suitable for the direct simulation of basic quantum systems. In more elaborate electronic structure calculations, a nonlinear version of the Schrödinger equation is used. However, there is always a need to solve systems similar to \cref{eq:schroedingerEquation} in self-consistent field (SCF) iterations~\cite{heidGradientFlowFinite2021,cancesConvergenceAnalysisDirect2021,herbstBlackboxInhomogeneousPreconditioning2021} or other iteration schemes to solve these nonlinear Schrödinger-type equations~\cite{antoineGPELabMatlabToolbox2014,antoineGPELabMatlabToolbox2015,antoineEfficientSpectralComputation2017}.
One of such examples is the Bose--Einstein condensate, either modeled with random or disorder potentials~\cite{altmannLocalizedComputationEigenstates2019,altmannQuantitativeAndersonLocalization2020} or modeled with the Gross--Pitaevskii eigenvalue problem~\cite{altmannJmethodGrossPitaevskii2021,altmannLocalizationDelocalizationGround2022,altmannQuantitativeAndersonLocalization2020,henningSobolevGradientFlow2020}. Other applications for the same equation \cref{eq:schroedingerEquation} arise in studying the power distribution in a nuclear reactor core~\cite{allaireHomogenizationSpectralProblem2000,allaireAnalyseAsymptotiqueSpectrale1997,freitagConvergenceInexactInverse2007}.
\par
Second, when it comes to the geometric setup of only a subset of dimensions expanding, applications arise, for example, in material science to study the electronic properties of plane-like, layered~\cite{cazeauxEnergyMinimizationTwo2020,carrRelaxationDomainFormation2018} or chain-like structures, such as carbon nanotubes~\cite{ajayanApplicationsCarbonNanotubes2001} or polymers chains~\cite{vuConjugatedPolymersSystematic2018}.
\par
Third, from a mathematical point of view, the study of elliptic operators in the context of source~\cite{dongDirectionalHomogenizationElliptic2020,dongRegularityEllipticSystems2018,suslinaHomogenizationPeriodicElliptic2004,chipotAsymptoticBehaviourElliptic2004,chipotASYMPTOTICBEHAVIOURSOLUTION2002,chipotGoesInfinityUpdate2014,chipotNumericalApproximationPoisson2021} or eigenvalue problems~\cite{zhangEstimatesEigenvaluesEigenfunctions2021,cancesSecondorderHomogenizationPeriodic2021} with homogenization is closely related since \(V\) is periodic (at least directional in our setup). Also, for eigenvalue problems with periodic coefficients, results in~\cite{chipotAsymptoticsEigenstatesElliptic2013,chipotEigenvaluesEigenfunctionsDomains2007,chipotAsymptoticBehaviourEigenmodes2008} show the presence of an asymptotic limit when the domain expands in some directions to infinity. Finally, from a technical point of view, our analysis in \cref{sec:factorization} extends aspects of the work by Allaire et al.\ in~\cite{allaireHomogenizationLocalization1D2002,allaireHomogenizationSpectralProblem2000,allaireHomogenizationSchrodingerEquation2005,allaireHomogenizationCriticalitySpectral1999,allaireAnalyseAsymptotiqueSpectrale1997,allaireBlochWaveHomogenization1998}. Especially the concept of factorization will be one of the main techniques in our analysis. It allows us to analytically describe the spectrum of the system in terms of easier cell problems. This idea traces back to~\cite{vanninathanHomogenizationEigenvalueProblems1981,kesavanHomogenizationEllipticEigenvalue1979,kesavanHomogenizationEllipticEigenvalue1979a}.

\subsection{Contribution and Main Results}\label{ssec:contribution}
This work aims to provide a numerical framework to solve the eigenvalue problem \cref{eq:schroedingerEquation} in a fixed number of eigensolver iterations for all domain sizes \(L \to \infty\). We, therefore, propose a shift-and-invert strategy with a quasi-optimally chosen shift. The theoretical derivation of this particular shift is based on the following factorization of the eigenfunctions (see \cref{thm:factorization} for a more precise statement)
\begin{alignat}{3}
  \phi_L^{(m)}
  &=
  \psi \cdot u_{y,1} \cdot u_{y,2}^{(m)}
  &&=
  \varphi_y \cdot u_{y,2}^{(m)}
  ,
  \label{eq:mainResult1}
  \\
  \lambda_L^{(m)}
  &=
  \lambda_{\psi} + \lambda_{u_{y,1}} + \lambda_{u_{y,2}}^{(m)}
  &&=
  \lambda_{\varphi_y} + \lambda_{u_{y,2}}^{(m)} = \lambda_{\varphi_y} + \mathcal{O}(1/L^2)
  .
  \label{eq:mainResult2}
\end{alignat}
The above characterization highlights that we can simply use \(\lambda_{\varphi_y}\) as the quasi-optimal shift since we will show that the remaining term \( \lambda_{u_{y,2}}^{(m)}\) tends to zero as \(L \to \infty\) for all~\(m\). The \(\mathcal{O}(1/L^2)\) contribution in \cref{eq:mainResult2} is not uniform in \(m\) since it depends on the \(m\)-th eigenvalue of a homogenized equation, as shown later in \cref{thm:asymptoticBehaviorOfXDirectionNew}. In contrast to existing literature, this statement considers the case where only a subset of dimensions expands, and the periodicity is directional, which is essential given the potential practical applications. The eigenpair \((\varphi_y, \lambda_{\varphi_y})\) can be obtained in a constant time since it does not depend on \(L \to \infty\) as it is a solution to a fixed-size spectral cell problem. We then show that this eigenpair is the asymptotic limit as
\begin{equation}
  \lim_{L \to \infty} \lambda_L^{(m)} = \lambda_{\varphi_y}
  .
\end{equation}
These results, then, directly imply that the preconditioned fundamental ratio \(r_L\) is uniformly bounded from above by a constant for all \(L > L^*\), which is smaller than one (see \cref{thm:fundamentalRatioBoundedNewNew} for a more precise statement). Since the convergence speed of the iterative eigensolvers depends on precisely this ratio, they converge in a constant number of iterations, and our goal is achieved.
\par
The main challenges arise in the analysis and the quasi-optimality proof of the preconditioner.
Using factorization results to construct a quasi-optimal preconditioner for an anisotropic domain size increase, which can be computed in \(\mathcal{O}(1)\), is not covered, up to our knowledge, in the existing literature. Moreover, although the idea of factorization is not new (see the references in \cref{ssec:stateoftheart}), it was not yet applied in the context of weighted and thus potentially degenerate Sobolev spaces. However, precisely this setup of degenerate weights is necessary for our quasi-optimality analysis since the zero Dirichlet boundary conditions on the \(\boldsymbol{y}\)-boundary significantly contribute to the asymptotic behavior of the spectrum.
In addition, proving quasi-optimality also requires uniform bounds of the preconditioned system's fundamental ratio. Thus, another challenge is interpreting the expanding problem as a homogenization problem in a degenerate situation. This observation allows us to explicitly determine the asymptotic behavior of the \(m\)-dependent contribution in \cref{eq:mainResult1,eq:mainResult2}. However, unlike the classical homogenization theory, our homogenized limit is purely determined by an \(p\)- rather than an \((p+q)\)-dimensional problem, which seems to be a unique specialty of anisotropic expansion problems.

\subsection{Outline}
In \cref{sec:factorization}, we present the theoretical framework for the factorization approach in \cref{ssec:factorization}, which allows us to consider the remaining simplified problems using the theory of homogenization in \cref{ssec:homogenization} to derive quasi-optimality statements. We then discretize and solve the eigensystem in \cref{sec:discretizationAndIterativeEigensolvers} and show that the theoretical results also hold when specific subspace properties are met in the discrete setting. Next, \cref{sec:numericalExamples} presents various numerical examples and shows the relevance of the method for solving practical problems. Finally, we conclude with some remarks and point out future work in \cref{sec:conlusion}.

\section{Factorization and Homogenization of the Model Problem}\label{sec:factorization}
This chapter will use factorization and homogenization to derive the asymptotic spectrum, which allows us to specify the limit eigenvalue \(\lambda_{\varphi_y} = \lim_{L \to \infty} \lambda_L^{(m)}\).

\subsection{Existence and Regularity Results}\label{ssec:existence}
The second-order partial differential operator in \cref{eq:schroedingerEquation} is self-adjoint (by the symmetry of the diffusion matrix \(\delta_{ij}\)), is positive-definite (by ellipticity/coercivity of the Laplacian plus a non-negative potential), and has bounded coefficients (since \(V \in L^\infty(\Omega_L)\)). We can therefore recall the classical existence results from~\cite{gilbargEllipticPartialDifferential2001,vanninathanHomogenizationEigenvalueProblems1981,henningSobolevGradientFlow2020} to establish the well-posedness of our problem, that there exists a sequence (\(m=1,\dots,\infty\)) of eigenvalues with finite multiplicity \(\lambda_L^{(m)}\) and a sequence of eigenfunctions \(\phi_L^{(m)}\) (orthogonal basis of \(H_0^1(\Omega_L)\)) such that \(0 < \lambda_L^{(1)} < \lambda_L^{(2)} \le \lambda_L^{(3)} \le \dots \to \infty\) and \(\phi_L^{(1)} > 0\) a.e.\ in \(\Omega_L\).

\subsection{Factorization of the Eigenfunctions and Eigenvalues}\label{ssec:factorization}
Our next step is to establish factorizations to solve the eigenvalue problem \cref{eq:schroedingerEquation}. Thus, we describe splitting an eigenfunction into a product of two or more functions and splitting an eigenvalue into the sum of two or more eigenvalues. This splitting can be seen as a generalization to the separation of variables for the pure Laplacian case.
\par
Let \(\mathcal{D}(\Omega_L) = C^\infty_{\text{c}}(\Omega_L)\) be the space of compactly supported test functions on \(\Omega_L\) and \(\rho\) a weight function (measurable and positive a.e.~in \(\Omega_L\)). We then use the weighted Sobolev~\cite{kutnerApplicationsWeightedSobolev1987,vanninathanHomogenizationEigenvalueProblems1981} spaces as
\begin{align}\label{eq:weightedSobolevSpace}
  H^1(\Omega_L;\rho)
  &=
  \left\{
    u \in \mathcal{D}'(\Omega_L)
    \;
    \middle|\
    {\|u\|}_{ H^1(\Omega_L;\rho)} < \infty
  \right\}
  ,
  \\
  H_{\mathcal{B}_{\boldsymbol{x}},\mathcal{B}_{\boldsymbol{y}}}^1(\Omega_L;\rho)
  &=
  \left\{
    u \in H^1(\Omega_L;\rho)
    \;
    \middle|\
    \left\{
      \begin{aligned}
        \mathcal{B}_{\boldsymbol{x}}(u) = 0 \text{ on } \partial {(0,L)}^p \times {(0,\ell)}^q
        \\
        \mathcal{B}_{\boldsymbol{y}}(u) = 0 \text{ on } {(0,L)}^p \times \partial {(0,\ell)}^q
      \end{aligned}
      \right.
  \right\}
  ,
\end{align}
for some general boundary operators \(\mathcal{B}_{\boldsymbol{x}},\mathcal{B}_{\boldsymbol{y}}\), which are equipped with the weighted norm
\begin{equation}
\begin{gathered}
  {\|\cdot\|}_{ H^1(\Omega_L;\rho)}
  =
  \sqrt{{\|\nabla \cdot\|}_{ L^2(\Omega_L;\rho)}^2 + {\|\cdot\|}_{ L^2(\Omega_L;\rho)}^2}
  ,
\end{gathered}
\end{equation}
using \({\|\cdot\|}_{ L^2(\Omega_L;\rho)}^2 := {\|\sqrt{\rho}\cdot\|}_{ L^2(\Omega_L)}^2\) in the classical \(L^2\)-sense. The corresponding scalar product is \({\langle \cdot, \cdot \rangle}_{L^2(\Omega_L;\rho)} = {\langle \rho \cdot, \cdot \rangle}_{L^2(\Omega_L)}\). For the boundary operators, we use \(\mathcal{B}_{\boldsymbol{x}}, \mathcal{B}_{\boldsymbol{y}} \in \{\mathcal{B}_d, \mathcal{B}_n, \mathcal{B}_\#\}\) with
\begin{align}
  \text{Dirichlet: } \mathcal{B}_d(u)(\boldsymbol{z})
  &=
  u(\boldsymbol{z})
  ,
  \\
  \text{Neumann: } \mathcal{B}_n(u)(\boldsymbol{z})
  &=
  \rho(\boldsymbol{z}) \nabla u(\boldsymbol{z}) \cdot \boldsymbol{n}(\boldsymbol{z})
  ,
  \\
  \text{Periodic: } \mathcal{B}_\#(u)(\boldsymbol{z})
  &=
  u(\boldsymbol{z}) - u\left({(z_i - n_i(\boldsymbol{z}) L)}_{i=1}^{p}, {(z_i - n_i(\boldsymbol{z}) \ell)}_{i=p+1}^{q}\right)
  ,
\end{align}
where the unit normal-vector is denoted by \(\boldsymbol{n}(\boldsymbol{z})\) for $\boldsymbol{z}\in\partial\Omega_L$.
\par
In the following, we use multiple eigenvalue problems and their solutions. Therefore, we unify the notation and introduce the abstract notation:
\begin{definition}[Prototype of a Schrödinger eigenvalue problem]\label{def:solutionPrototype}
  For \(\Omega_L = {(0,L)}^p \times {(0,\ell)}^q \subset \mathbb{R}^d\), \(0 \le \rho,V \in L^\infty(\Omega_L)\), and \(1/\rho \in L^1_{\textnormal{loc}}(\Omega_L)\), we define
  \begin{equation}
    \left(
      u^{(m)}_{\mathcal{B}_{\boldsymbol{x}}, \mathcal{B}_{\boldsymbol{y}}, \rho, V}(\Omega_L)
      ,
      \lambda^{(m)}_{\mathcal{B}_{\boldsymbol{x}}, \mathcal{B}_{\boldsymbol{y}}, \rho, V}(\Omega_L)
    \right)
    ,
  \end{equation}
  to be the \(m\)-th eigenpair (including multiplicities) of the generalized Schrödinger-type eigenvalue problem: Find \((u,\lambda) \in ( H_{\mathcal{B}_{\boldsymbol{x}},\mathcal{B}_{\boldsymbol{y}}}^1(\Omega_L;\rho) \setminus \{0\} ) \times \mathbb{R}\), such that
  \begin{equation}\label{eq:prototypeEigenvalueProb}
    - \nabla \cdot \left( \rho \nabla u \right) + V u = \lambda \rho u \text{ in } \Omega_L
    .
  \end{equation}
\end{definition}
\begin{remark}\label{rem:regularWeightImpliesClassicalSobolevSpaces}
  If the weight \(\rho\) is zero only at the boundary of \(\Omega_L\), then we have \(1/\rho \in L^1_{\textnormal{loc}}(\Omega_L)\), which implies that \(H_{\mathcal{B}_{\boldsymbol{x}},\mathcal{B}_{\boldsymbol{y}}}^1(\Omega_L;\rho)\) is a Banach space~\cite[p235]{kutnerApplicationsWeightedSobolev1987}. On the other hand, in the case of \(\rho\) being bounded from above and uniformly positive, i.e.,~\(0 < \mathsf{c} < \rho < \mathsf{C}\) a.e.\ in \(\Omega_L\), the \(\rho\)-weighted Sobolev space from \cref{eq:weightedSobolevSpace} is equivalent to the classical Sobolev space \(H_{\mathcal{B}_{\boldsymbol{x}},\mathcal{B}_{\boldsymbol{y}}}^1(\Omega_L;\rho) = H_{\mathcal{B}_{\boldsymbol{x}},\mathcal{B}_{\boldsymbol{y}}}^1(\Omega_L)\), and we omit \(\rho\) in the notation.
\end{remark}
\begin{remark}[Weak form]\label{rem:weakFormPrototype}
  The corresponding weak formulation of \cref{eq:prototypeEigenvalueProb} reads: Find \((u,\lambda) \in ( H_{\mathcal{B}_{\boldsymbol{x}},\mathcal{B}_{\boldsymbol{y}}}^1(\Omega_L;\rho) \setminus \{0\} ) \times \mathbb{R}\) such that
  \begin{equation}\label{eq:weakFormPrototype}
    \forall v \in H_{\mathcal{B}_{\boldsymbol{x}},\mathcal{B}_{\boldsymbol{y}}}^1(\Omega_L;\rho)
    : \quad
    \int_{\Omega_L} \rho \nabla u \cdot \nabla v \dd \boldsymbol{z}
    +
    \int_{\Omega_L} V u v \dd \boldsymbol{z}
    =
    \lambda \int_{\Omega_L} \rho u v \dd \boldsymbol{z}
    .
  \end{equation}
\end{remark}
\begin{remark}[Min-max characterization]
  Since the weight \(\rho\) is a scalar function, \cref{eq:prototypeEigenvalueProb} is self-adjoint for the presented boundary conditions, and we can express the eigenpair through the min-max characterization (c.f.~\cite{kesavanHomogenizationEllipticEigenvalue1979,courantMethodsMathematicalPhysics1989a}):
  \begin{align}
    \lambda^{(m)}
    &=
    \min_{
      \substack{W_m \subset H_{\mathcal{B}_{\boldsymbol{x}},\mathcal{B}_{\boldsymbol{y}}}^1(\Omega_L;\rho) \\ \text{dim} W_m = m}
    }
    \max_{
      \substack{u \in W_m \\ u \ne 0}
    }
    \mathcal{R}_{\rho, V}(u)
    ,
  \end{align}
  with the Rayleigh quotient, defined by
  \begin{equation}
    \mathcal{R}_{\rho,V}(u)
    =
    \frac{
      \int_{\Omega_L} \rho(\boldsymbol{z}) \nabla u(\boldsymbol{z}) \cdot \nabla u(\boldsymbol{z}) \dd \boldsymbol{z} + \int_{\Omega_L} V(\boldsymbol{z}) u^2(\boldsymbol{z}) \dd \boldsymbol{z}
    }{
      \int_{\Omega_L} \rho(\boldsymbol{z}) u^2(\boldsymbol{z}) \dd \boldsymbol{z}
    }
    ,
  \end{equation}
  which is identical for all the considered boundary conditions \(\mathcal{B}_{\boldsymbol{x}},\mathcal{B}_{\boldsymbol{y}} \in \{ \mathcal{B}_d, \mathcal{B}_n, \mathcal{B}_\#\}\) of \cref{def:solutionPrototype}.
\end{remark}
We define \(E^\#_{\boldsymbol{x}} : H^1_{\mathcal{B}_\#,\mathcal{B}_{\boldsymbol{y}}}(\Omega_1;\rho) \to H^1_{\mathcal{B}_\#,\mathcal{B}_{\boldsymbol{y}}}(\Omega_L;\rho)\) as the periodic extension operator in the \(\boldsymbol{x}\)-direction for an \(\boldsymbol{x}\)-periodic weight \(\rho\) and \(\mathcal{B}_{\boldsymbol{y}} \in \{ \mathcal{B}_d, \mathcal{B}_n, \mathcal{B}_\#\}\).
We are now prepared to state our first main theoretical result:
\begin{theorem}[Factorzation of eigenfunctions and summation of eigenvalues]\label{thm:factorization}
  The \(m\)-th eigenfunction of the Schrödinger eigenvalue problem \cref{eq:schroedingerEquation} can be factorized into
  \begin{alignat}{3}
    u_{\mathcal{B}_d,\mathcal{B}_d,1,V}^{(m)}
    =
    \psi \cdot u^{(m)}
    \label{eq:facVec1}
    &&&=
    \psi \cdot u_{x,1} \cdot u_{x,2}^{(m)}
    &=
    \varphi_x \cdot u_{x,2}^{(m)}
    \\
    &&&=
    \psi \cdot u_{y,1} \cdot u_{y,2}^{(m)}
    \label{eq:facVec3}
    &=
    \varphi_y \cdot u_{y,2}^{(m)}
  \end{alignat}
  while the \(m\)-th eigenvalue can be summed correspondingly as
  \begin{alignat}{3}
    \lambda_{\mathcal{B}_d,\mathcal{B}_d,1,V}^{(m)}
    =
    \lambda_{\psi} + \lambda_{u}^{(m)}
    &&&=
    \lambda_{\psi} + \lambda_{u_{x,1}} + \lambda_{u_{x,2}}^{(m)}
    &=
    \lambda_{\varphi_x} + \lambda_{u_{x,2}}^{(m)}
    \label{eq:fac1}
    \\
    &&&=
    \lambda_{\psi} + \lambda_{u_{y,1}} + \lambda_{u_{y,2}}^{(m)}
    &=
    \lambda_{\varphi_y} + \lambda_{u_{y,2}}^{(m)}
    \label{eq:fac3}
  \end{alignat}
  where
  \begin{alignat}{3}
    \psi
    &= E^\#_{\boldsymbol{x}} u_{\mathcal{B}_\#,\mathcal{B}_\#,1,V}^{(1)}(\Omega_1)
    ,
    \;
    & u^{(m)}
    &= u_{\mathcal{B}_d,\mathcal{B}_d,\psi^2,0}^{(m)}(\Omega_L)
    ,
    \;
    &
    \\
    u_{x,1}
    &= u_{\mathcal{B}_d,\mathcal{B}_\#,\psi^2,0}^{(1)}(\Omega_L)
    ,
    \;
    & u_{x,2}^{(m)}
    &= u_{\mathcal{B}_n,\mathcal{B}_d,\varphi_x^2,0}^{(m)}(\Omega_L)
    ,
    \;
    & \varphi_x
    &= u_{\mathcal{B}_d,\mathcal{B}_\#,1,V}^{(1)}(\Omega_L)
    ,
    \;
    \\
    u_{y,1}
    &= E^\#_{\boldsymbol{x}} u_{\mathcal{B}_\#,\mathcal{B}_d,\psi^2,0}^{(1)}(\Omega_1)
    ,
    \;
    & u_{y,2}^{(m)}
    &= u_{\mathcal{B}_d,\mathcal{B}_n,\varphi_y^2,0}^{(m)}(\Omega_L)
    ,
    \;
    & \varphi_y
    &= E^\#_{\boldsymbol{x}} u_{\mathcal{B}_\#,\mathcal{B}_d,1,V}^{(1)}(\Omega_1)
    .
  \end{alignat}
\end{theorem}
A graphical representation of \cref{thm:factorization} is presented in \cref{fig:factorization} for \(m=1\), where the scale separation of, e.g., \(\varphi_y^{(1)}\) into a short and \(u_{y,2}^{(1)}\) into a large scale is visible. Moreover, the first excited eigenfunctions with \(m \in \{2,3,4\}\) are visualized in \cref{fig:factorizationExcited}. Herein, the \(m\)-dependence entirely goes into the \(u_{y,2}^{(m)}\) function for the excited eigenfunctions since \(\varphi_y^{(1)}\) is fixed.
\begin{figure}[t]
  \centering
  \newcommand{\datapath}{./plots/factorization/decomp}%
  \input{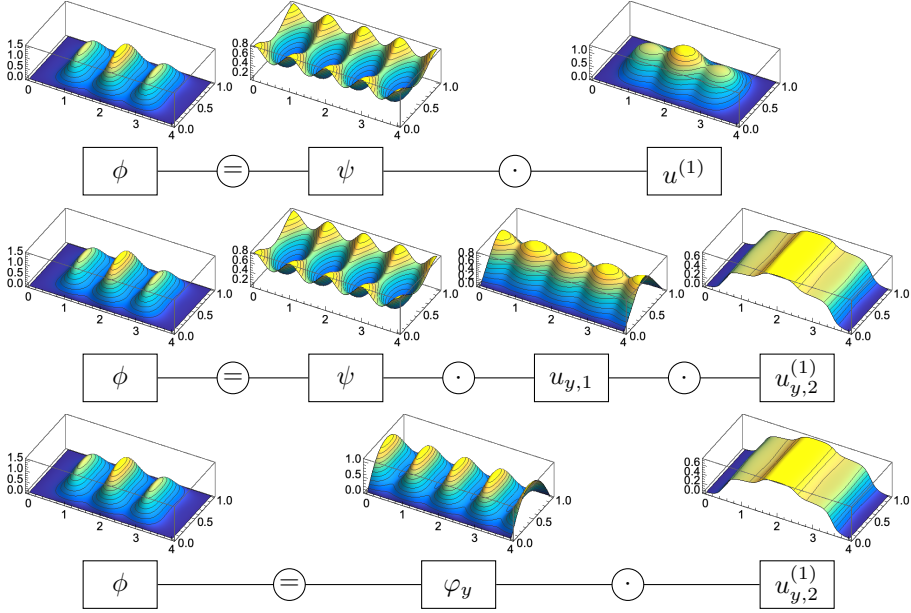}
  \caption{
    Visualization of the factorization for the ground state solution of \(-\Delta \phi + V \phi = \lambda \phi , \phi = 0 \;\mathrm{on}\;\partial\Omega_4\) with \(V(x,y) = 10^2 {(\sin{x})}^2 {(\sin{y})}^2\).
  }\label{fig:factorization}
\end{figure}
\begin{figure}[t]
  \centering
  \newcommand{\datapath}{./plots/factorization/decomp}%
  \input{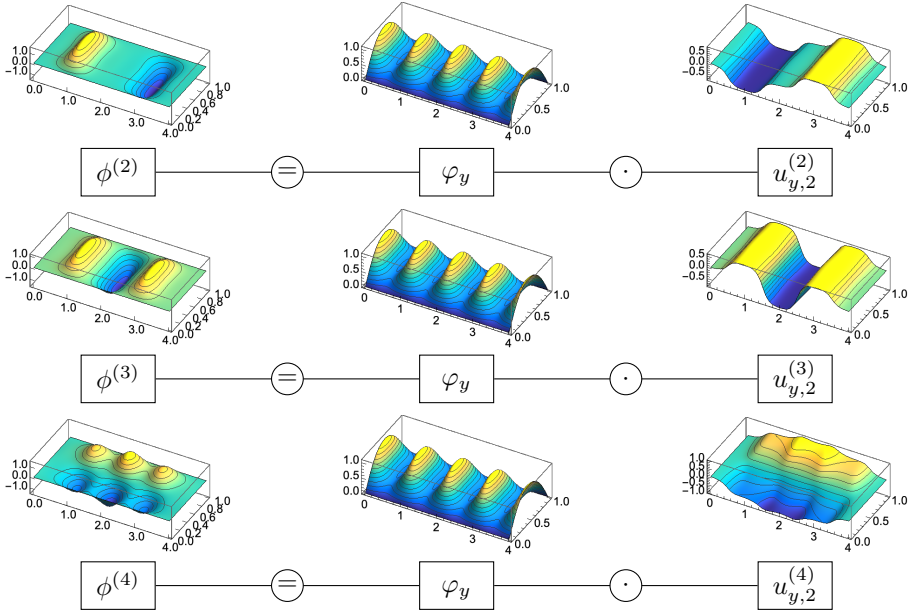}
  \caption{
    Visualization of the factorization for some excited states of \(-\Delta \phi^{(m)} + V \phi^{(m)} = \lambda^{(m)} \phi^{(m)}\) with \(V(x,y) = 10^2 {(\sin{x})}^2 {(\sin{y})}^2\): By construction, the \(m\)-dependence entirely goes into the \(u_{y,2}\) contribution.
  }\label{fig:factorizationExcited}
\end{figure}
\par
To prove \cref{thm:factorization}, we need to show that the factorizations are valid changes of variables. The application of the first \textit{factorization principle} in \cref{eq:facVec1}, i.e.,
\begin{equation}\label{eq:factorizationPrinciplePsi}
  u^{(m)} = u_{\mathcal{B}_d,\mathcal{B}_d,\psi^2,0}^{(m)} = u_{\mathcal{B}_d,\mathcal{B}_d,1,V}^{(m)} / \psi
  ,
\end{equation}
removes the potential \(V\) from the eigenvalue problem while still encoding the corresponding information through \(\psi^2\). The inducing function \(\psi = E^\#_{\boldsymbol{x}} u_{\mathcal{B}_\#,\mathcal{B}_\#,1,V}^{(1)}(\Omega_1)\) is the solution to a spectral cell problem, and it was shown in~\cite{allaireHomogenizationSpectralProblem2000} that this factorization is indeed a diffeomorphism in \(H_0^1(\Omega_L)\).
Such factorization operators will apply a change of variables in the min-max characterization of the eigenvalue.
\par
In contrast to \cref{eq:factorizationPrinciplePsi} where \(\psi^2\) is bounded from below a.e.\ by a positive constant, we will also need factorization operators induced by functions tending to zero at boundary parts due to homogeneous Dirichlet boundary conditions. In such cases, we need to adapt the factorization principle as the boundedness of the division operator is not directly visible, which makes the analysis much more subtle. Thus, we need:
\begin{lemma}[Factorization operator with degeneracy and singularity]\label{lemma:facOperator}
  Let the inducing function \(u_{y,1} := E^\#_{\boldsymbol{x}} u_{\mathcal{B}_{\#},\mathcal{B}_d,\psi^2,0}^{(1)}(\Omega_1) \in H_{\mathcal{B}_{\#},\mathcal{B}_d}^1(\Omega_L)\) with \(0 < \mathsf{c} < \psi^2 < \mathsf{C}\) a.e.\ be given. Then, the linear factorization operator defined by
  \begin{equation}
  \begin{aligned}\label{eq:factorizationOperator}
    T: H_{\mathcal{B}_d,\mathcal{B}_d}^1(\Omega_L) &\to H_{\mathcal{B}_d,\mathcal{B}_n}^1(\Omega_L; u_{y,1}^2)
    \\
    u
    &\mapsto
    T\left( u \right)
    :=
    \boldsymbol{z} \mapsto u(\boldsymbol{z}) / u_{y,1}(\boldsymbol{z}) \text{ a.e.~in } \Omega_L
  \end{aligned}
  \end{equation}
  is bi-continuous and thus a diffeomorphism.
\end{lemma}
\begin{proof}
  Noting that \(T\) is a division operator, the corresponding multiplication operator \(T^{-1}\) is the (left and right) inverse operator. For a simpler notation, we use \(H_0^1(\Omega_L) = H_{\mathcal{B}_d,\mathcal{B}_d}^1(\Omega_L)\). So we study
  \begin{equation}
  \begin{aligned}\label{eq:abstractOperator}
    \tilde{T} &: H^1_0(\Omega_L) \to W
    &
    \tilde{T}^{-1} &: W \to H^1_0(\Omega_L)
    \\
    u &\mapsto u_{y,2} := \tilde{T}(u) = \frac{u}{u_{y,1}}
    &
    u_{y,2} &\mapsto u := \tilde{T}^{-1}(u_{y,2}) = u_{y,1} u_{y,2}
  \end{aligned}
  \end{equation}
  with the abstract set \(W := \text{Im}(\tilde{T}) = \text{Dom}(\tilde{T}^{-1})\) as
  \begin{equation}
    W = \left\{ u_{y,2} \in \mathcal{D}'(\Omega_L) \;\middle|\; \exists u \in H^1_0(\Omega_L) : u_{y,2} = \tilde{T}(u) \right\}.
  \end{equation}
  We show that the image of \(\tilde{T}\) is the \(u_{y,1}^2\)-induced space \(H_{\mathcal{B}_d,\mathcal{B}_n}^1(\Omega_L;u_{y,1}^2)\).
  We have \(W \subset \mathcal{D}'(\Omega_L)\) since \(1/u_{y,1} \in L^1_{\textnormal{loc}}(\Omega_L)\) as \(u_{y,1}=0\) only at the \(\boldsymbol{y}\)-boundary, and therefore \(\forall \phi \in \mathcal{D}(\Omega_L) = C^\infty_{\text{c}}(\Omega_L)\), \(\tilde{T}(\phi) = \int_{\Omega_L} \frac{1}{u_{y,1}} \phi < {\|\phi\|}_\infty \int_{\text{supp}(\phi)} \frac{1}{u_{y,1}} < \infty\) is a linear bounded functional since \(\text{supp}(\phi)\) is compact. Then we have \(W \subset H_{\mathcal{B}_d,\mathcal{B}_n}^1(\Omega_L;u_{y,1}^2)\) since \(u_{y,2} \in W\) implies that there exists an \(u \in H_0^1(\Omega_L)\), such that \( u_{y,2} = \tfrac{u}{u_{y,1}}\), which yields \(\mathcal{B}_d(u_{y,2})=0\) on \(\partial \Omega_{\boldsymbol{x}}\) and trivially fulfilled \(\mathcal{B}_n(u_{y,2})=0\) on \(\partial \Omega_{\boldsymbol{y}}\), and allows us to take
  \begin{equation}
    \tilde{T}^{-1}(u_{y,2}) = u_{y,1} u_{y,2} = \frac{u_{y,1}}{u_{y,1}} u = u \in H^1_0(\Omega_L) \Rightarrow u_{y,2} \in H_{\mathcal{B}_d,\mathcal{B}_n}^1(\Omega_L;u_{y,1}^2)
    .
  \end{equation}
  We also have \(H_{\mathcal{B}_d,\mathcal{B}_n}^1(\Omega_L;u_{y,1}^2) \subset W\) since \(u_{y,2} \in H_{\mathcal{B}_d,\mathcal{B}_n}^1(\Omega_L;u_{y,1}^2)\) implies \(\tilde{T}^{-1}(u_{y,2}) = u_{y,1} u_{y,2} = u \in H^1_0(\Omega_L)\)
  such that
  \begin{equation}
    \exists u \in H^1_0(\Omega_L) : u_{y,2} = \frac{u}{u_{y,1}} = \tilde{T}(u) \Rightarrow u_{y,2} \in W
    .
  \end{equation}
  Thus \(W = H_{\mathcal{B}_d,\mathcal{B}_n}^1(\Omega_L;u_{y,1}^2)\), so \(\tilde{T}\) defined by from \cref{eq:abstractOperator} coincides with \(T\) from \cref{eq:factorizationOperator}. Moreover, the weighted Sobolev space \(H_{\mathcal{B}_d,\mathcal{B}_n}^1(\Omega_L;u_{y,1}^2)\) is a Banach space since \(u_{y,1}\) only degenerates on the boundary (c.f.~\cref{rem:regularWeightImpliesClassicalSobolevSpaces}).
  \par
  Since \(T\) and \(T^{-1}\) are both surjective, \(T\) is bijective.
  We now show the continuity of \(T^{-1}\) as this is the more straightforward direction being a multiplication operator. Since all \(u_{\mathcal{B}_d,\mathcal{B}_d,\psi^2,0}^{(m)}\) form a basis of \(H_0^1(\Omega_L)\) (c.f. \cref{ssec:existence}), it suffices to show continuity for all basis functions \(u_{\mathcal{B}_d,\mathcal{B}_d,\psi^2,0}^{(m)}\) and to conclude by the linearity of the operator \(T^{-1}\). Thus, let \(m\) be fixed, \(u := u_{\mathcal{B}_d,\mathcal{B}_d,\psi^2,0}^{(m)}\), and \(u_{y,2} := T(u_{\mathcal{B}_d,\mathcal{B}_d,\psi^2,0}^{(m)})\). Then it follows for \(u = T^{-1}(u_{y,2}) = u_{y,1} u_{y,2}\) that
  \begin{align}\label{eq:10}
    \int_{\Omega_L} \psi^2 \nabla u \cdot \nabla u \dd \boldsymbol{z}
    =
    \int_{\Omega_L} \psi^2 \left( u_{y,1}^2 \nabla u_{y,2} \cdot \nabla u_{y,2} + \nabla u_{y,1} \cdot \nabla (u_{y,2}^2 u_{y,1}) \right) \dd \boldsymbol{z}
    ,
  \end{align}
  which is well defined by the assumption that \(u\) is the corresponding eigenfunction for the first expression in \cref{eq:10}. We further have for the last term in \cref{eq:10} that
  \begin{equation}
  \begin{aligned}
    \int_{\Omega_L} \psi^2 \nabla u_{y,1} \cdot \nabla (u_{y,2}^2 u_{y,1}) \dd \boldsymbol{z}
    &=
    - \int_{\Omega_L} \nabla \cdot \left( \psi^2 \nabla u_{y,1} \right) u_{y,2}^2 u_{y,1} \dd \boldsymbol{z}
    \\
    &=
    \lambda_{u_{y,1}}^{(1)} \int_{\Omega_L} \left(\psi^2 u_{y,1}\right) \left(u_{y,2}^2 u_{y,1}\right) \dd \boldsymbol{z}
    ,
  \end{aligned}
  \end{equation}
  by the definition of \(u_{y,1} = E^\#_{\boldsymbol{x}} u_{\mathcal{B}_{\#},\mathcal{B}_d,\psi^2,0}^{(1)}(\Omega_1)\) according to \cref{eq:prototypeEigenvalueProb} multiplied with \(u_{y,2}^2 u_{y,1}\). Together with \(0 < \mathsf{c} < \psi^2 < \mathsf{C}\), it then follows from \cref{eq:10} that
  \begin{equation}
  \label{eq:11}
  \begin{aligned}
    \mathsf{c} \int_{\Omega_L} \nabla u \cdot \nabla u \dd \boldsymbol{z}
    &\le
    \mathsf{C} \int_{\Omega_L} u_{y,1}^2 \nabla u_{y,2} \cdot \nabla u_{y,2} \dd \boldsymbol{z}
    +
    \mathsf{C}\lambda_{u_{y,1}}^{(1)} \int_{\Omega_L} u_{y,1}^2 u_{y,2}^2 \dd \boldsymbol{z}
    \\
    &\le
    \mathsf{C}\max{\{1,\lambda_{u_{y,1}}^{(1)}\}} {||u_{y,2}||}^2_{H^1(\Omega_L;u_{y,1}^2)}
    .
  \end{aligned}
  \end{equation}
  Adding \({||u||}^2_{L^2(\Omega_L)} = {||u_{y,2}||}^2_{L^2(\Omega_L;u_{y,1}^2)}\) on both sides of \cref{eq:11} yields
  \begin{align}
    {||\nabla u||}^2_{L^2(\Omega_L)} + {||u||}^2_{L^2(\Omega_L)}
    \le
    \frac{\mathsf{C}\max{\{1,\lambda_{u_{y,1}}^{(1)}\}}}{\mathsf{c}} {||u_{y,2}||}^2_{H^1(\Omega_L;u_{y,1}^2)}
    + {||u_{y,2}||}^2_{L^2(\Omega_L;u_{y,1}^2)}
    ,
  \end{align}
  which finally, with \(u_{y,2} = T(u)\), provides us that \({||u||}^2_{H^1(\Omega_L)} \le \mathsf{D} {||T(u)||}_{H^1(\Omega_L;u_{y,1}^2)}^2\) for some \(\mathsf{D}>0\). This is equivalent to
  \begin{align}\label{eq:TInvIsConti}
    {||T^{-1}(u_{y,2})||}^2_{H^1(\Omega_L)}
    \le
    \mathsf{D} {||u_{y,2}||}_{H^1(\Omega_L;u_{y,1}^2)}^2
    .
  \end{align}
  As \(T^{-1}\) is a linear, bijective, and continuous (by \cref{eq:TInvIsConti}) operator between two Banach spaces, the inverse \({(T^{-1})}^{-1} = T\) is also continuous~\cite[p35]{brezisFunctionalAnalysisSobolev2010} with \({||T(u)||}^2_{H^1(\Omega_L;u_{y,1}^2)} \le \tilde{\mathsf{D}} {||u||}_{H^1(\Omega_L)}^2\). Bi-continuity of \(T\) implies that \(T\) and \(T^{-1}\) are continuously Fr\'echet-differentiable since they are both linear. Thus, \(T\) is a diffeomorphism~\cite{galewskiCONDITIONSHAVINGDIFFEOMORPHISM2014}.
\end{proof}
In order to apply the above-defined factorization operator in the min-max setting, we have the following:
\begin{lemma}[Rayleigh Quotients after Factorization]\label{lem:rayleighFactorization}
  Let \(\phi^{(m)} = u_{\mathcal{B}_d,\mathcal{B}_d,1,V}^{(m)}(\Omega_L)\) and \(u^{(m)},\psi\) be given as in \cref{thm:factorization}. Then, after the factorization of \(\phi^{(m)} = {u^{(m)}} \cdot \psi\), the corresponding Rayleigh quotient reads
  \begin{equation}\label{eq:18}
    \mathcal{R}_{1,V}(\phi^{(m)})
    =
    \mathcal{R}_{\psi^2,0}\left({u^{(m)}}\right)
    +
    \lambda_\psi
    .
  \end{equation}
\end{lemma}
\begin{proof}
  We first note that the factorization allows for the splitting
  \begin{equation}\label{eq:factorizationSplitting}
  \begin{aligned}
    \nabla \phi^{(m)} \cdot \nabla \phi^{(m)}
    &=
    \nabla \left({u^{(m)}} \psi\right) \cdot \nabla \left({u^{(m)}} \psi\right)
    \\
    &=
    \psi^2 \nabla {u^{(m)}} \cdot \nabla {u^{(m)}} + \nabla \psi \cdot \nabla \left({\left(u^{(m)}\right)}^2 \psi\right)
    .
  \end{aligned}
  \end{equation}
  Using the splitting \cref{eq:factorizationSplitting}, we obtain that \(\mathcal{R}_{1,V}(\phi^{(m)})\) is equal to
  \begin{equation}\label{eq:rayleighCalculationFactorization}
  \begin{aligned}
    & \frac{\int_{\Omega_L} \psi^2 \nabla {u^{(m)}} \cdot \nabla {u^{(m)}} \dd \boldsymbol{z}}{\int_{\Omega_L} \psi^2 {\left(u^{(m)}\right)}^2 \dd \boldsymbol{z}}
    +
    \frac{
      \int_{\Omega_L} \left( \nabla \psi \cdot \nabla \left({\left(u^{(m)}\right)}^2 \psi\right) + V \psi \left({\left(u^{(m)}\right)}^2 \psi\right) \right) \dd \boldsymbol{z}
    }{
      \int_{\Omega_L} \psi \left({\left(u^{(m)}\right)}^2 \psi \right) \dd \boldsymbol{z}
    }
    \\
    = {} &
    \mathcal{R}_{\psi^2,0}\left({u^{(m)}}\right)
    +
    \frac{
      \int_{\Omega_L} \left( - \nabla \cdot \left(\nabla \psi\right) + V \psi  \right) {\left(u^{(m)}\right)}^2 \psi \dd \boldsymbol{z}
    }{
      \int_{\Omega_L} \psi \left({\left(u^{(m)}\right)}^2 \psi \right) \dd \boldsymbol{z}
    }
    \\
    = {} &
    \mathcal{R}_{\psi^2,0}\left({u^{(m)}}\right)
    +
    \frac{
      \int_{\Omega_L} \lambda_\varphi \psi \left({\left(u^{(m)}\right)}^2 \psi \right) \dd \boldsymbol{z}
    }{
      \int_{\Omega_L} \psi \left({\left(u^{(m)}\right)}^2 \psi \right) \dd \boldsymbol{z}
    }
    =
    \mathcal{R}_{\psi^2,0}\left({u^{(m)}}\right)
    +
    \lambda_\psi
    ,
  \end{aligned}
  \end{equation}
  by the definition of the eigenfunction \(\psi = E^\#_{\boldsymbol{x}} u_{\mathcal{B}_\#,\mathcal{B}_d,1,V}^{(1)}(\Omega_1)\).
\end{proof}
We can now combine all the above to prove \cref{thm:factorization}:
\begin{proof}[Proof of \cref{thm:factorization}]\label{proof:factorizationThm}
  For a simpler notation, let \(\phi = u_{\mathcal{B}_d,\mathcal{B}_d,1,V}^{(m)}\change(\Omega_L)\).
  We apply the factorization principle twice with
  \begin{equation}\label{eq:factorizations}
    \phi = u \psi = {(T_{\psi})}^{-1}(u)
    \quad \text{and} \quad
    u = u_{y,1} u_{y,2} = {(T_{u_{y,1}})}^{-1}(u_{y,2})
    .
  \end{equation}
  The operations are defined in \cref{eq:factorizationPrinciplePsi,lemma:facOperator}, and the latter ensures that these are valid changes of variables in the min-max characterization. With \cref{lem:rayleighFactorization} and the first change of variables, we then obtain for \(\lambda_{\mathcal{B}_d,\mathcal{B}_d,1,V}^{(m)}(\Omega_L)\) that
  \begin{equation}
  \begin{aligned}
    \min_{\substack{W_m \subset H_{\mathcal{B}_d,\mathcal{B}_d}^1(\Omega_L) \\ \text{dim} W_m = m}}
    \max_{\substack{\phi \in W_m \\ \phi \ne 0}} \mathcal{R}_{1,V}(\phi)
    &=
    \min_{\substack{W_m \subset H_{\mathcal{B}_d,\mathcal{B}_d}^1(\Omega_L) \\ \text{dim} W_m = m}}
    \max_{\substack{u \in W_m \\ u \ne 0}} \mathcal{R}_{\psi^2,0}(u)
    + \lambda_{\mathcal{B}_\#,\mathcal{B}_\#,1,V}^{(1)}(\Omega_1)
    \\
    &=
    \lambda_{\mathcal{B}_d,\mathcal{B}_d,\psi^2,0}^{(m)}(\Omega_L)
    + \lambda_{\mathcal{B}_\#,\mathcal{B}_\#,1,V}^{(1)}(\Omega_1)
    .
  \end{aligned}
  \end{equation}
  The potential \(V\) is now encoded in the diffusion coefficient \(\psi^2\). We now continue in the same fashion with the second factorization and obtain
  \begin{equation}
  \begin{aligned}
    \lambda_{\mathcal{B}_d,\mathcal{B}_d,\psi^2,0}^{(m)}
    &=
    \min_{\substack{W_m \subset H_{\mathcal{B}_d,\mathcal{B}_n}^1(\Omega_L;u_{y,1}^2) \\ \text{dim} W_m = m}}
    \max_{\substack{u_{y,2} \in W_m \\ u_{y,2} \ne 0}} \mathcal{R}_{\psi^2 u_{y,1}^2,0}(u_{y,2})
    + \lambda_{\mathcal{B}_\#,\mathcal{B}_d,\psi^2,0}^{(1)}(\Omega_1)
    \\
    &=
    \lambda_{\mathcal{B}_d,\mathcal{B}_n,\psi^2 u_{y,1}^2,0}^{(m)}(\Omega_L)
    + \lambda_{\mathcal{B}_\#,\mathcal{B}_d,\psi^2,0}^{(1)}(\Omega_1)
    .
  \end{aligned}
  \end{equation}
  Here, we used \({(T_{u_{y,1}})}^{-1}(u_{y,2})\) being a diffeomorphism by \cref{lemma:facOperator} and, therefore, a well-defined change of variables. The corresponding eigenfunction multiplication in \cref{eq:facVec3} is a direct result of the factorizations of \cref{eq:factorizations}.
  If we then also apply the \(\psi\)-factorization to \(\varphi_y\), we obtain \(\varphi_y = \psi u_{y,1}\) and \(\lambda_{\varphi_y} = \lambda_{\psi} + \lambda_{u_{y,1}}\), which concludes the proof of \cref{eq:fac3}. The other relations, i.e.~\cref{eq:fac1,eq:facVec1}, follow analogously with applying their respective factorizations \({(T(\cdot))}^{-1}\) in \cref{eq:factorizations}.
\end{proof}

\subsection{Homogenization in the Expanding Directions}\label{ssec:homogenization}
In order to entirely characterize the asymptotic behavior of the spectrum as $L\to\infty$, we will now consider the contribution \(\lambda_{\mathcal{B}_d,\mathcal{B}_n,\psi^2 u_{y,1}^2,0}^{(m)}(\Omega_L)\) in \cref{eq:fac3} as the only one that depends on \(m\) after the factorization of \cref{thm:factorization}. Then, we can make a precise statement about the asymptotic behavior of this remainder.
\begin{theorem}[Asymptotic behavior of expanding direction]\label{thm:asymptoticBehaviorOfXDirectionNew}
  Let \(\psi,u_{y,1}\) be given as in \cref{thm:factorization} and define \(\rho := {(\psi u_{y,1})}^2\). The asymptotic behavior of the eigenpair \(u_{\mathcal{B}_d,\mathcal{B}_n,\rho,0}^{(m)}(\Omega_L),\lambda_{\mathcal{B}_d,\mathcal{B}_n,\rho,0}^{(m)}(\Omega_L)\) for \(L \to \infty\) is
  \begin{align}\label{eq:eigenpairConvergenceToHomogenizedLimit}
    \lambda_{\mathcal{B}_d,\mathcal{B}_n,\rho,0}^{(m)}
    &=
    \frac{1}{L^2} \left( \nu^{(m) } + \mathcal{O}\left( \frac{1}{L} \right) \right)
    ,
    \\
    L^{p/2} u_{\mathcal{B}_d,\mathcal{B}_n,\rho,0}^{(m)}(\boldsymbol{x}/L,\boldsymbol{y})
    &\rightharpoonup
    u_0^{(m)}(\boldsymbol{x}) \text{ weakly up to a subseq.~in } H_{\mathcal{B}_d,\mathcal{B}_n}^1(\Omega_1)
    ,
  \end{align}
  where \((u_0^{(m)},\nu^{(m)}) \in (H_0^1({(0,1)}^p) \setminus \{0\}) \times \mathbb{R}\) is the solution to the \(p\)-dimensional homogenized eigenvalue problem
  \begin{align}\label{eq:homogenizedEigenpair}
    \left\{
    \begin{aligned}
      - \nabla \cdot \left( \bar{D} \nabla u_0^{(m)} \right) &= \nu^{(m)} \bar{C} u_0^{(m)} \text{ in } {(0,1)}^p
      \\
      u_0^{(m)} &= 0 \text{ on } \partial {(0,1)}^p
    \end{aligned}
    \right.
    ,
  \end{align}
  with the constant homogenized coefficients, \(\bar{D} \in \mathbb{R}^{p \times p}\), \(\bar{C} \in \mathbb{R}\), given by
  \begin{align}\label{eq:homogenizedCoeffsNewNew}
    \bar{D}_{ij} = \int_{{(0,1)}^p} \int_{{(0,\ell)}^q} \rho \left( \delta_{ij} + \frac{\partial \theta_j}{\partial x_i} \right) \dd \boldsymbol{y} \dd \boldsymbol{x}
    , \quad
    \bar{C} = \int_{{(0,1)}^p} \int_{{(0,\ell)}^q} \rho \dd \boldsymbol{y} \dd \boldsymbol{x}
    ,
  \end{align}
  for \(i,j = 1,\dots,p\). The corrector functions \({\left\{\theta_i (\boldsymbol{x},\boldsymbol{y})\right\}}_{1 \le i \le p}\) are defined as cell problem solutions on the periodic unit cell, as
  \begin{align}\label{eq:thetaCorrectorProblem}
    \left\{
      \begin{aligned}
        - \nabla \cdot \left( \rho(\tilde{\boldsymbol{x}},\boldsymbol{y}) \left( \boldsymbol{e}_i + \nabla \theta_i(\tilde{\boldsymbol{x}},\boldsymbol{y}) \right)\right) &= 0 \text{ in } \Omega_1 = {(0,1)}^p \times {(0,\ell)}^q
        \\
        \tilde{\boldsymbol{x}} &\mapsto \theta_i(\tilde{\boldsymbol{x}},\boldsymbol{y}) \; \tilde{\boldsymbol{x}} \text{-periodic}
      \end{aligned}
      \right.
    .
  \end{align}
  Furthermore, it holds that \(\nu^{(1)} < \nu^{(2)}\).
\end{theorem}
\begin{proof}\label{proof:asymptoticBehaviorOfXDirectionNew}
  The proof is divided into five steps. We first apply a spatial transformation to identify the directional homogenization problem. In the second step, we show the existence of a weakly converging subsequence for the linear source problem. Then, the oscillating test function method provides the homogenized operators whose dimensions can be further reduced by considering the directional framework in the fourth step. The last step transfers the results to the eigenvalue problem.
  \par
  \textbf{Step 1:} \textit{Identification of a directional homogenization problem by transformation.}
  \par
  To operate on fixed spatial domains, we map the problem from \(\Omega_L\) to the reference domain \(\Omega_1 = {(0,1)}^p \times {(0,\ell)}^q\) by \(\boldsymbol{x} \mapsto \boldsymbol{x}/L =:\varepsilon \boldsymbol{x}\) and observe for the transformed weight function \(\rho_\varepsilon(\boldsymbol{x},\boldsymbol{y}) := \rho(\boldsymbol{x}/\varepsilon,\boldsymbol{y})\) that \(1/\rho_\varepsilon \in L^1_{\textnormal{loc}}(\Omega_1)\) and \(\rho_\varepsilon > 0 \text{ a.e.\ in } \Omega_1\).
  Thus, the correct framework is the weighted space \(H^1(\Omega_1;\rho_\varepsilon)\).
  We now encode the Dirichlet boundary conditions on the \(\boldsymbol{x}\)-boundary (as in~\cite[p6]{papanicolauAsymptoticAnalysisPeriodic1978}) in the sense of traces with \(\Gamma_{\text{D}} := {\{0,1\}}^p \times {(0,\ell)}^q \subset \partial \Omega_1\) in the subspace
  \begin{equation}
    \mathbb{V}_{\rho_\varepsilon} = \{ \phi \in H^1(\Omega_1;\rho_\varepsilon) \mid \phi = 0 \text{ on } \Gamma_{\text{D}} \} \subset H^1(\Omega_1;\rho_\varepsilon)
    .
  \end{equation}
  Here, \(\mathbb{V}_{\rho_\varepsilon}\) is a Banach space since \(\rho_\varepsilon=0\) occurs only on the boundary \(\Omega_1\) (c.f.~\cref{rem:regularWeightImpliesClassicalSobolevSpaces}). The weak form of the eigenvalue problem reads: Find \((u^{(m)}_\varepsilon, \lambda^{(m)}_\varepsilon) \in (\mathbb{V}_{\rho_\varepsilon} \setminus \{0\}) \times \mathbb{R}\), such that
  \begin{align}\label{eq:epsDependentEigenvalueProblem}
    \forall v \in \mathbb{V}_{\rho_\varepsilon} : \quad
    a_\varepsilon(u^{(m)}_\varepsilon,v)
    = \lambda^{(m)}_\varepsilon \int_{\Omega_1} \rho_\varepsilon(\boldsymbol{x},\boldsymbol{y}) u^{(m)}_\varepsilon v \dd \boldsymbol{x} \dd \boldsymbol{y}
    ,
  \end{align}
  with the bilinear form
  \begin{equation}\label{eq:bilinearForm}
    a_\varepsilon(u^{(m)}_\varepsilon,v)
    =
    \int_{\Omega_1}
    \rho_\varepsilon(\boldsymbol{x},\boldsymbol{y})
    \left(
      \frac{\partial u^{(m)}_\varepsilon}{\partial x_i} \frac{\partial v}{\partial x_i}
      + \frac{1}{\varepsilon^2} \frac{\partial u^{(m)}_\varepsilon}{\partial y_i} \frac{\partial v}{\partial y_i}
    \right) \dd \boldsymbol{x} \dd \boldsymbol{y}
    ,
  \end{equation}
  using index notation. In \cref{eq:epsDependentEigenvalueProblem}, we moved the \(\varepsilon^2\)-scaling to \(\lambda^{(m)}_\varepsilon\) as
  \begin{equation}\label{eq:homogenizationGetEpsScalingIntoEval}
    \lambda_{\mathcal{B}_d,\mathcal{B}_n,\rho_\varepsilon,0}^{(m)} = \varepsilon^2 \lambda^{(m)}_\varepsilon
    ,
  \end{equation}
  which follows from the min-max principle. This operation will be justified later when the existence of this \(\varepsilon^2\)-transformed eigenvalue problem is shown for \(\varepsilon \to 0\).
  \par
  \textbf{Step 2:} \textit{Extraction of a weakly converging subsequence for the linear equation.}
  \par
  From~\cite{kesavanHomogenizationEllipticEigenvalue1979,kesavanHomogenizationEllipticEigenvalue1979a}, we know that the homogenization of eigenvalue problems uses the same homogenized operators as for the corresponding source problem. Hence, we consider the bilinear form of the corresponding source problem to derive the homogenized operators. Therefore, given a family \(f_\varepsilon(\boldsymbol{x},\boldsymbol{y}) \in \mathbb{V}_{\rho_\varepsilon}'\) with \(f_\varepsilon(\boldsymbol{x},\boldsymbol{y}) \rightarrow f_0(\boldsymbol{x}) \text{ in } \mathbb{V}_{\rho_\varepsilon}'\), we study the variational formulation
  \begin{equation}\label{eq:epsDepWeakFormSource}
    \forall v \in \mathbb{V}_{\rho_\varepsilon}:
    \quad
     a_\varepsilon(u_\varepsilon,v) = {\langle f_\varepsilon,v \rangle}_{\mathbb{V}_{\rho_\varepsilon}' \times \mathbb{V}_{\rho_\varepsilon}}
    .
  \end{equation}
  The restriction of the family \(f_\varepsilon\) to \(\boldsymbol{y}\)-constant functions in the limit will be justified later when we show that the homogenized limit \(u_0\) will have exactly this form. Thus, since we want to derive the eigenvalue problem from the source problem, \(f_\varepsilon\) has to mimic the properties of the sequence \(u_\varepsilon\). The bilinear form \(a_\varepsilon(u_\varepsilon,v)\) is \(\mathbb{V}_{\rho_\varepsilon}\)-elliptic (for \(\varepsilon \le 1\)), since for all \(u_\varepsilon \in \mathbb{V}_{\rho_\varepsilon}\), we have
  \begin{equation}
  \begin{aligned}
    a_\varepsilon(u_\varepsilon,u_\varepsilon)
    &\overset{\varepsilon \le 1}{\ge}
    {\| \nabla u_\varepsilon \|}^2_{L^2(\Omega_1;\rho_\varepsilon)} \overset{\text{F.-in.}}{\ge} \frac{1}{2} {\| \nabla u_\varepsilon \|}^2_{L^2(\Omega_1;\rho_\varepsilon)} + \frac{\mathsf{C}_{\text{F}}}{2} {\| u_\varepsilon \|}^2_{L^2(\Omega_1;\rho_\varepsilon)}
    \\
    &\ge
    \frac{1}{2}\min\{1,\mathsf{C}_{\text{F}}\} {\|u_\varepsilon\|}^2_{H^1(\Omega_1;\rho_\varepsilon)}
    =:
    \mathsf{C} {\|u_\varepsilon\|}^2_{H^1(\Omega_1;\rho_\varepsilon)}
    ,
  \end{aligned}
  \end{equation}
  after using the weighted Friedrichs inequality~\cite[p199]{leonardiBestConstantWeighted1994} (for homogeneous Dirichlet boundary condition on parts of the boundary).
  Continuity also holds for all \(\varepsilon > 0\) with a continuity constant proportional to \(1/\varepsilon^2\). Thus, the problem is well-posed and admits a unique solution for all \(\epsilon > 0\) in \(\mathbb{V}_{\rho_\varepsilon}\) (Lax--Milgram, c.f.~\cite[p126]{chipotAnisotropicSingularPerturbation2007}). We, however, are interested in precisely the limit \(\varepsilon \to 0\), which, at first, seems to be problematic since the continuity constant would tend to infinity if we do not further specify the \(\partial \boldsymbol{y}\)-behavior in \cref{eq:bilinearForm}.
  However, we take (as in~\cite[p24]{papanicolauAsymptoticAnalysisPeriodic1978}) \(u_\varepsilon\) in the bilinear form and use the coercivity to obtain
  \begin{equation}
    \mathsf{C} {\|u_\varepsilon\|}^2_{H^1(\Omega_1;\rho_\varepsilon)}
    \le
    a_\varepsilon(u_\varepsilon,u_\varepsilon)
    =
    {\langle f_\varepsilon, u_\varepsilon \rangle}_{\mathbb{V}_{\rho_\varepsilon}' \times \mathbb{V}_{\rho_\varepsilon}}
    \le
    {\|f_\varepsilon\|}_{H^{-1}(\Omega_1;\rho_\varepsilon)} {\|u_\varepsilon\|}_{H^1(\Omega_1;\rho_\varepsilon)}
    ,
  \end{equation}
  with the operator norm \({\|f_\varepsilon\|}_{H^{-1}(\Omega_1;\rho_\varepsilon)} = {\|f_\varepsilon\|}_{\mathbb{V}_{\rho_\varepsilon}'} \to {\|f_0\|}_{\mathbb{V}_{\rho_\varepsilon}'} \le \mathsf{D}\) by our assumption on the family \(f_\varepsilon \in \mathbb{V}_{\rho_\varepsilon}'\). Therefore, \(u_\varepsilon\) is uniformly bounded in \(H^1(\Omega_1;\rho_\varepsilon)\) since \({\|u_\varepsilon\|}_{H^1(\Omega_1;\rho_\varepsilon)} \le \frac{1}{\mathsf{C}} {\|f_\varepsilon\|}_{H^{-1}(\Omega_1;\rho_\varepsilon)} < \infty\).
  Now recall that \(\rho_\varepsilon = \rho(\boldsymbol{x}/\varepsilon,\boldsymbol{y})\) is \(\boldsymbol{x}\)-periodic and thus weakly converges to its \(\boldsymbol{x}\)-average \(\rho_0(\boldsymbol{y})\)~\cite[Lem.~1.8.]{allaireBriefIntroductionHomogenization2012}. From the boundedness of \(u_\varepsilon\) in \(H^1(\Omega_1;\rho_\varepsilon)\), we can follow with~\cite[Prop.~2.1.]{zhikovWeightedSobolevSpaces1998} that there exists a \(u_0 \in H^1(\Omega_1;\rho_0)\), such that there exists a converging subsequence of \(u_\varepsilon\), still denoted by \(u_\varepsilon\) by abuse of notation, that weakly convergences in \(H^1(\Omega_1;\rho_0)\). This ensures the existence of the desired homogenized limit \(u_0\) of \(u_\varepsilon\) as \(\varepsilon \to 0\).
  We can also directly infer \(\sqrt{\rho_\varepsilon} \frac{\partial u_\varepsilon}{\partial \boldsymbol{y}} \to \boldsymbol{0} \) in \(L^2(\Omega_1{})\) by taking the limit in
  \begin{equation}
    \begin{aligned}
    \frac{\mathsf{C}}{\varepsilon^2} \int_{\Omega_1} \rho_\varepsilon {\left|\frac{\partial u_\varepsilon}{\partial \boldsymbol{y}}\right|}^2 \dd \boldsymbol{x} \dd \boldsymbol{y}
    &\le
    a_\varepsilon(u_\varepsilon,u_\varepsilon)
    \\
    &\le
    {\|f_\varepsilon\|}_{H^{-1}(\Omega_1;\rho_\varepsilon)} {\|u_\varepsilon\|}_{H^1(\Omega_1;\rho_\varepsilon)}
    \le
    \frac{1}{\mathsf{C}} {\|f_\varepsilon\|}^2_{H^{-1}(\Omega_1;\rho_\varepsilon)}
    < \infty
    ,
  \end{aligned}
  \end{equation}
  since the norm of \(u_\varepsilon\) is bounded by the norm of \(f_\varepsilon\), which implies that
  \begin{equation}\label{eq:homogenizedLimitIsYConstant}
    \lim_{\varepsilon \to 0} {\left\|\sqrt{\rho_\varepsilon} \frac{\partial u_\varepsilon}{\partial \boldsymbol{y}} - \boldsymbol{0} \right\|}_{L^2(\Omega_1{})} = 0
    ,
  \end{equation}
  since \(\frac{1}{C} {\|f_\varepsilon\|}^2_{H^{-1}(\Omega_1;\rho_\varepsilon)}\) is bounded for all \(\varepsilon\), including \(\varepsilon = 0\). Therefore, \(\sqrt{\rho_\varepsilon} \frac{\partial u_\varepsilon}{\partial \boldsymbol{y}} \to \boldsymbol{0}\) in \(L^2(\Omega_1{})\), which will be important later to reduce the dimension of the homogenized equation from \(p+q\) dimensions to just \(p\). Since \(\rho_0\) is nonzero a.e.\ on \(\Omega_1\) (recall that \(\rho_\varepsilon = 0\) only happens on the \(\boldsymbol{y}\)-boundary), we have \(\frac{\partial u_\varepsilon}{\partial \boldsymbol{y}} \to \boldsymbol{0}\) in \(L^2(\Omega_1)\). Thus, a homogenized limit \(u_0\) with \(\partial u_0 / \partial \boldsymbol{y} = \boldsymbol{0}\) exists for the sequence \(u_\varepsilon\).
  \par
  \textbf{Step 3:} \textit{Derivation of the homogenized operators using oscillating test functions.}%
  \par
  Since we know that there exists a homogenized limit \(u_0\), we aim to derive the corresponding homogenized equation for \(u_0\). Therefore, consider
  \begin{equation}
    \xi_\varepsilon(\boldsymbol{x},\boldsymbol{y})
    :=
    \sqrt{\rho_\varepsilon(\boldsymbol{x},\boldsymbol{y})} \nabla u_\varepsilon(\boldsymbol{x},\boldsymbol{y})
    .
  \end{equation}
  Following the usual arguments~\cite[p24]{papanicolauAsymptoticAnalysisPeriodic1978}, from the uniform boundedness of \(u_\varepsilon\), it follows that
  \begin{equation}\label{eq:boundednessOfXi}
    {\|\xi_\varepsilon\|}_{L^2(\Omega_1)}
    =
    {|u_\varepsilon|}_{H^1(\Omega_1;\rho_\varepsilon)}
    \le
    {\|u_\varepsilon\|}_{H^1(\Omega_1;\rho_\varepsilon)}
    \le
    \frac{1}{C} {\|f_{\varepsilon}\|}_{H^{-1}(\Omega_1;\rho_\varepsilon)}
    <
    \infty
    .
  \end{equation}
  Therefore, we can again extract subsequences \( \xi_\varepsilon\), still denoted by \(\xi_\varepsilon\), such that \(\xi_\varepsilon \rightharpoonup \xi_0\) in \(L^2(\Omega_1)\) weakly.
  This convergence implies that the equation of interest
  \begin{equation}
    {\langle \xi_\varepsilon, \nabla v \rangle}_{L^2(\Omega_1)} = {\langle f_\varepsilon, v \rangle}_{\mathbb{V}_{\rho_\varepsilon}} \quad \forall v \in \mathbb{V}_{\rho_\varepsilon}
    ,
  \end{equation}
  has a limit for \(\varepsilon \to 0\) as
  \begin{equation}\label{eq:abstractHomogenizedEq}
    {\langle \xi_0, \nabla v \rangle}_{L^2(\Omega_1)} = {\langle f_0, v \rangle}_{\mathbb{V}_{\rho_0}} \quad \forall v \in \mathbb{V}_{\rho_0}
    .
  \end{equation}
  To explicitly state this limit equation, we need to calculate \(\xi_0\). We employ the oscillatory test function method~\cite[p10]{allaireBriefIntroductionHomogenization2012} to overcome the problem of \(\xi_\varepsilon = \sqrt{\rho(\boldsymbol{x}/\varepsilon,\boldsymbol{y})} \nabla u_\varepsilon\) being a product of two weakly converging functions and thus not simply being the product of both limits for \(\varepsilon \to 0\). We need to adapt the method to account for the directional periodicity and the additional \(\varepsilon^{-2}\)-scaling of the \((\partial u_\varepsilon / \partial y_i)\)-term in the bilinear form \cref{eq:bilinearForm}. Thus, let \(\varphi \in \mathcal{D}({(0,1)}^p)\) be a smooth, only \(\boldsymbol{x}\)-dependent, and compactly supported test function (i.e., \(\varphi \in C_{\text{c}}^\infty({(0,1)}^p)\)). Then, inspired by the first two terms in the asymptotic expansion for \(u_\varepsilon\), we define the test function \(\varphi_\varepsilon\) as
  \begin{equation}\label{eq:specialTestFunction}
    \varphi_\varepsilon (\boldsymbol{x},\boldsymbol{y})
    :=
    \varphi(\boldsymbol{x}) + \varepsilon \sum_{i=1}^p \frac{\partial \varphi(\boldsymbol{x})}{\partial x_i} \theta_i(\boldsymbol{x}/\varepsilon, \boldsymbol{y})
  \end{equation}
  where, with \(\tilde{\boldsymbol{x}}:=\boldsymbol{x}/\varepsilon\), \(\theta_i(\tilde{\boldsymbol{x}},\boldsymbol{y})\) is the solution to the corrector problem \cref{eq:thetaCorrectorProblem}, which admits for all \(i=1,\dots,p\) a unique solution \(\theta_i \in H^1(\Omega_1;\rho_1) \,/\, \mathbb{R}\) (due to periodic boundary conditions).
  Since \(\theta_i(\tilde{\boldsymbol{x}},\boldsymbol{y})\) is \(\tilde{\boldsymbol{x}}\)-periodic, it converges weakly to its average in \(H^1(\Omega_1;\rho_0)\) as \(\varepsilon \to 0\). Thus, the expression \(\varepsilon \theta_i(\tilde{\boldsymbol{x}},\boldsymbol{y})\) in \cref{eq:specialTestFunction} converges to zero since \(\varepsilon \to 0\). This implies that \(\varphi_\epsilon\) has a well-defined limit \(\varphi_\varepsilon \rightharpoonup \varphi_0 = \varphi(\boldsymbol{x})\) for \(\varepsilon \to 0\).
  \par
  In the following, we group the gradients into \((p,q)\)-blocks by using the notation \(\nabla (\cdot) := {(\nabla_{\boldsymbol{x}} (\cdot),\nabla_{\boldsymbol{y}} (\cdot))}^T\). As the derivative of \(\varphi_\varepsilon\) is required in the variational formulation, we derive from \cref{eq:specialTestFunction}, using the chain and product rule, that
  \begin{equation}
    \nabla \varphi_\varepsilon
    =
    \sum_{i=1}^{p}
    \frac{\partial \varphi(\boldsymbol{x})}{\partial x_i}
    \left(
      e_i
      +
      \begin{pmatrix} \nabla_{\tilde{\boldsymbol{x}}}\theta_i(\tilde{\boldsymbol{x}},\boldsymbol{y}) \\ \varepsilon \nabla_{\boldsymbol{y}}\theta_i(\tilde{\boldsymbol{x}},\boldsymbol{y}) \end{pmatrix}
    \right)
    + \varepsilon \sum_{i=1}^{p}
    \left(
      \frac{\partial}{\partial x_i}
      \begin{pmatrix} \nabla_{\boldsymbol{x}} \varphi (\boldsymbol{x}) \\ 0 \end{pmatrix}
      \theta_i(\tilde{\boldsymbol{x}},\boldsymbol{y})
    \right)
    .
  \end{equation}
  We then insert the test function \(\varphi_\varepsilon\) into the bilinear form \cref{eq:bilinearForm} to obtain
  \begin{equation}
  \label{eq:insertionOfSpecialtestfunction}
  \begin{aligned}
    a_\varepsilon(u_\varepsilon,\varphi_\varepsilon)
    &=
    \int_{\Omega_1} \rho_\varepsilon(\boldsymbol{x},\boldsymbol{y}) \nabla u_\varepsilon
    \cdot
    \left(
      \sum_{i=1}^{p}
      \frac{\partial \varphi(\boldsymbol{x})}{\partial x_i}
      \left(
        e_i
        +
        \begin{pmatrix} \nabla_{\tilde{\boldsymbol{x}}}\theta_i(\tilde{\boldsymbol{x}},\boldsymbol{y}) \\ \varepsilon^{-1} \nabla_{\boldsymbol{y}}\theta_i(\tilde{\boldsymbol{x}},\boldsymbol{y}) \end{pmatrix}
      \right)
    \right)
    \dd \boldsymbol{x} \dd \boldsymbol{y}
    \\
    &+
    \varepsilon \int_{\Omega_1} \rho_\varepsilon(\boldsymbol{x},\boldsymbol{y}) \nabla u_\varepsilon
    \cdot
    \left(
      \sum_{i=1}^{p}
      \left(
        \frac{\partial}{\partial x_i}
        \begin{pmatrix} \nabla_{\boldsymbol{x}} \varphi(\boldsymbol{x}) \\ 0 \end{pmatrix}
        \theta_i(\tilde{\boldsymbol{x}},\boldsymbol{y})
      \right)
    \right)
    \dd \boldsymbol{x} \dd \boldsymbol{y}
    .
  \end{aligned}
  \end{equation}
  The last term in \cref{eq:insertionOfSpecialtestfunction} vanishes in the limit since it can be bounded by a constant times \(\varepsilon\) by the Cauchy--Schwarz inequality as the \((\varphi,\theta_i)\)-term is uniformly bounded in \(L^2(\Omega_1;\rho_\varepsilon)\) by uniformly boundedness of the data in \cref{eq:thetaCorrectorProblem} and \(\varphi\)-smoothness. The other term, \(\nabla u_\varepsilon\), is uniformly bounded in \(L^2(\Omega_1;\rho_\varepsilon)\) by \cref{eq:boundednessOfXi}.
  Integration by parts (with Dirichlet in \(\boldsymbol{x}\)- and trivially fulfilled Neumann data in \(\boldsymbol{y}\)-direction) in the other term of \cref{eq:insertionOfSpecialtestfunction} yields
  \begin{equation}\label{eq:moveDivergence}
  \begin{aligned}
    &
    \int_{\Omega_1} \rho_\varepsilon(\boldsymbol{x},\boldsymbol{y}) \nabla u_\varepsilon
    \cdot
    \left(
      \sum_{i=1}^{p} \frac{\partial \varphi(\boldsymbol{x})}{\partial x_i}
      \left(
        e_i
        +
        \begin{pmatrix} \nabla_{\tilde{\boldsymbol{x}}}\theta_i(\tilde{\boldsymbol{x}},\boldsymbol{y}) \\ \varepsilon^{-1} \nabla_{\boldsymbol{y}}\theta_i(\tilde{\boldsymbol{x}},\boldsymbol{y}) \end{pmatrix}
      \right)
    \right)
    \dd \boldsymbol{x} \dd \boldsymbol{y}
    \\
    &=
    - \int_{\Omega_1} u_\varepsilon \nabla \cdot
    \left(
      \rho_\varepsilon(\boldsymbol{x},\boldsymbol{y})
      \sum_{i=1}^{p} \frac{\partial \varphi(\boldsymbol{x})}{\partial x_i}
      \left(
        e_i
        +
        \begin{pmatrix} \nabla_{\tilde{\boldsymbol{x}}}\theta_i(\tilde{\boldsymbol{x}},\boldsymbol{y}) \\ \varepsilon^{-1} \nabla_{\boldsymbol{y}}\theta_i(\tilde{\boldsymbol{x}},\boldsymbol{y}) \end{pmatrix}
      \right)
    \right)
    \dd \boldsymbol{x} \dd \boldsymbol{y}
    .
  \end{aligned}
  \end{equation}
  The divergence term in \cref{eq:moveDivergence} can be further simplified to
  \begin{equation}\label{eq:evaluateDivergenceTerm}
  \begin{aligned}
    &
    \nabla \cdot
    \left(
      \rho_\varepsilon(\boldsymbol{x},\boldsymbol{y})
      \sum_{i=1}^{p} \frac{\partial \varphi(\boldsymbol{x})}{\partial x_i}
      \left(
        e_i
        +
        \begin{pmatrix} \nabla_{\tilde{\boldsymbol{x}}}\theta_i(\tilde{\boldsymbol{x}},\boldsymbol{y}) \\ \varepsilon^{-1} \nabla_{\boldsymbol{y}}\theta_i(\tilde{\boldsymbol{x}},\boldsymbol{y}) \end{pmatrix}
      \right)
    \right)
    \\
    &=
    \sum_{i=1}^{p}
    \frac{\partial}{\partial x_i} \begin{pmatrix} \nabla_{\boldsymbol{x}} \varphi (\boldsymbol{x}) \\ 0 \end{pmatrix}
    \cdot {\rho_\varepsilon(\boldsymbol{x},\boldsymbol{y})}
    \left(
      e_i
      +
      \begin{pmatrix} \nabla_{\tilde{\boldsymbol{x}}}\theta_i(\tilde{\boldsymbol{x}},\boldsymbol{y}) \\ 0 \end{pmatrix}
    \right)
    \\
    &+
    \varepsilon^{-1} \sum_{i=1}^{p} \frac{\partial \varphi(\boldsymbol{x})}{\partial x_i}
    \left[
      \begin{pmatrix}
        \nabla_{\tilde{\boldsymbol{x}}}
        \\
        \nabla_{\boldsymbol{y}}
      \end{pmatrix}
      \cdot
      \left(
        \rho(\tilde{\boldsymbol{x}},\boldsymbol{y})
        \left(
          e_i +
          \begin{pmatrix}
            \nabla_{\tilde{\boldsymbol{x}}}
            \\
            \nabla_{\boldsymbol{y}}
          \end{pmatrix}
          \theta_i(\tilde{\boldsymbol{x}},\boldsymbol{y})
        \right)
      \right)
    \right]
    ,
  \end{aligned}
  \end{equation}
  From \cref{eq:evaluateDivergenceTerm}, we extract
  \begin{equation}\label{eq:divergenceProductAndChainRule}
    \varepsilon^{-1}
    \begin{pmatrix}
      \nabla_{\tilde{\boldsymbol{x}}}
      \\
      \nabla_{\boldsymbol{y}}
    \end{pmatrix}
    \cdot
    \left(
      \rho(\tilde{\boldsymbol{x}},\boldsymbol{y})
      \left(
        e_i +
        \begin{pmatrix}
          \nabla_{\tilde{\boldsymbol{x}}}
          \\
          \nabla_{\boldsymbol{y}}
        \end{pmatrix}
        \theta_i(\tilde{\boldsymbol{x}},\boldsymbol{y})
      \right)
    \right)
    ,
  \end{equation}
  which is zero (even for \(\varepsilon \to 0\)) due to the particular definition of the correctors \(\theta_i\) in \cref{eq:thetaCorrectorProblem}. Here, we notice that the \(\varepsilon^{-2}\)-scaling in the \(\boldsymbol{y}\)-term of the initial expression precisely aligns with the extruding additional \(\varepsilon^{-1}\) that appeared in \cref{eq:divergenceProductAndChainRule} by the chain rule. Thus, the only remaining term of \cref{eq:evaluateDivergenceTerm} is
  \begin{equation}\label{eq:remainingTermFromDivergence}
    \sum_{i=1}^{p}
    \frac{\partial}{\partial x_i} \begin{pmatrix} \nabla_{\boldsymbol{x}} \varphi (\boldsymbol{x}) \\ 0 \end{pmatrix}
    \cdot {\rho_\varepsilon(\boldsymbol{x},\boldsymbol{y})}
    \left(
      e_i
      +
      \begin{pmatrix} \nabla_{\tilde{\boldsymbol{x}}}\theta_i(\tilde{\boldsymbol{x}},\boldsymbol{y}) \\ 0 \end{pmatrix}
    \right)
    ,
  \end{equation}
  which is bounded in \(L^2(\Omega_1)\) and, thus, weakly converges as \(\varepsilon \to 0\) to its average in the \(\tilde{\boldsymbol{x}}\)-direction~\cite[Lem.~1.8.]{allaireBriefIntroductionHomogenization2012}.
  \par
  In \cref{eq:moveDivergence}, now recall that \(u_\varepsilon\) converges strongly to \(u_0\) in \(L^2(\Omega_1;\rho_0)\) (by the Rellich theorem, c.f.~\cite[Thm.~4.3.21]{allaireNumericalAnalysisOptimization2007}). Thus, we can take the limit of the right-hand side in \cref{eq:moveDivergence}. So in total, we can take the limit of \cref{eq:insertionOfSpecialtestfunction}, which is the product of the limit of \(u_\varepsilon \to u_0\) (strongly) with the weak limit \cref{eq:remainingTermFromDivergence} of the divergence term. Thus the weak form \cref{eq:epsDepWeakFormSource} reduces for \(\varepsilon \to 0\) to
  \begin{equation}\label{eq:takingLimitsInWeakForm}
  \begin{aligned}
    -& \int_{\Omega_1}
    u_0(\boldsymbol{x}) \nabla \cdot
    \left(
      \int_{{(0,1)}^p}
      \rho(\tilde{\boldsymbol{x}},\boldsymbol{y})
      \sum_{i=1}^{p} \frac{\partial \varphi(\boldsymbol{x})}{\partial x_i}
      \left(
        e_i
        +
        \begin{pmatrix} \nabla_{\tilde{\boldsymbol{x}}}\theta_i(\tilde{\boldsymbol{x}},\boldsymbol{y}) \\ 0 \end{pmatrix}
      \right)
      \dd \tilde{\boldsymbol{x}}
    \right)
    \dd \boldsymbol{x} \dd \boldsymbol{y}
    \\
    =&
    \lim_{\varepsilon \to 0} \int_{\Omega_1} \rho_\varepsilon(\boldsymbol{x},\boldsymbol{y}) \nabla u_\varepsilon(\boldsymbol{x},\boldsymbol{y}) \cdot \nabla \varphi_\varepsilon(\boldsymbol{x},\boldsymbol{y}) \dd \boldsymbol{x} \dd \boldsymbol{y}
    \\
    =&
    \lim_{\varepsilon \to 0} \int_{\Omega_1} \rho_\varepsilon(\boldsymbol{x},\boldsymbol{y}) f_\varepsilon(\boldsymbol{x},\boldsymbol{y}) \varphi_\varepsilon(\boldsymbol{x},\boldsymbol{y}) \dd \boldsymbol{x} \dd \boldsymbol{y}
    \\
    =&
    \int_{\Omega_1} \rho_0(\boldsymbol{y}) f_0(\boldsymbol{x}) \varphi(\boldsymbol{x}) \dd \boldsymbol{x} \dd \boldsymbol{y}
    .
  \end{aligned}
  \end{equation}
  We can rewrite the left-hand side of \cref{eq:takingLimitsInWeakForm} using a compact notation as
  \begin{equation}
  \begin{aligned}
    & {\left[
      \int_{{(0,1)}^p} \rho(\tilde{\boldsymbol{x}},\boldsymbol{y}) \sum_{i=1}^p \frac{\partial \varphi(\boldsymbol{x})}{\partial x_i}
      \left(
        e_i
        +
        \begin{pmatrix} \nabla_{\tilde{\boldsymbol{x}}}\theta_i(\tilde{\boldsymbol{x}},\boldsymbol{y}) \\ 0 \end{pmatrix}
      \right)
      \dd \tilde{\boldsymbol{x}}
    \right]}_j
    \\
    =&
    \sum_{i=1}^p \frac{\partial \varphi(\boldsymbol{x})}{\partial x_i} \int_{{(0,1)}^p} \rho(\tilde{\boldsymbol{x}},\boldsymbol{y}) \left( \delta_{ij} + \frac{\partial \theta_i}{\partial {\tilde{x}}_j} \right) \dd \tilde{\boldsymbol{x}}
    =:
    \sum_{i=1}^p \frac{\partial \varphi(\boldsymbol{x})}{\partial x_i} \tilde{D}_{ji}(\boldsymbol{y})
    ,
  \end{aligned}
  \end{equation}
  where we identify the last expression as \({[\tilde{D}^T(\boldsymbol{y}) \nabla \varphi(\boldsymbol{x})]}_j\). As the last step, we reverse the integration by parts and obtain the variational formulation of the homogenized equation for \(u_0\), which is still posed on the \((p+q)\)-dimensional domain \(\Omega_1\), but with \(u_0\) only \(\boldsymbol{x}\)-dependent according to \cref{eq:homogenizedLimitIsYConstant}. The problem then reads: Find \(u_0(\boldsymbol{x}) \in \mathbb{V}_{\rho_0}\), such that
  \begin{equation}\label{eq:homogenizedProblemSourcePrelim}
    \int_{\Omega_1} \tilde{D}(\boldsymbol{y}) \nabla u_0(\boldsymbol{x}) \cdot \nabla \varphi(\boldsymbol{x}) \dd \boldsymbol{x} \dd \boldsymbol{y}
    =
    \int_{\Omega_1} \tilde{C}(\boldsymbol{y}) f_0(\boldsymbol{x}) \varphi(\boldsymbol{x}) \dd \boldsymbol{x} \dd \boldsymbol{y}
    \quad \forall \varphi \in C_{\text{c}}^\infty({(0,1)}^p)
    ,
  \end{equation}
  with the \(\boldsymbol{y}\)-dependent operators
  \begin{align}\label{eq:preliminaryHomogenizedOps}
    \tilde{D}_{ij}(\boldsymbol{y}) = \int_{{(0,1)}^p}
    \rho(\tilde{\boldsymbol{x}},\boldsymbol{y}) \left( \delta_{ij} + \frac{\partial \theta_j}{\partial {\tilde{x}}_i} \right) \dd \tilde{\boldsymbol{x}}
    , \quad
    \tilde{C}(\boldsymbol{y}) = \int_{{(0,1)}^p}
    \rho(\tilde{\boldsymbol{x}},\boldsymbol{y}) \dd \tilde{\boldsymbol{x}}
    ,
  \end{align}
  for \(i,j = 1,\dots,p\).
  In \cref{eq:preliminaryHomogenizedOps}, we remark that these operators look very similar to the usual homogenized operators, e.g., in~\cite{allaireHomogenizationSpectralProblem2000}, with the difference that the integration only takes place in the \(p\) expanding directions over \({(0,1)}^p\).
  \par
  \textbf{Step 4:} \textit{Dimension reduction of the homogenized linear equation.}
  \par
  In our setup, we can, however, further reduce the homogenized limit equation \cref{eq:homogenizedProblemSourcePrelim} since, by definition, \(\nabla \varphi (\boldsymbol{x}) = {(\nabla_{\boldsymbol{x}} \varphi(\boldsymbol{x}),0)}^T\). This allows us to concretize further that \(u_0(\boldsymbol{x}) \in H_0^1({(0,1)}^p)\) since \(u_0(\boldsymbol{x}) \in \mathbb{V}_{\rho_0(\boldsymbol{y})}\) implies \(u_0(\boldsymbol{x})=0\) on \(\partial {(0,1)}^p\) and \({\|u_0(\boldsymbol{x})\|}_{H^1({(0,1)}^p)} < \infty\) since for any \(u(\boldsymbol{x}) \in H^1(\Omega_1;\rho_0)\) with \(\rho_0(\boldsymbol{y}) > 0 \text{ a.e.~in } {(0,\ell)}^q\), it holds that
  \begin{equation}
  \begin{aligned}
    {\|u(\boldsymbol{x})\|}^2_{H^1(\Omega_{1};\rho_0)}
    =
    \left( \int_{{(0,1)}^q} \rho_0(\boldsymbol{y}) \dd \boldsymbol{y} \right)
    {\|u(\boldsymbol{x})\|}^2_{H^1({(0,1)}^p)}
    =
    \overline{\rho_0}^y {\|u(\boldsymbol{x})\|}^2_{H^1({(0,1)}^p)}
    <
    \infty
    ,
  \end{aligned}
  \end{equation}
  since \(\overline{\rho_0}^y \in \mathbb{R}\) is a strictly positive constant. Thus, the homogenized equation reduces from the \((p+q)\)- to the \(p\)-dimensional variational problem: Find \(u_0 \in H_0^1({(0,1)}^p)\), such that
  \begin{equation}\label{eq:homogenizedProblemSource}
    \int_{{(0,1)}^p} \bar{D} \nabla u_0(\boldsymbol{x}) \cdot \nabla \varphi(\boldsymbol{x}) \dd \boldsymbol{x}
    =
    \int_{{(0,1)}^p} \bar{C} f_0(\boldsymbol{x}) \varphi(\boldsymbol{x}) \dd \boldsymbol{x}
    \quad \forall \varphi \in C_{\text{c}}^\infty({(0,1)}^p)
    ,
  \end{equation}
  with the constant homogenized coefficients, defined by \cref{eq:homogenizedCoeffsNewNew} as the integral of \(\tilde{C}(\boldsymbol{y})\) and \(\tilde{D}(\boldsymbol{y})\) from \cref{eq:preliminaryHomogenizedOps} over \({(0,1)}^q\).
  \par
  The homogenized equation \cref{eq:homogenizedProblemSource} is formulated on \(H_0^1({(0,1)}^p)\). Recall that the test function \(\varphi \in C_{\text{c}}^\infty({(0,1)}^p)\) was chosen arbitrarily. Since \(C_{\text{c}}^\infty({(0,1)}^p)\) is dense in \(H_{0}^1({(0,1)}^p)\) by the definition of \(H_0^1\) as the closure of \(C_{\text{c}}^\infty\) under the \(H^1\)-norm~\cite[Def.~4.3.8.]{allaireNumericalAnalysisOptimization2007}, \cref{eq:homogenizedProblemSourcePrelim} holds \(\forall \varphi \in H_0^1({(0,1)}^p)\).
  As the homogenized operator satisfies coercivity (c.f.~\cite[Rem.~2.6.]{papanicolauAsymptoticAnalysisPeriodic1978}), the theorem of Lax--Milgram ensures the uniqueness of the homogenized limit \(u_0\). This, on the other hand, implies that any subsequence of \(u_\varepsilon\) converges to \(u_0\) in the limit. Thus, the entire sequence \(u_\varepsilon\) converges to the same limit \(u_0\) following the standard arguments from, e.g.,~\cite{allaireBriefIntroductionHomogenization2012}.
  \par
  \textbf{Step 5:} \textit{Derivation of the homogenized eigenvalue equation.}
  \par
  Since we now have derived the homogenized equation for the source problem, we can directly deduce from~\cite[Thm.~2.1.]{kesavanHomogenizationEllipticEigenvalue1979} that the eigenvalues and -functions converge to the homogenized eigenvalue equation, posed with the same operator as in \cref{eq:homogenizedProblemSource}, resulting in
  \begin{equation}
    (\lambda_\varepsilon^{(m)},u_\varepsilon^{(m)})
    \to
    (\nu^{(m)}, u_0^{(m)})
    \text{ in } \mathbb{R} \times \left( H^1_0({(0,1)}^p) \text{ weakly up to subseq.} \right)
    ,
  \end{equation}
  where the homogenized eigenpair is defined through \cref{eq:homogenizedEigenpair}. We furthermore have \(\lambda_\varepsilon^{(m)} = \nu^{(m)} + \mathcal{O}(\varepsilon)\)~\cite[p201]{kesavanHomogenizationEllipticEigenvalue1979a}~\cite[p1638]{santosaFirstOrderCorrectionsHomogenized1993a}~\cite[p942]{allaireAnalyseAsymptotiqueSpectrale1997}.
  The convergence of the eigenfunctions holds up to a subsequence because of the eigenvalue multiplicity of the homogenized limit.
  To account for the normalization constraint after the initial transformation of \(\boldsymbol{x} \mapsto \varepsilon \boldsymbol{x}\), we note that \(\|u_0(\boldsymbol{\cdot},\boldsymbol{\cdot})\|_{L^2(\Omega_L)} = L^{p/2} \|u_0(\boldsymbol{\cdot}/L,\boldsymbol{\cdot})\|_{L^2(\Omega_1)}\) by the transformation rule and recall the \((1/L^2)\)-scaling from \cref{eq:homogenizationGetEpsScalingIntoEval} for the eigenvalues, which implies \cref{eq:eigenpairConvergenceToHomogenizedLimit}.
  \par
  The limit eigenvalue \(\nu^{(m)}\) is simple and \(\nu^{(1)} < \nu^{(2)} < \nu^{(3)} < \dots \to \infty\) by the Sturm--Liouville theory for the particular case of \(p=1\) with \(\bar{D}_{11},\bar{C} > 0\). Furthermore, we have \(\nu^{(1)} < \nu^{(2)} \le \nu^{(3)} \le \dots \to \infty\) for the general case of \(p \ge 2\) since the eigenvalue problem is elliptic. However, multiplicities could exceed one for higher eigenvalues.
\end{proof}
We are now ready to prove the quasi-optimality of the spectral shift \(\sigma = \lambda_{\varphi_y}\):
\begin{theorem}\label{thm:fundamentalRatioBoundedNewNew}
  For the quasi-optimal shift \(\sigma = \lambda_{\varphi_y} = \lambda_{\mathcal{B}_\#,\mathcal{B}_d,1,V}^{(1)}(\Omega_1)\), the asymptotic shifted fundamental eigenvalue ratio of the linear periodic Schrödinger eigenvalue problem \cref{eq:schroedingerEquation} converges to a positive constant \(\mathsf{C} < 1\) as \(L \to \infty\), that is
  \begin{equation}\label{eq:fundamentalRatioConvergence}
    0
    \le
    \frac{
      \lambda_{\mathcal{B}_d,\mathcal{B}_d,1,V}^{(1)}(\Omega_L)
      -
      \lambda_{\mathcal{B}_\#,\mathcal{B}_d,1,V}^{(1)}(\Omega_1)
    }{
      \lambda_{\mathcal{B}_d,\mathcal{B}_d,1,V}^{(2)}(\Omega_L)
      -
      \lambda_{\mathcal{B}_\#,\mathcal{B}_d,1,V}^{(1)}(\Omega_1)
    }
    =
    \frac{
      \lambda_{\mathcal{B}_d,\mathcal{B}_n,\varphi_y^2,0}^{(1)}(\Omega_L)
    }{
      \lambda_{\mathcal{B}_d,\mathcal{B}_n,\varphi_y^2,0}^{(2)}(\Omega_L)
    }
    \to
    \mathsf{C}
    < 1
    .
  \end{equation}
\end{theorem}
\begin{proof}
  The proof follows from \cref{thm:asymptoticBehaviorOfXDirectionNew} since \(L^2 \lambda_{\mathcal{B}_d,\mathcal{B}_n,\varphi_y^2,0}^{(m)}(\Omega_L) = \nu^{(m)} + o\left(\frac{1}{L}\right)\) and \(\nu^{(1)} < \nu^{(2)}\).
\end{proof}
We will see later that pre-asymptotic effects lead to a non-monotonic convergence of \cref{eq:fundamentalRatioConvergence}. However, since the convergence holds in the limit, we can make a statement for uniform boundedness if \(L\) is sufficiently large.
\begin{corollary}\label{cor:shiftedFundamentalRationUniformlyBounded}
  There exists a constant \(\mathsf{D} \in [0,1)\) and a length \(L^* \in \mathbb{R}^+\), such that the quasi-optimally shifted ratio from \cref{thm:fundamentalRatioBoundedNewNew} is uniformly bounded from above by \(\mathsf{D}\) for all \(L > L^*\). That is 
  \begin{equation}
    0
    \le
    \frac{
      \lambda_{\mathcal{B}_d,\mathcal{B}_d,1,V}^{(1)}(\Omega_L)
      -
      \lambda_{\mathcal{B}_\#,\mathcal{B}_d,1,V}^{(1)}(\Omega_1)
    }{
      \lambda_{\mathcal{B}_d,\mathcal{B}_d,1,V}^{(2)}(\Omega_L)
      -
      \lambda_{\mathcal{B}_\#,\mathcal{B}_d,1,V}^{(1)}(\Omega_1)
    }
    < \mathsf{D}
    < 1
    \quad \forall L > L^*
    .
  \end{equation}
\end{corollary}
\begin{proof}
  The proof follows directly from the convergence result of \cref{thm:fundamentalRatioBoundedNewNew}.
\end{proof}
\begin{remark}\label{rem:absoluteOrderingPreserved}
  The quasi-optimal shift \(\sigma = \lambda_{\mathcal{B}_d,\mathcal{B}_d,1,V}^{(1)}(\Omega_L)\) does not affect the absolute eigenvalue ordering in the sense that \(0 < |\lambda_L^{(1)} - \sigma| < |\lambda_L^{(2)} - \sigma| \le |\lambda_L^{(3)} - \sigma| \le \dots \to \infty\) since all \(\lambda_L^{(i)}\) are positive and \(\sigma < \lambda_L^{(1)}\). This property ensures, for example, that the unshifted and the \(\sigma\)-shifted inverse power method (see \cref{def:shiftedIPM}) always converge to the same eigenpair.
\end{remark}
\cref{thm:asymptoticBehaviorOfXDirectionNew} gives an abstract description of the homogenized equation. However, for our present setup, we can even solve the equation analytically (which will be important later in \cref{ssec:homogenizationNumerics}):
\begin{remark}\label{rem:productOrSolenoidalImpliesDiagonalHomOp}
  The homogenization problem in \cref{thm:asymptoticBehaviorOfXDirectionNew} is posed with an isotropic operator \(\rho I\), and \(\rho\) is either periodic or zero on the unit cell boundaries \(\partial \Omega_1\). Thus, every column of \(\rho I\) is a solenoidal vector field in \(\Omega_1\) in the integral sense by the divergence theorem. Hence, we can conclude with~\cite[p17]{jikovHomogenizationDifferentialOperators1994} that the homogenized operator is diagonal.
  Then, the diagonality allows us to explicitly state the homogenized eigenpairs as the Laplacian eigenfunctions on the hyper rectangle with scaled Laplacian eigenvalues as \(\nu^{(m)} = \pi^2 \left( \sum_{i=1}^p \bar{D}_{ii} m_i^2 \right) / \bar{C}\) and \(u^{(m)}(\boldsymbol{x}) = \mathcal{N}^{(m)} \prod_{i}^{p} \sin(m_i \pi x_i)\), where the set \(\mathcal{M} = \{m_i,\dots,m_p\} \in {\mathbb{N}}^p, |\mathcal{M}|=m\), is chosen to minimize \(\nu^{(m)}\). The \(\mathcal{N}^{(m)}\) factors are defined by the normalization condition \(\int_{{(0,1)}^p} \bar{C} {\left(u^{(m)}\right)}^2 = 1\).
\end{remark}
We now return to the convergence properties of the eigenvalue solvers. \cref{thm:fundamentalRatioBoundedNewNew} implies a constant number of iterations for all eigensolvers that are shift-and-invert preconditioned with \(\sigma = \lambda_{\mathcal{B}_\#,\mathcal{B}_d,1,V}^{(1)}(\Omega_1)\) and depend on the fundamental ratio. With this strategy, the eigensolver can reach a given residual norm with a constant number of iterations for all \(L \to \infty\).

\section{Spatial Discretization and Iterative Eigensolvers}\label{sec:discretizationAndIterativeEigensolvers}
To solve the eigenvalue problem \cref{eq:schroedingerEquation} numerically, we will discretize the continuous equation on a finite-dimensional space. Then, we solve the resulting system with a preconditioned algebraic eigensolver.

\subsection{Galerkin Finite Element Approach}
Consider a conforming and shape-regular partition \(\mathcal{T}_h\) of the domain \(\Omega_L\) into finite elements \(\tau \in \mathcal{T}_h\), which have a polygonal shape. We write \(\mathcal{T}_h\) for partitions where every element has a diameter of at most \(2h\)~\cite[p36]{sunFiniteElementMethods2016}. Define the finite element subspace \(\mathbb{H}_h(\Omega_L) \subset H_{0}^1(\Omega_L)\), consisting of polynomial functions with total degree \(r\) from the polynomial space \(\mathcal{P}_r\), to be \(\mathbb{H}_h(\Omega_L) = \left\{ u \in H_{0}^1(\Omega_L) \; \middle|\ u|_{\tau} \in \mathcal{P}_r(\tau) \; \forall \tau \in \mathcal{T}_h \right\}\).
We then search for a discrete solution \(\phi_h^{(m)} \in \left( \mathbb{H}_h(\Omega_L) \setminus \{0\} \right)\), such that
\begin{equation}\label{eq:weakFormDiscrete}
    \forall v_h \in \mathbb{H}_h(\Omega_L)
    : \quad
    \int_{\Omega_L} \nabla \phi_h^{(m)} \cdot \nabla v_h \dd \boldsymbol{z}
    +
    \int_{\Omega_L} V \phi_h^{(m)} v_h \dd \boldsymbol{z}
    =
    \lambda_h^{(m)} \int_{\Omega_L} \phi_h^{(m)} v_h \dd \boldsymbol{z}
    .
\end{equation}
Let now \(\boldsymbol{x}_h^{(m)}\) be the coefficient vector that represents \(\phi_h^{(m)}\) in a given basis of \(\mathbb{H}_h(\Omega_L)\). We then obtain the equivalent generalized algebraic eigenvalue problem: Find \(\boldsymbol{x}_h^{(m)} \in \mathbb{R}^{n} \setminus \{\boldsymbol{0}\}\), such that
\begin{equation}\label{eq:algebraicEigenvalueProblemGeneralized}
  \boldsymbol{A} \boldsymbol{x}_h^{(m)} = \lambda_h^{(m)} \boldsymbol{B} \boldsymbol{x}_h^{(m)}
  ,
\end{equation}
where \(\boldsymbol{A} \in \mathbb{R}^{n \times n}\) consists of the usual stiffness matrix plus the contribution from the potential, and \(\boldsymbol{B} \in \mathbb{R}^{n \times n}\) denotes the mass matrix. Both \(\boldsymbol{A}\) and \(\boldsymbol{B}\) as finite representations of the continuous operators in \cref{eq:schroedingerEquation} are symmetric positive definite. Since the discrete problem is formulated on a subspace \(\mathbb{H}_h(\Omega_L) \subset H^1_0(\Omega_L)\), we have by the min-max characterization that \(\lambda^{(m)} \le \lambda_h^{(m)}\). Furthermore, we have \(\lambda_h^{(m)} \to \lambda^{(m)}\) for \(h \to 0\)~\cite{babuskaFiniteElementGalerkinApproximation1989,sunFiniteElementMethods2016}.
\par
For the calculation of the quasi-optimal shift \(\lambda_{\varphi_y}\), we solve
\begin{equation}\label{eq:weakFormDiscreteCellProblem}
  \forall v_h \in \mathbb{H}_h^{\varphi_y}(\Omega_1)
  : \quad
  \int_{\Omega_1} \nabla \varphi_{y,h}^{(1)} \cdot \nabla v_h \dd \boldsymbol{z}
  +
  \int_{\Omega_1} V \varphi_{y,h}^{(1)} v_h \dd \boldsymbol{z}
  =
  \lambda_{\varphi_y,h}^{(1)} \int_{\Omega_L} \varphi_{y,h}^{(1)} v_h \dd \boldsymbol{z}
  ,
\end{equation}
where \(\mathbb{H}_h^{\varphi_y}(\Omega_1) := \left\{ u \in H^1_{\mathcal{B}_\#,\mathcal{B}_d}(\Omega_1) \; \middle|\ u|_{\tau} \in \mathcal{P}_r(\tau) \; \forall \tau \in (\mathcal{T}_h \cap \Omega_1) \right\}\) with the same \(\mathcal{T}_h\) and \(r\) as for \(\mathbb{H}_h(\Omega_L)\) (assuming \(\Omega_1\)-aligned elements).

\subsection{Quasi-Optimally Preconditioned Eigenvalue Algorithms}\label{ssec:asymptoticPreconditioner}
To solve the resulting discrete eigenvalue problem \cref{eq:algebraicEigenvalueProblemGeneralized}, we use the analytic results from \cref{sec:factorization} to obtain the quasi-optimal shift as
\begin{equation}
  \sigma = \lambda_\infty = \lim_{L \to \infty} \lambda_L^{(1)} = \lambda_{\varphi_y} \approx \lambda_{\varphi_y,h}^{(1)}
  ,
\end{equation}
by combining the results of \cref{thm:factorization,thm:asymptoticBehaviorOfXDirectionNew}. Using \(\sigma\), we construct the preconditioner \(\boldsymbol{P}={(\boldsymbol{A}-\sigma\boldsymbol{B})}^{-1}\). With the generalized Rayleigh quotient \(R_{\boldsymbol{A},\boldsymbol{B}}(\boldsymbol{x}) = (\boldsymbol{x}^T \boldsymbol{A} \boldsymbol{x}) / (\boldsymbol{x}^T \boldsymbol{B} \boldsymbol{x})\), we define:
\begin{definition}[Shifted Inverse Power Method, \(\text{IP}_{\sigma}\)]\label{def:shiftedIPM}
  Let \(\boldsymbol{A},\boldsymbol{B} \in \mathbb{R}^{n \times n}\) and a start vector \(\boldsymbol{x}_0 \in \mathbb{R}^n\) be given, repeat
  \begin{equation}
    \tilde{\boldsymbol{x}}_k = \boldsymbol{P} \boldsymbol{B} \boldsymbol{x}_{k-1} \label{eq:shiftedIPM}
    , \quad
    \boldsymbol{x}_k = \tilde{\boldsymbol{x}}_k / \sqrt{\tilde{\boldsymbol{x}}_k^T \boldsymbol{B} \tilde{\boldsymbol{x}}_k}
    , \quad
    \lambda_k = R_{\boldsymbol{A},\boldsymbol{B}}(\boldsymbol{x}_k)
    ,
  \end{equation}
  until \({\|\boldsymbol{A} \boldsymbol{x}_k - \lambda_k \boldsymbol{B} \boldsymbol{x}_k\|}_2 < \texttt{TOL}\) or \(k > k_{\text{max}}\).
\end{definition}
\begin{definition}[Locally Optimal Preconditioned Conjugate Gradient Method, \(\text{LOPCG}_{\sigma}\)]\label{def:LOPCG}
  Let \(\boldsymbol{A},\boldsymbol{B} \in \mathbb{R}^{n \times n}\) and the start vectors \(\boldsymbol{x}_{-1}, \boldsymbol{x}_0 \in \mathbb{R}^n\) be given, repeat
  \begin{align}\label{eq:LOPCG}
    \begin{gathered}
    \boldsymbol{w}_k = \boldsymbol{P} (\boldsymbol{A} \boldsymbol{x}_{k-1} - R_{\boldsymbol{A},\boldsymbol{B}}(\boldsymbol{x}_{k-1}) \boldsymbol{B} \boldsymbol{x}_{k-1})
    , \quad
    S_k = \text{span}(\{\boldsymbol{x}_{k-1},\boldsymbol{w}_{k},\boldsymbol{x}_{k-2}\})
    \\
    \tilde{\boldsymbol{x}}_k = \arg \min_{\boldsymbol{y} \in S_k} R_{\boldsymbol{A},\boldsymbol{B}}(\boldsymbol{y})
    , \quad
    \boldsymbol{x}_k = \tilde{\boldsymbol{x}}_k / \sqrt{\tilde{\boldsymbol{x}}_k^T \boldsymbol{B} \tilde{\boldsymbol{x}}_k} , \quad \lambda_k = R_{\boldsymbol{A},\boldsymbol{B}}(\boldsymbol{x}_k)
    ,
    \end{gathered}
  \end{align}
  until \({\|\boldsymbol{A} \boldsymbol{x}_k - \lambda_k \boldsymbol{B} \boldsymbol{x}_k\|}_2 < \texttt{TOL}\) or \(k > k_{\text{max}}\).
\end{definition}
In \cref{eq:LOPCG}, the locally optimal step is calculated by minimizing in a 3-dimensional subspace with the standard Rayleigh--Ritz method~\cite{baiTemplatesSolutionAlgebraic2000} as \(\tilde{\boldsymbol{x}}_k = \alpha_1 \boldsymbol{x}_{k-1} + \alpha_2 \boldsymbol{w}_{k} + \alpha_3 \boldsymbol{x}_{k-2}\), where the coefficients \(\boldsymbol{\alpha} \in \mathbb{R}^3\) are derived from the smallest eigenpair solution of the 3-dimensional eigenvalue problem \(\boldsymbol{V}^T \boldsymbol{A} \boldsymbol{V} \boldsymbol{\alpha} = \lambda^{(1)} \boldsymbol{V}^T \boldsymbol{B} \boldsymbol{V} \boldsymbol{\alpha}\) with \(\boldsymbol{V} = \begin{bmatrix} \boldsymbol{x}_{k-1} \; \boldsymbol{w}_{k} \; \boldsymbol{x}_{k-2} \end{bmatrix} \in \mathbb{R}^{n \times 3}\).
\par
The above two methods represent the class of gap-dependent iterative eigenvalue algorithms. Optimization-inspired Riemannian gradient algorithms also depend on the fundamental ratio~\cite[Thm.~3.2]{henningSobolevGradientFlow2020}. Alternative approaches, such as the Rayleigh quotient iteration or block algorithms, are not considered in our setup since the former has no guaranteed convergence to the ground state~\cite[p53]{baiTemplatesSolutionAlgebraic2000}. At the same time, the latter requires an \(L\)-proportional block size to retain a quasi-optimal convergence~\cite[p54]{baiTemplatesSolutionAlgebraic2000}.

\section{Numerical Experiments}\label{sec:numericalExamples}
This section concerns the numerical evaluation of the proposed eigensolver preconditioner. We implemented our method using the \texttt{Gridap}~\cite{badiaGridapExtensibleFinite2020} framework in the Julia programming language~\cite{bezansonJuliaFreshApproach2017}. \texttt{Gridap} turned out to be a very well-suited framework for our tests since it allowed us to quickly implement weak formulations in a high-level fashion, similar to the \texttt{FEniCS}~\cite{alnaesFEniCSProjectVersion2015,theisenFenicsR13TensorialMixed2021} framework in Python. For reproducibility, we provide all examples publicly in~\cite{theisenDdEigenLabJlDomainDecomposition2022}.

\subsection{Homogenization of a Degenerate Eigenvalue Problem with Two Expanding Directions in Three Dimensions}\label{ssec:homogenizationNumerics}
Before we employ the constructed preconditioner for the linear Schrödinger eigenvalue problem \cref{eq:schroedingerEquation}, we first investigate the homogenization results of \cref{thm:asymptoticBehaviorOfXDirectionNew} since these results can be applied and studied independently. Thus, the theoretical predictions about the convergence of the \(m\)-dependent contribution \(u_{\mathcal{B}_d,\mathcal{B}_n,\rho,0}^{(m)}(\Omega_L),\lambda_{\mathcal{B}_d,\mathcal{B}_n,\rho,0}^{(m)}(\Omega_L)\) in three dimensions (\(p=2, q=1\)) are studied numerically. We prescribe the weight function by
\begin{equation}\label{eq:constructedWeight}
  \rho(\boldsymbol{x},\boldsymbol{y})
  =
  {\left(
    \tfrac{27}{4} y_1^2 (1-y_1) \left( 10 \cos{(\pi x_1)}^2 + 10 \cos{(\pi x_2)}^2 + \tfrac{11}{10} - \sin{(\pi y_1)}^2 \right)
  \right)}^2
  .
\end{equation}
Note that we do not set \(\rho={(\psi u_{y,1})}\) as in \cref{thm:asymptoticBehaviorOfXDirectionNew} since we want to demonstrate the results for the more general case of \(\rho\) not being induced by eigenfunctions but only satisfying the periodicity- and zero-condition on the \(\boldsymbol{x}\)- and \(\boldsymbol{y}\)-boundary respectively. The weight function \(\rho\) in \cref{eq:constructedWeight} is positive a.e.\ and vanishes only on the \(\boldsymbol{y}\)-boundary. By construction, \(\rho\) is \(\boldsymbol{x}\)-periodic, thus fulfilling all requirements of \cref{thm:asymptoticBehaviorOfXDirectionNew}. We intentionally use the \(\boldsymbol{x}\)-symmetry also to confirm the convergence of degenerate eigenpairs. For a better evaluation, we do not solve for \(\Omega_L\) but solve an equivalent problem on the reference domain \(\Omega_1\), where we factorized the \((1/L^2)\)-scaling (see \cref{thm:asymptoticBehaviorOfXDirectionNew}) of the eigenvalue without affecting the eigenfunctions. To be precise, we check if the solution to
\begin{align}\label{eq:numericalValidationHomogenizationTheoremEquation}
  \left\{
  \begin{aligned}
    - \nabla \cdot \left( \rho(L \boldsymbol{x} ,\boldsymbol{y})
    \diag{\left(1,1,\tfrac{1}{L^2}\right)}
    \nabla u_{1/L}^{(m)} \right) &= \lambda_{1/L}^{(m)} \rho(L \boldsymbol{x} ,\boldsymbol{y}) u_{1/L}^{(m)} \text{ in } {(0,1)}^3
    \\
    u_{1/L}^{(m)} &= 0 \text{ on } \partial {(0,1)}^2 \times {(0,1)}
  \end{aligned}
  \right.
  ,
\end{align}
converge to \((u_0^{(m)},\nu^{(m)})\) from \cref{eq:homogenizedEigenpair} in the limit for \(L \to \infty\). The calculation of this homogenized limit first needs the corrector functions to define the homogenized operators. Thus, we solve the corrector equation \cref{eq:thetaCorrectorProblem} using \(\mathbb{Q}_2\) finite elements on a structured mesh with \(300\) intervals per direction. These corrector solutions allow the construction of the homogenized coefficients (according to \cref{eq:homogenizedEigenpair,eq:homogenizedCoeffsNewNew}) with \(\bar{D} \approx \text{diag}(38.75893,38.75893)\) and \(\bar{C} \approx 57.86864\). We observe that \(\bar{D}_{11}=\bar{D}_{22}\) as the result of choosing an \((x_1,x_2)\)-symmetric weight function \(\rho\). The homogenized diffusion matrix is diagonal since we have \(\int_{\Omega_1} \nabla \cdot (\rho I ) = 0\) by the divergence theorem as \(\rho\) defined by \cref{eq:constructedWeight} is either periodic or zero on the boundary of the unit cube, which resembles the case of \cref{rem:productOrSolenoidalImpliesDiagonalHomOp}. Therefore, we can solve the homogenized equation analytically (with the expressions from \cref{rem:productOrSolenoidalImpliesDiagonalHomOp}) to obtain
\begin{equation}\label{eq:homogenizedLimitAnalytSols}
\begin{gathered}
  \nu^{(1)} = \frac{\pi^2 \left( 1^2 \bar{D}_{11} + 1^2 \bar{D}_{22} \right)}{\bar{C}} = \frac{2 \pi^2 \bar{D}_{11}}{\bar{C}}
  , \;
  \nu^{(2)} =
  \nu^{(3)} = \frac{5 \pi^2 \bar{D}_{11}}{\bar{C}}
  , \;
  \nu^{(4)} = \frac{8 \pi^2 \bar{D}_{11}}{\bar{C}}
  \\
  u_0^{(1)} = \mathcal{N} \sin(\pi x_1) \sin(\pi x_2)
  , \;
  u_0^{(2)} = \mathcal{N} \sin(2 \pi x_1) \sin(\pi x_2)
  ,
  \\
  u_0^{(3)} = \mathcal{N} \sin( \pi x_1) \sin(2 \pi x_2)
  , \;
  u_0^{(4)} = \mathcal{N} \sin(2 \pi x_1) \sin(2 \pi x_2)
\end{gathered}
\end{equation}
with the normalization constant \(\mathcal{N}=2/\sqrt{\bar{C}} \approx 0.26291\) since
\begin{equation}
  \left(\int_0^1 \sin^2{(m_1 \pi x_1)}\dd x_1\right) \cdots \left(\int_0^1 \sin^2{(m_p \pi x_p)}\dd x_p\right) = 2^{-p/2} \; \forall \boldsymbol{m} \in \mathbb{N}^p
  .
\end{equation}
\par
We then solve the eigenvalue problem \cref{eq:numericalValidationHomogenizationTheoremEquation} using the Galerkin method with \(\mathbb{Q}_2\) elements and a structured partition of both expanding directions into \(12L\) intervals. According to \cref{thm:asymptoticBehaviorOfXDirectionNew}, the non-relevant third direction is only discretized with six partitions since it is not relevant in the limit. Finally, we solve the corresponding algebraic eigenvalue problem using a block LOPCG method up to a tolerance of \(10^{-6}\).
\begin{figure}[t]
  \centering
  \includegraphics[width=0.4\linewidth]{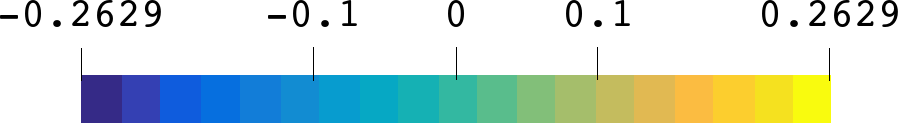}
  \\[-1ex]
  \raisebox{0.07\linewidth}{\raisebox{-3pt}{\(u_{1/L,h}^{(1)}\):}}
  \subfloat[
    \(L=2^0\)
  ]{\includegraphics[width=0.135\linewidth]{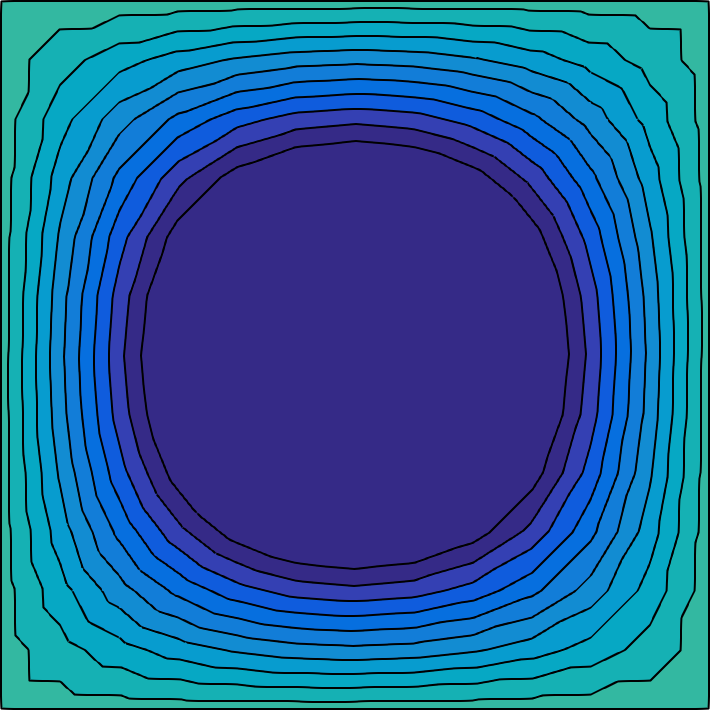}}
  \hfill
  \subfloat[
    \(L=2^1\)
  ]{\includegraphics[width=0.135\linewidth]{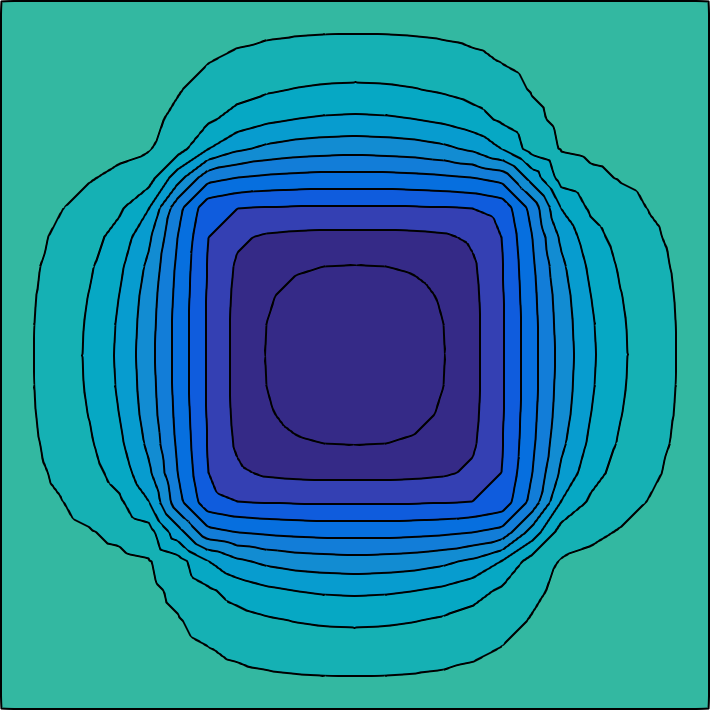}}
  \hfill
  \subfloat[
    \(L=2^2\)
  ]{\includegraphics[width=0.135\linewidth]{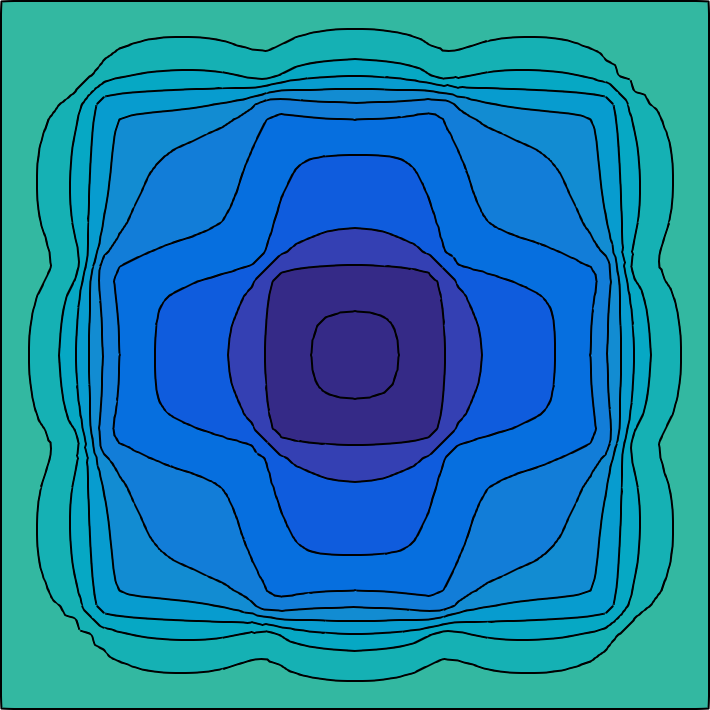}}
  \hfill
  \subfloat[
    \(L=2^3\)
  ]{\includegraphics[width=0.135\linewidth]{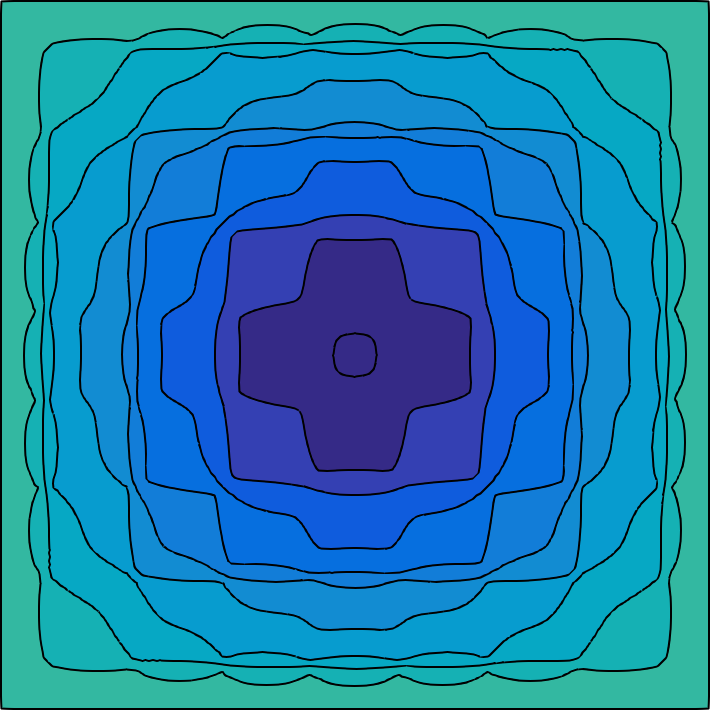}}
  \hfill
  \subfloat[
    \(L=2^4\)
  ]{\includegraphics[width=0.135\linewidth]{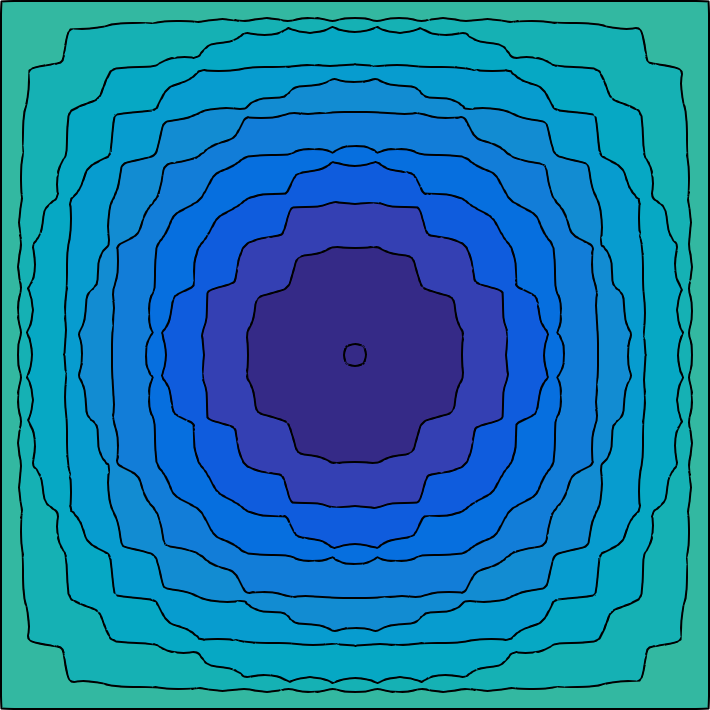}}
  \hspace{2pt}
  \raisebox{0.07\linewidth}{\raisebox{-3pt}{\(\overset{L \to \infty}{\rightarrow}\)}}
  \subfloat[
    \(u_0^{(1)}\)
  ]{\includegraphics[width=0.135\linewidth]{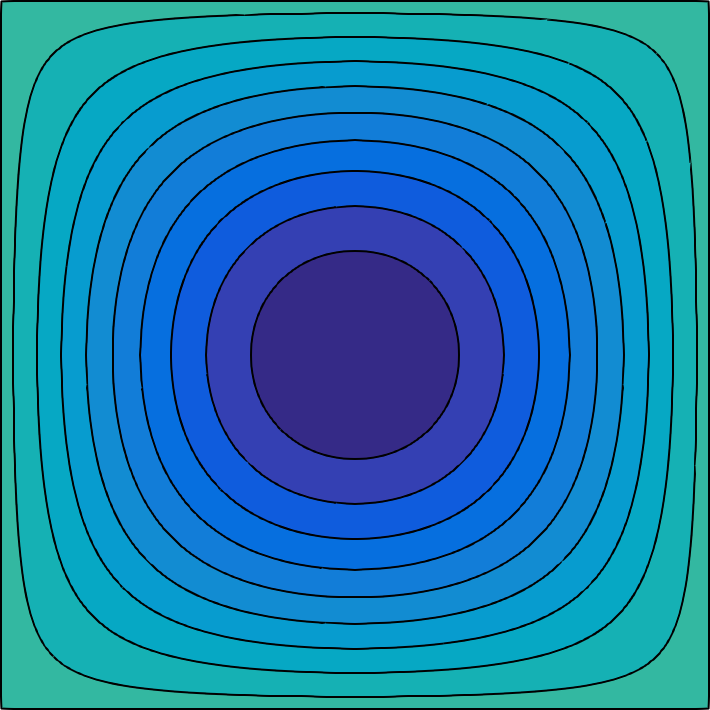}}
  \\[-2ex]
  \raisebox{0.07\linewidth}{\raisebox{-3pt}{\(u_{1/L,h}^{(2)}\):}}
  \subfloat[
    \(L=2^0\)
  ]{\includegraphics[width=0.135\linewidth]{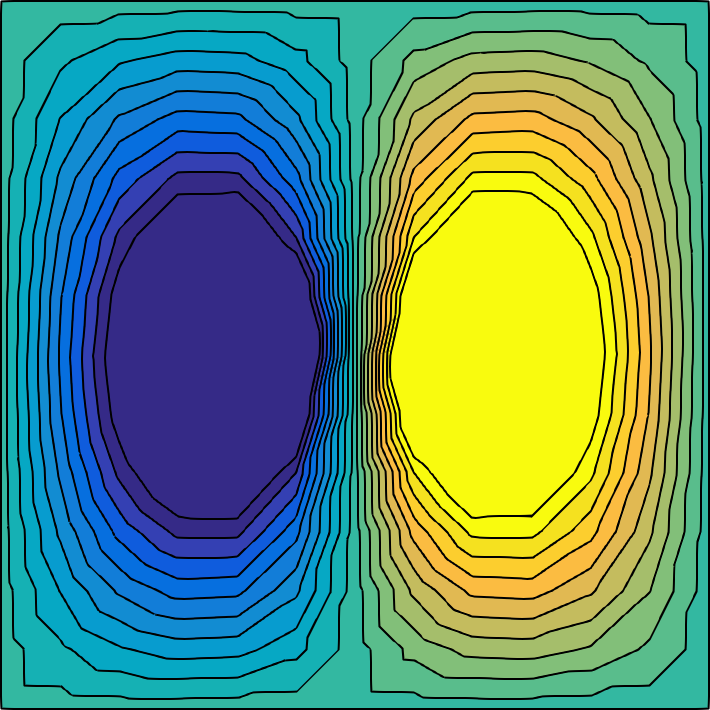}}
  \hfill
  \subfloat[
    \(L=2^1\)
  ]{\includegraphics[width=0.135\linewidth]{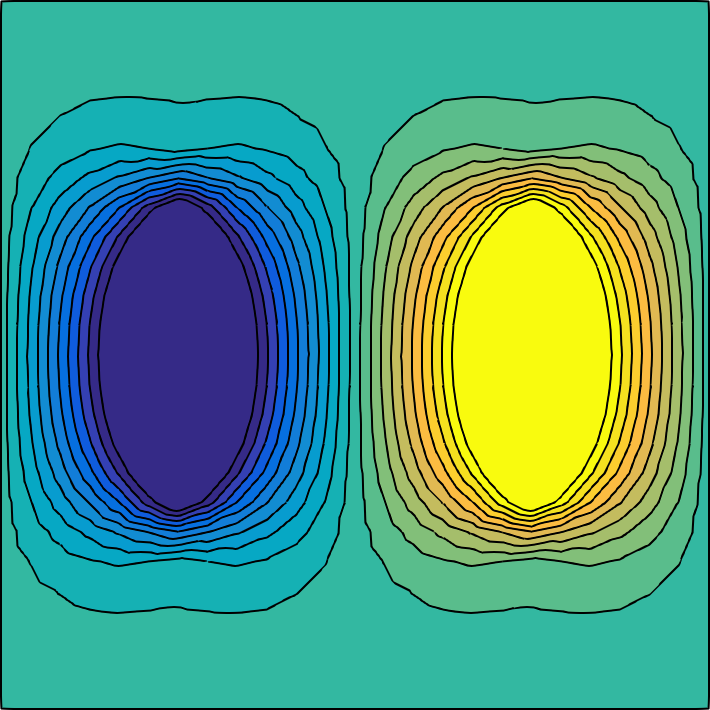}}
  \hfill
  \subfloat[
    \(L=2^2\)
  ]{\includegraphics[width=0.135\linewidth]{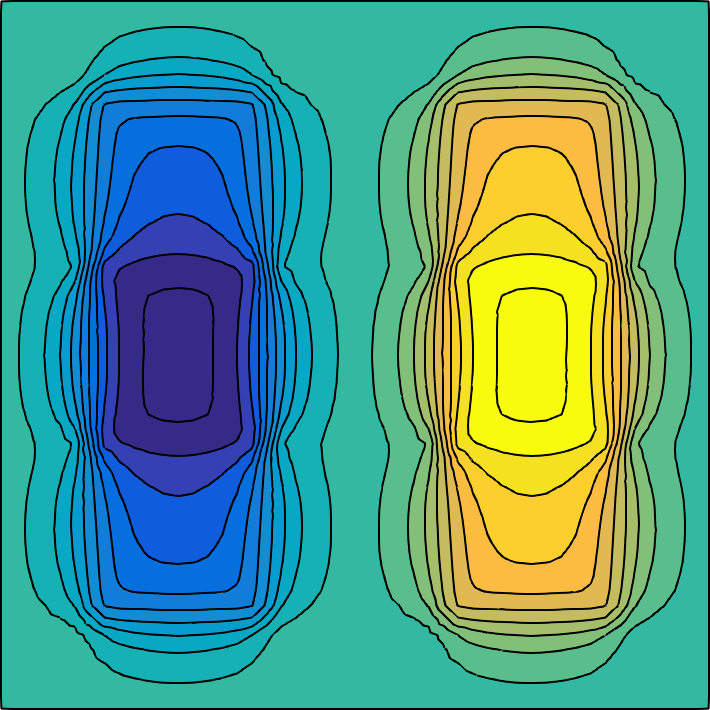}}
  \hfill
  \subfloat[
    \(L=2^3\)
  ]{\includegraphics[width=0.135\linewidth]{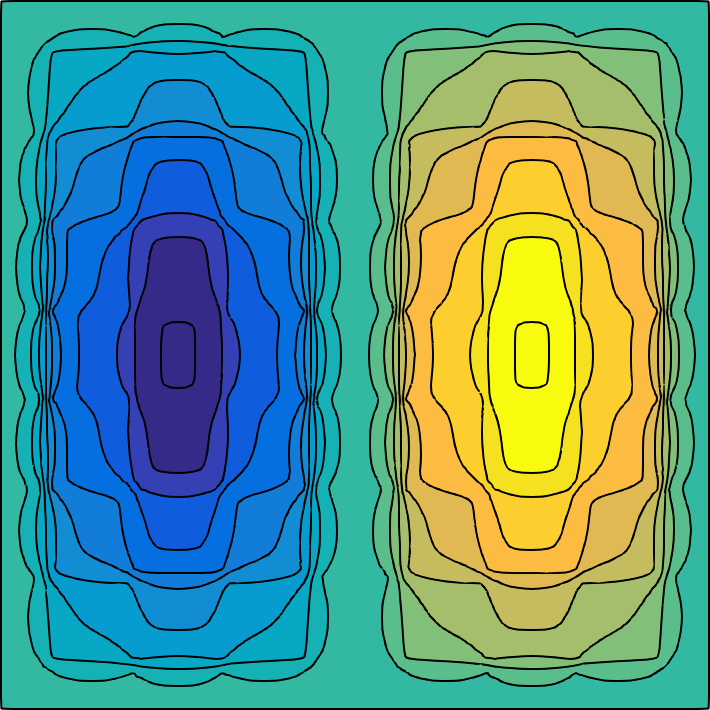}}
  \hfill
  \subfloat[
    \(L=2^4\)
  ]{\includegraphics[width=0.135\linewidth]{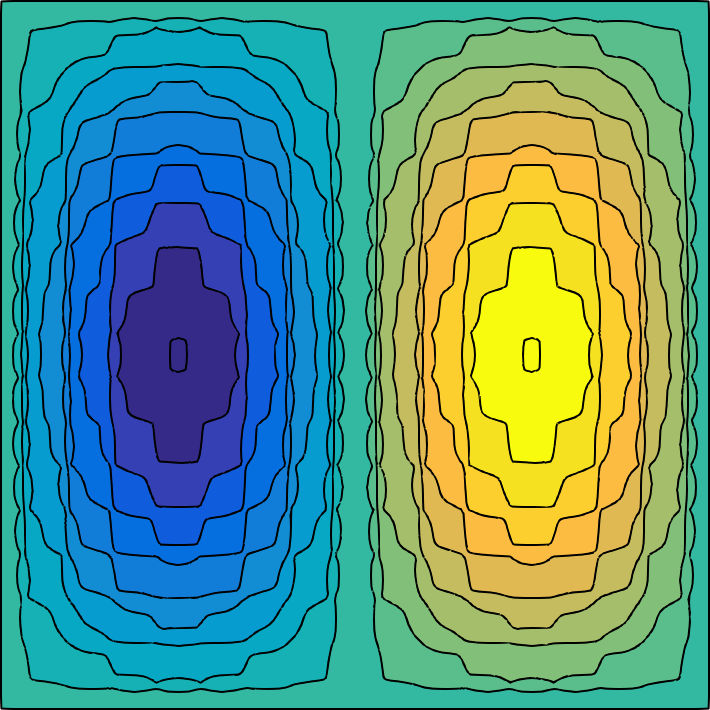}}
  \hspace{2pt}
  \raisebox{0.07\linewidth}{\raisebox{-3pt}{\(\overset{L \to \infty}{\rightarrow}\)}}
  \subfloat[
    \(u_0^{(2)}\)
  ]{\includegraphics[width=0.135\linewidth]{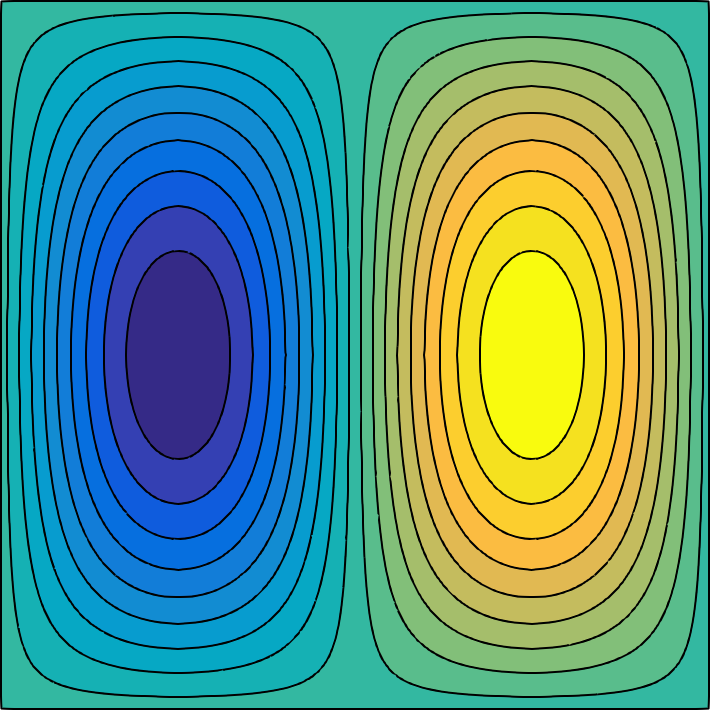}}
  \\[-2ex]
  \raisebox{0.07\linewidth}{\raisebox{-3pt}{\(u_{1/L,h}^{(3)}\):}}
  \subfloat[
    \(L=2^0\)
  ]{\includegraphics[width=0.135\linewidth]{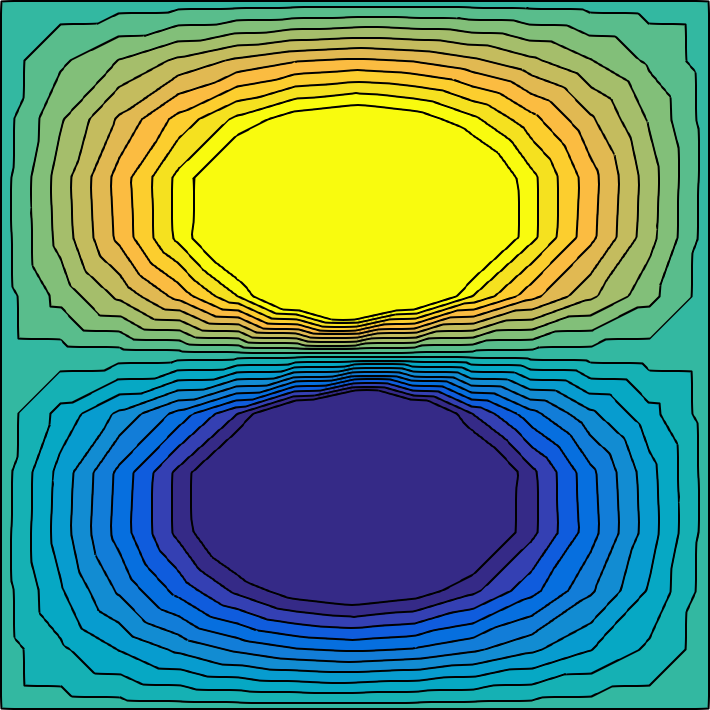}}
  \hfill
  \subfloat[
    \(L=2^1\)
  ]{\includegraphics[width=0.135\linewidth]{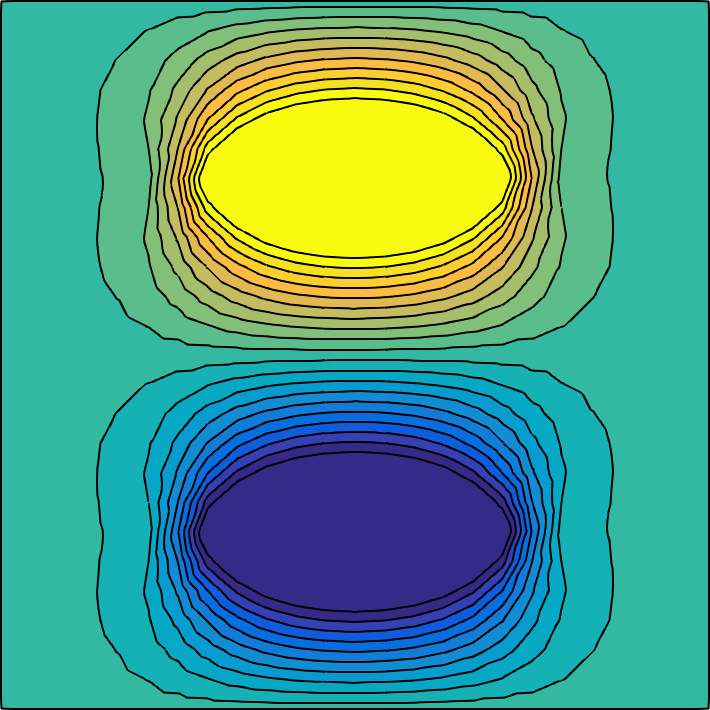}}
  \hfill
  \subfloat[
    \(L=2^2\)
  ]{\includegraphics[width=0.135\linewidth]{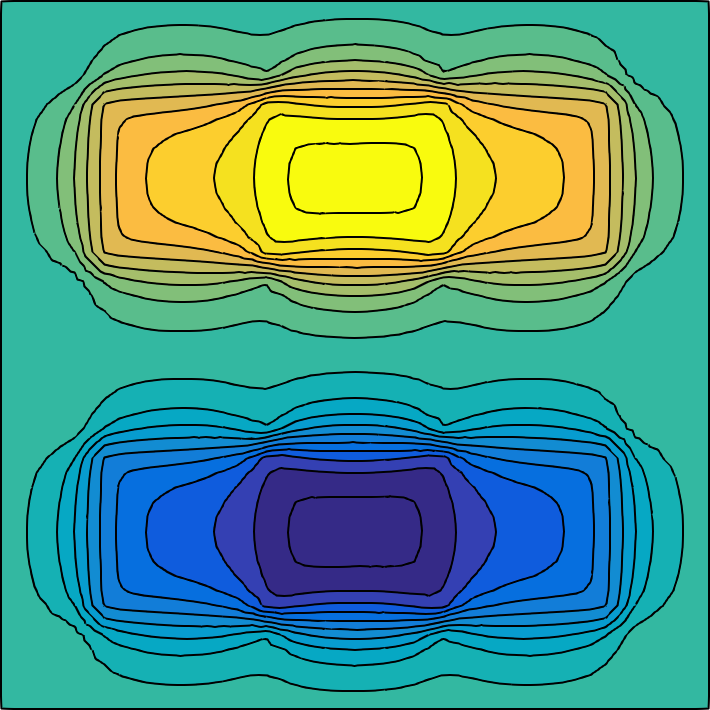}}
  \hfill
  \subfloat[
    \(L=2^3\)
  ]{\includegraphics[width=0.135\linewidth]{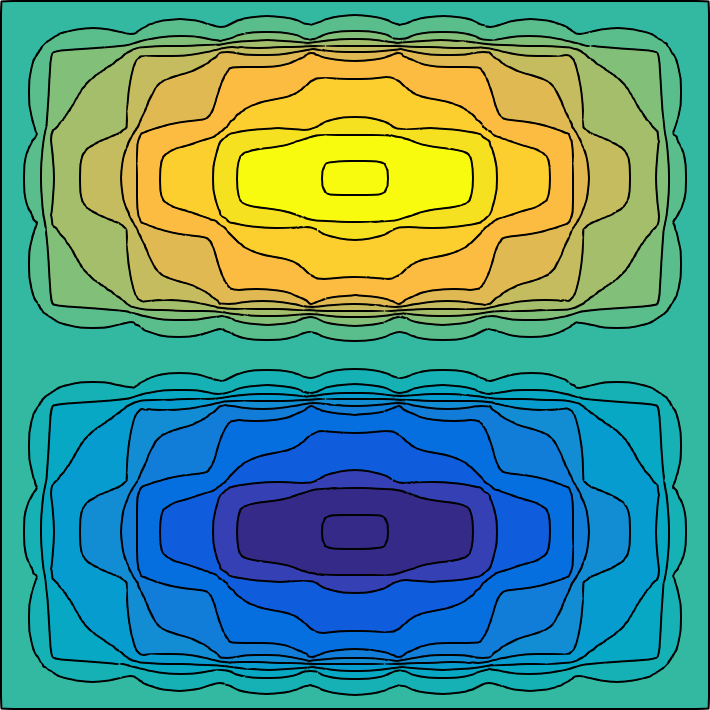}}
  \hfill
  \subfloat[
    \(L=2^4\)
  ]{\includegraphics[width=0.135\linewidth]{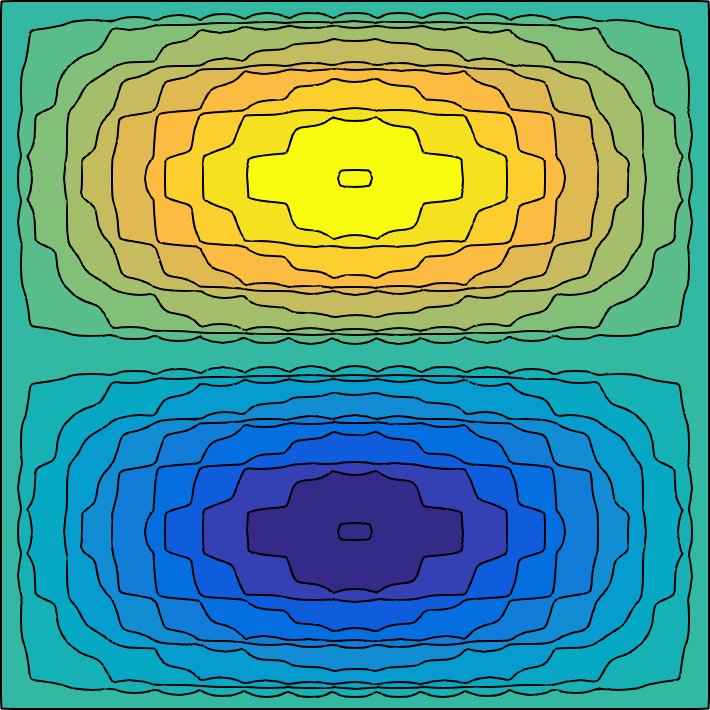}}
  \hspace{2pt}
  \raisebox{0.07\linewidth}{\raisebox{-3pt}{\(\overset{L \to \infty}{\rightarrow}\)}}
  \subfloat[
    \(u_0^{(3)}\)
  ]{\includegraphics[width=0.135\linewidth]{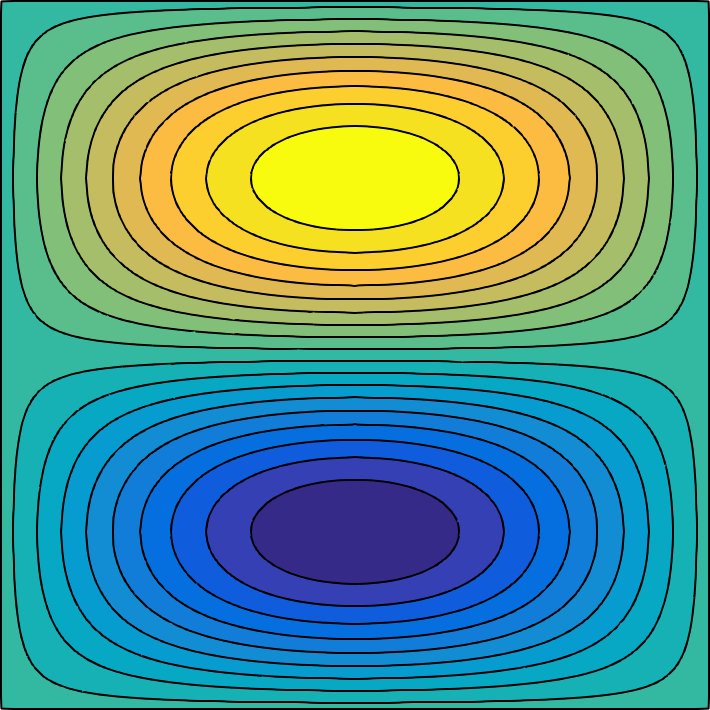}}
  \\[-2ex]
  \raisebox{0.07\linewidth}{\raisebox{-3pt}{\(u_{1/L,h}^{(4)}\):}}
  \subfloat[
    \(L=2^0\)
  ]{\includegraphics[width=0.135\linewidth]{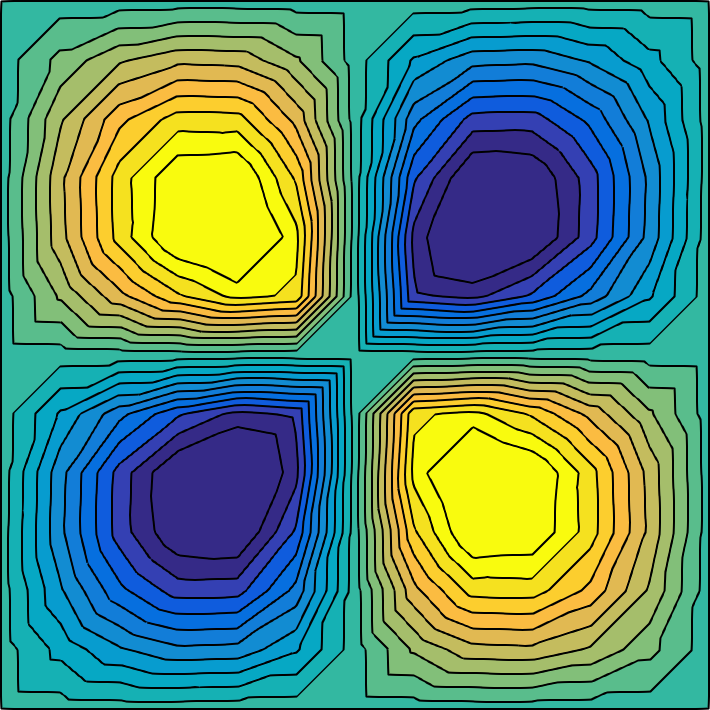}}
  \hfill
  \subfloat[
    \(L=2^1\)
  ]{\includegraphics[width=0.135\linewidth]{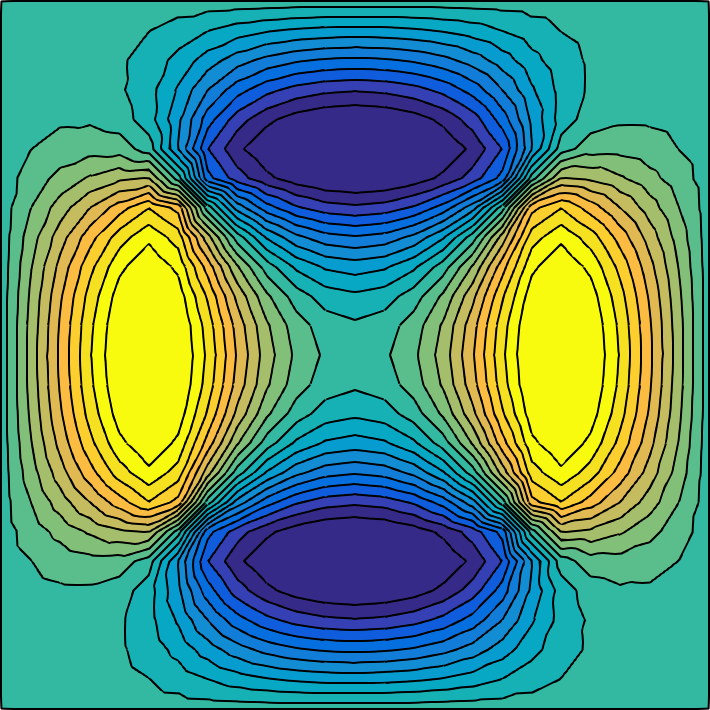}}
  \hfill
  \subfloat[
    \(L=2^2\)
  ]{\includegraphics[width=0.135\linewidth]{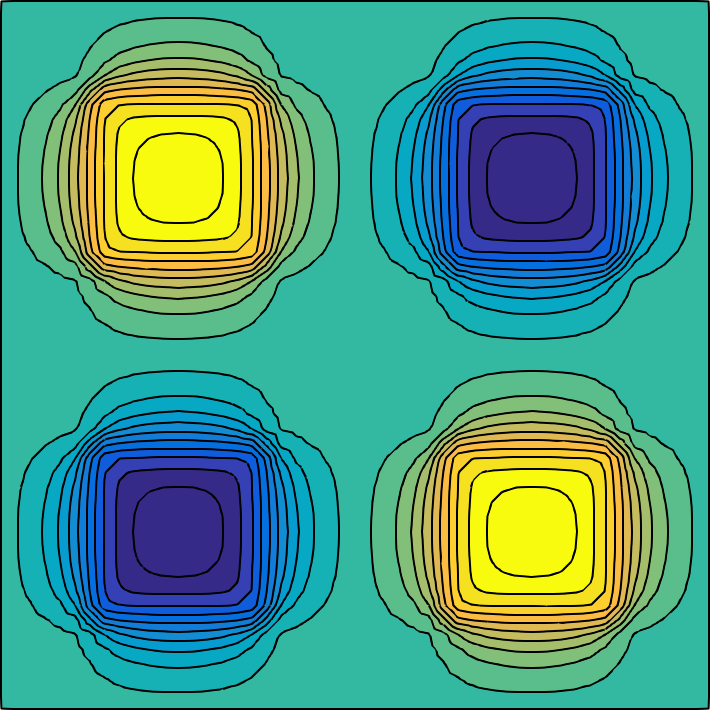}}
  \hfill
  \subfloat[
    \(L=2^3\)
  ]{\includegraphics[width=0.135\linewidth]{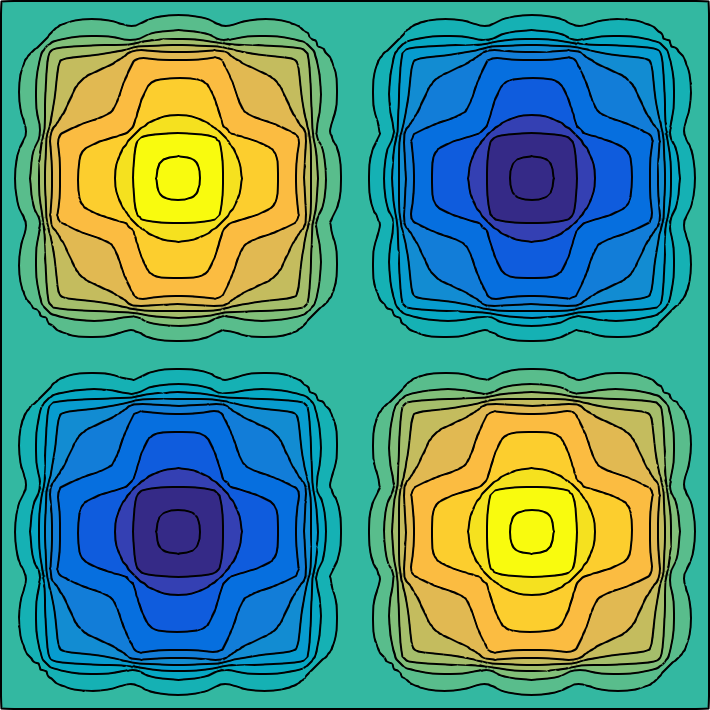}}
  \hfill
  \subfloat[
    \(L=2^4\)
  ]{\includegraphics[width=0.135\linewidth]{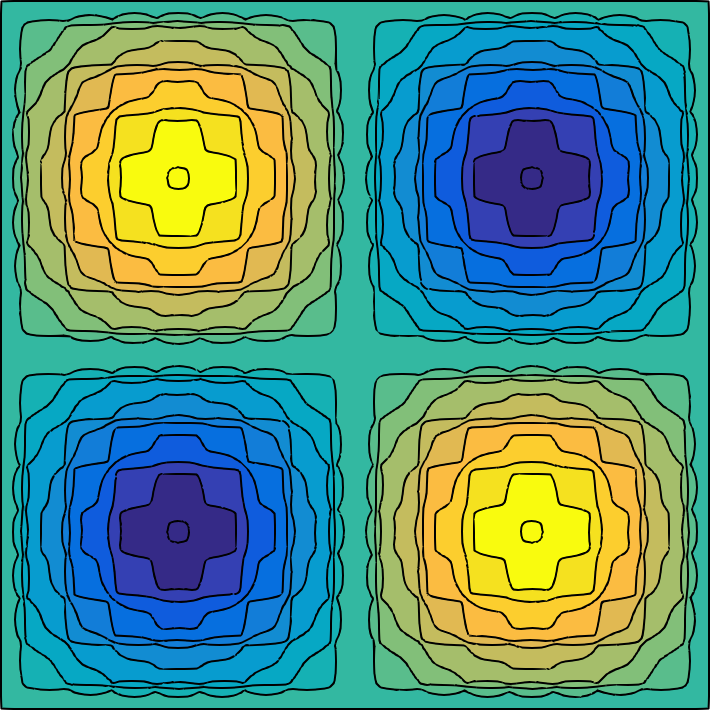}}
  \hspace{2pt}
  \raisebox{0.07\linewidth}{\raisebox{-3pt}{\(\overset{L \to \infty}{\rightarrow}\)}}
  \subfloat[
    \(u_0^{(4)}\)
  ]{\includegraphics[width=0.135\linewidth]{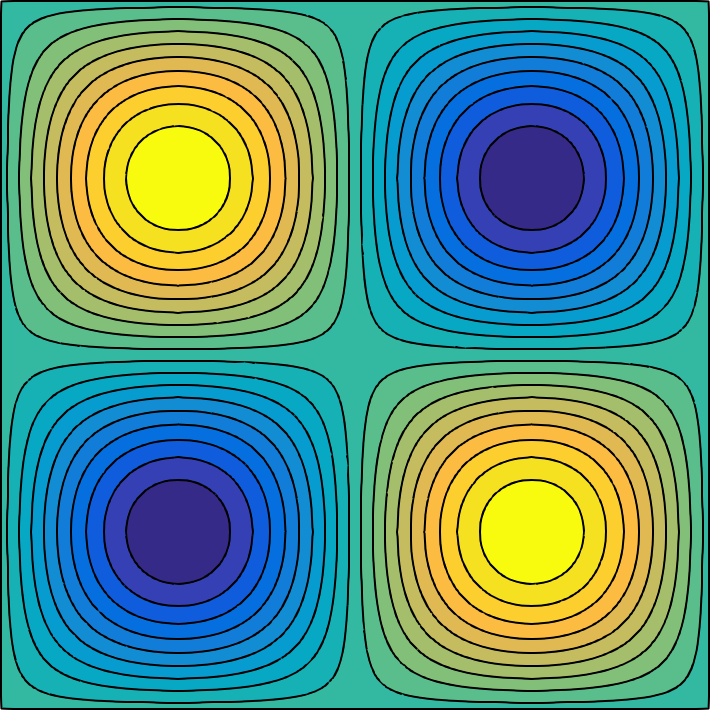}}
  \caption{
    The first four calculated eigenfunctions of the eigenvalue homogenization problem \cref{eq:numericalValidationHomogenizationTheoremEquation} converge weakly for \(L\to\infty\) to the solutions of the homogenized equation. The figure presents two-dimensional cut-planes through the middle of the domain at \(y_1=1/2\).
  }\label{fig:homogeniationEigenfunctions}
\end{figure}
\par
For the error comparison, we project the analytical solutions \cref{eq:homogenizedLimitAnalytSols} into the corresponding subspace \(\mathbb{H}_h\). Since \(\nu^{(2)} = \nu^{(3)}\) by our construction of \(\rho\), the corresponding eigenspace is two-dimensional, and the eigensolver returns some basis of this space. To resolve these spatial rotations and allow for an error comparison, we align the second and third eigenfunction by modifying their discrete eigenvectors with
\begin{equation}
  \boldsymbol{x}_h^{(2)} = \langle \boldsymbol{x}_h^{(2)} , \boldsymbol{x}_0^{(2)} \rangle \boldsymbol{x}_h^{(2)} + \langle \boldsymbol{x}_h^{(3)} , \boldsymbol{x}_0^{(2)} \rangle \boldsymbol{x}_h^{(3)}
  , \quad
  \boldsymbol{x}_h^{(3)} = \langle \boldsymbol{x}_h^{(2)} , \boldsymbol{x}_0^{(3)} \rangle \boldsymbol{x}_h^{(2)} + \langle \boldsymbol{x}_h^{(3)} , \boldsymbol{x}_0^{(3)} \rangle \boldsymbol{x}_h^{(3)}
  ,
\end{equation}
where \(\boldsymbol{x}_0^{(m)}\) denotes the \(m\)-th homogenized eigenvector. The resulting discrete eigenfunctions \(u_{1/L,h}^{(m)}\) for \(m=1,2,3,4\) are visualized in \cref{fig:homogeniationEigenfunctions} for \(L \in \{2^0,\dots,2^4\}\) together with the corresponding homogenized solutions \(u_0^{(m)}\). We can observe that for larger domain lengths \(L\), the eigenfunctions converge to their corresponding limits if we would neglect the oscillatory isolines that indicate strong gradients.
This observation corresponds to our theoretical results that the convergence is only weak when considering the \(H^1(\Omega_1)\)-norm. To quantify the convergence, we evaluate the relative \(L^2(\Omega_1)\)-error of the eigenfunctions and the relative eigenvalue error in \cref{fig:homogenization3dPlots} for \(L \in \mathbb{R}\) with a sampling rate of \(\Delta L = 0.1\). We measure a first-order converge for the \(L^2\)-error and at least first-order convergence for the eigenvalues. This observation matches the theoretical results from \cref{ssec:homogenization} since we proved strong convergence in \(L^2\) of the eigenfunctions and convergence of the eigenvalues to \(\nu^{(m)}\). We also examine the eigenvalues and their ratios \(\lambda_{1/L,h}^{(m)} / \lambda_{1/L,h}^{(m+1)}\) in \cref{fig:homogenization3dPlots}, where the degeneracy of \(m=2\), pre-asymptotic effects, and a non-monotonic convergence is visible. This observation confirms the prediction of \cref{cor:shiftedFundamentalRationUniformlyBounded} that the fundamental ratio can only be uniformly bounded for all \(L > L^*\) when pre-asymptotic effects have vanished.
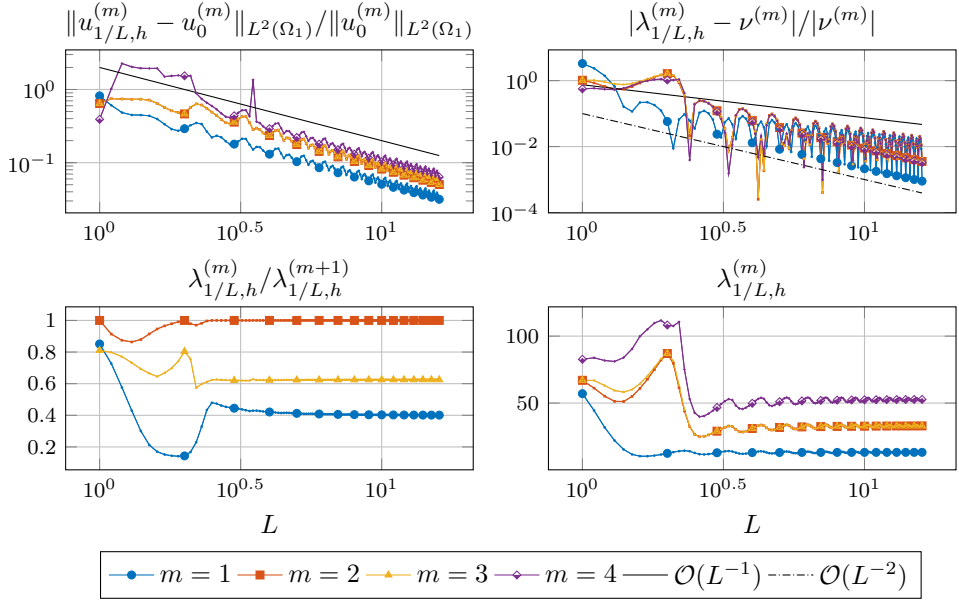
\begin{figure}[t]
  \centering
  \newcommand{\datapath}{./plots/homogenization-3d/errors}%



\begin{tikzpicture}

    \definecolor{color1}{rgb}{0,    ,0.4470,0.7410}
    \definecolor{color2}{rgb}{0.8500,0.3250,0.0980}
    \definecolor{color3}{rgb}{0.9290,0.6940,0.1250}
    \definecolor{color4}{rgb}{0.4940,0.1840,0.5560}
    \definecolor{color5}{rgb}{0.4660,0.6740,0.1880}
    \definecolor{color6}{rgb}{0.3010,0.7450,0.9330}
    \definecolor{color7}{rgb}{0.6350,0.0780,0.1840}
    \pgfplotscreateplotcyclelist{matlab}{
      color1,every mark/.append style={solid},mark=*\\
      color2,every mark/.append style={solid},mark=square*\\
      color3,every mark/.append style={solid},mark=triangle*\\
      color4,every mark/.append style={solid},mark=halfsquare*\\
      color5,every mark/.append style={solid},mark=pentagon*\\
      color6,every mark/.append style={solid},mark=halfcircle*\\
      color7,every mark/.append style={solid,rotate=180},mark=halfdiamond*\\
      color1,every mark/.append style={solid},mark=diamond*\\
      color2,every mark/.append style={solid},mark=halfsquare right*\\
      color3,every mark/.append style={solid},mark=halfsquare left*\\
    }

  \begin{groupplot}[
    group style={
      group size= 2 by 2,
      vertical sep=1.25cm,
    },
    xmode=log,
    ymode=log,
    legend style={font=\normalsize},
    legend pos=south west,
    legend columns=6,
    height=3.75cm,
    width=7cm,
    cycle list name=matlab,
    ticklabel style = {font=\footnotesize},
    grid=major,
  ]

    \nextgroupplot[
      title={$\|u_{1/L,h}^{(m)} - u_0^{(m)}\|_{L^2(\Omega_1)} / \|u_0^{(m)}\|_{L^2(\Omega_1)}$},
      title style={yshift=-0.25cm},
      legend to name=legend_homogenization,
      xtick={1,3.1622776601683795,10,31.622776601683793},
    ]
    \foreach \column in {1,...,4}{
      \addplot+[
        mark size = 1.5pt,
        mark repeat = 10,
        mark phase = 0,
      ] table [
        x index=0,
        y index=\column,
        col sep=comma,
      ] {\datapath/errorsl2.csv};
    }
    \foreach \column in {1,...,4}{

      \pgfmathsetmacro{\shift}{-4+\column-1};
      \pgfplotsset{cycle list shift=\shift}

      \addplot+[
        mark size = 0.3pt,
        forget plot  
      ] table [
        x index=0,
        y index=\column,
        col sep=comma,
      ] {\datapath/errorsl2.csv};
    }
    \addplot+[
      domain=1:16,
      color=black,
      mark=none
    ]{2*x^(-1)};
    \addplot+[ 
      domain=1:1,
      densely dash dot,
      color=black,
      mark=none
    ]{0.05*x^(-2)};
    \legend{$m=1$,$m=2$,$m=3$,$m=4$,$\mathcal{O}(L^{-1})$,$\mathcal{O}(L^{-2})$}

    \nextgroupplot[
      title={$|\lambda_{1/L,h}^{(m)} - \nu^{(m)} | / |\nu^{(m)}|$},
      title style={yshift=-0.25cm},
      xtick={1,3.1622776601683795,10,31.622776601683793},
    ]
    \foreach \column in {1,...,4}{
      \addplot+[
        mark size = 1.5pt,
        mark repeat = 10,
        mark phase = 0,
      ] table [
        x index=0,
        y index=\column,
        col sep=comma,
      ] {\datapath/errorsevals.csv};
    }
    \pgfplotsset{cycle list shift=-4}
    \foreach \column in {1,...,4}{
      \addplot+[
        mark size = 0.3pt,
      ] table [
        x index=0,
        y index=\column,
        col sep=comma,
      ] {\datapath/errorsevals.csv};
    }
    \addplot+[
      domain=1:16,
      color=black,
      mark=none
    ]{0.75*x^(-1)};
    \addplot+[
      domain=1:16,
      densely dash dot,
      color=black,
      mark=none
    ]{0.1*x^(-2)};

    \nextgroupplot[
      title={$\lambda_{1/L,h}^{(m)} / \lambda_{1/L,h}^{(m+1)}$},
      title style={yshift=-0.25cm},
      xtick={1,3.1622776601683795,10,31.622776601683793},
      xlabel={\(L\)},
      ymode=linear,
    ]
    \foreach \column in {1,...,3}{
      \addplot+[
        mark size = 1.5pt,
        mark repeat = 10,
        mark phase = 0,
      ] table [
        x index=0,
        y index=\column,
        col sep=comma,
      ] {\datapath/ratioevals.csv};
    }
    \pgfplotsset{cycle list shift=-3}
    \foreach \column in {1,...,3}{
      \addplot+[
        mark size = 0.3pt,
      ] table [
        x index=0,
        y index=\column,
        col sep=comma,
      ] {\datapath/ratioevals.csv};
    }

    \nextgroupplot[
      title={$\lambda_{1/L,h}^{(m)}$},
      title style={yshift=-0.25cm},
      xtick={1,3.1622776601683795,10,31.622776601683793},
      xlabel={\(L\)},
      ymode=linear,
    ]
    \foreach \column in {1,...,4}{
      \addplot+[
        mark size = 1.5pt,
        mark repeat = 10,
        mark phase = 0,
      ] table [
        x index=0,
        y index=\column,
        col sep=comma,
      ] {\datapath/eigval_nonhom.csv};
    }
    \pgfplotsset{cycle list shift=-4}
    \foreach \column in {1,...,4}{
      \addplot+[
        mark size = 0.3pt,
      ] table [
        x index=0,
        y index=\column,
        col sep=comma,
      ] {\datapath/eigval_nonhom.csv};
    }

  \end{groupplot}

  \coordinate (c3) at ($(group c1r1)!.5!(group c2r1)$);
  \node[below] at (c3 |- current bounding box.south){
    \pgfplotslegendfromname{legend_homogenization}
  };

\end{tikzpicture}

  \caption{
    Errors between the solution of \cref{eq:numericalValidationHomogenizationTheoremEquation} and the corresponding homogenized limit: We can observe the first-order convergence for all eigenfunctions in the \(L^2\)-norm and at least first-order convergence for the eigenvalues. The ratios between two adjacent eigenvalues reveal a degenerate state and a non-monotonic convergence for the fundamental ratio \(\lambda_{1/L,h}^{(1)}/\lambda_{1/L,h}^{(2)}\).
  }\label{fig:homogenization3dPlots}
\end{figure}

\subsection{The Quasi-Optimal Shift-And-Invert Preconditioner}\label{ssec:experimentUsingThePreconditioner}
To show the practical advantage of using the quasi-optimal preconditioning technique of \cref{ssec:asymptoticPreconditioner}, we compare the convergence histories of the IP and LOPCG method for the cases of no shift (\(\sigma=0\)), a good shift (\(\sigma=0.99\lambda_\infty\)), and the quasi-optimal shift (\(\sigma = \lambda_\infty\)). We then aim to solve \cref{eq:schroedingerEquation} on \(\Omega_L\) for \(\ell=1\) and an increasing \(L\). The quasi-optimal shift \(\lambda_\infty = \lambda_{\mathcal{B}_\#,\mathcal{B}_d,1,V}^{(1)}(\Omega_1)\) is obtained in constant time for all \(L\) since it only depends on the fixed unit cell \(\Omega_1\).
The calculations use \(\mathbb{Q}_1\) finite elements on a regular mesh with mesh size \(h=1/100\) and the \(x\)-periodic potential \(V(x,y) = 10^2 \sin{(\pi x)}^2 y^2\). We chose the start vectors \(\boldsymbol{x}_0 = \boldsymbol{1}, \boldsymbol{x}_{-1} = \boldsymbol{e}_1\). The solvers aim to reduce the spectral residual \(\boldsymbol{r}_k = \boldsymbol{A} \boldsymbol{x}_k - R_{\boldsymbol{A},\boldsymbol{B}}(\boldsymbol{x}_k) \boldsymbol{B} \boldsymbol{x}_k\) below the tolerance \(\texttt{TOL} = 10^{-10}\) and stop after 100 iterations. Both algorithms converged to the lowest eigenpair since the shifting strategy is order-preserving (c.f.~\cref{rem:absoluteOrderingPreserved}), and the start vector \(\boldsymbol{x}_0\) can not be orthogonal to the non-sign-changing ground state.
\par
The results in \cref{fig:preconditionerTest} indicate the drastic reduction in convergence speed for the unshifted algorithms. For the case of quasi-optimal preconditioning, both eigensolvers only need a couple of iterations to converge, as predicted by \cref{thm:fundamentalRatioBoundedNewNew}. When applying a good but not quasi-optimal shift of \(0.99 \lambda_\infty\), fast convergence rates for lower values of \(L\) can be observed. However, the convergence also deteriorates in the asymptotic limit of \(L \to \infty\). This fact underlines the requirement for \(\sigma\) to be the exact asymptotic limit if the method shall provide convergence in a fixed number of iterations for all possible \(L\). Furthermore, all three cases show a faster convergence of the LOPCG compared to the IP method.
\begin{figure}[t]
  \centering
  \newcommand{\datapath}{./plots/preconditioner-test/errors}%



\begin{tikzpicture}

    \definecolor{color1}{rgb}{0,    ,0.4470,0.7410}
    \definecolor{color2}{rgb}{0.8500,0.3250,0.0980}
    \definecolor{color3}{rgb}{0.9290,0.6940,0.1250}
    \definecolor{color4}{rgb}{0.4940,0.1840,0.5560}
    \definecolor{color5}{rgb}{0.4660,0.6740,0.1880}
    \definecolor{color6}{rgb}{0.3010,0.7450,0.9330}
    \definecolor{color7}{rgb}{0.6350,0.0780,0.1840}
    \pgfplotscreateplotcyclelist{matlab}{
      color1,every mark/.append style={solid},mark=*\\
      color2,every mark/.append style={solid},mark=square*\\
      color3,every mark/.append style={solid},mark=triangle*\\
      color4,every mark/.append style={solid},mark=halfsquare*\\
      color5,every mark/.append style={solid},mark=pentagon*\\
      color6,every mark/.append style={solid},mark=halfcircle*\\
      color7,every mark/.append style={solid,rotate=180},mark=halfdiamond*\\
      color1,every mark/.append style={solid},mark=diamond*\\
      color2,every mark/.append style={solid},mark=halfsquare right*\\
      color3,every mark/.append style={solid},mark=halfsquare left*\\
    }

  \begin{groupplot}[
    group style={
      group size= 3 by 2,
      horizontal sep=0.15cm,
      vertical sep=0.45cm,
    },
    ymin=1E-11,
    ymax=1E+1,
    xmode=linear,
    ymode=log,
    legend style={font=\normalsize},
    legend pos=south west,
    legend columns=6,
    height=3.75cm,
    width=5.15cm,
    cycle list name=matlab,
    ticklabel style = {font=\footnotesize},
    grid=major,
    ]

    \nextgroupplot[
      xmin=-1,xmax=21,
      title={$\sigma=0$},
      legend to name=legend_precondtioner_test,
    ]
    \foreach \Lx in {1.0,2.0,4.0,8.0,16.0,31.41}{
      \addplot+[
      ] table [
        x index=0,
        y index=2,
        col sep=comma,
        skip first n=1,
      ] {\datapath/errors_IP_000_\Lx.csv};
    }
    \legend{$L=1$,$L=2$,$L=4$,$L=8$,$L=16$,$L=31.41$}

    \nextgroupplot[
      xmin=-1,xmax=21,
      title={$\sigma=0.99\lambda_\infty$},
      yticklabels={},
    ]
    \foreach \Lx in {1.0,2.0,4.0,8.0,16.0,31.41}{
      \addplot+[
      ] table [
        x index=0,
        y index=2,
        col sep=comma,
        skip first n=1,
      ] {\datapath/errors_IP_099_\Lx.csv};
    }

    \nextgroupplot[
      xmin=-1,xmax=21,
      title={$\sigma=\lambda_\infty$},
      yticklabels={},
      ylabel=$\text{IP}_\sigma$: ${\|\boldsymbol{r}_k\|}_2$,
      yticklabel pos=right,
    ]
    \foreach \Lx in {1.0,2.0,4.0,8.0,16.0,31.41}{
      \addplot+[
      ] table [
        x index=0,
        y index=2,
        col sep=comma,
        skip first n=1,
      ] {\datapath/errors_IP_100_\Lx.csv};
    }

    \nextgroupplot[
      xmin=-1,xmax=21,
    ]
    \foreach \Lx in {1.0,2.0,4.0,8.0,16.0,31.41}{
      \addplot+[
      ] table [
        x index=0,
        y index=2,
        col sep=comma,
        skip first n=1,
      ] {\datapath/errors_LOPCG_000_\Lx.csv};
    }

    \nextgroupplot[
      xmin=-1,xmax=21,
      yticklabels={},
    ]
    \foreach \Lx in {1.0,2.0,4.0,8.0,16.0,31.41}{
      \addplot+[
      ] table [
        x index=0,
        y index=2,
        col sep=comma,
        skip first n=1,
      ] {\datapath/errors_LOPCG_099_\Lx.csv};
    }
    \nextgroupplot[
      xmin=-1,xmax=21,
      yticklabels={},
      ylabel=$\text{LOPCG}_\sigma$: ${\|\boldsymbol{r}_k\|}_2$,
      yticklabel pos=right,
    ]
    \foreach \Lx in {1.0,2.0,4.0,8.0,16.0,31.41}{
      \addplot+[
      ] table [
        x index=0,
        y index=2,
        col sep=comma,
        skip first n=1,
      ] {\datapath/errors_LOPCG_100_\Lx.csv};
    }

  \end{groupplot}

  \coordinate (c3) at ($(group c1r1)!.5!(group c3r1)$);
  \node[below] at (c3 |- current bounding box.south)
    {\pgfplotslegendfromname{legend_precondtioner_test}};

\end{tikzpicture}

  \caption{
    A comparison of the \(\text{IP}_{\sigma}\) and \(\text{LOPCG}_\sigma\) for the cases of \(\sigma=0\), \(\sigma=0.99 \lambda_\infty\), and \(\sigma=\lambda_\infty\) for different domain lengths \(L\).
  }\label{fig:preconditionerTest}
\end{figure}
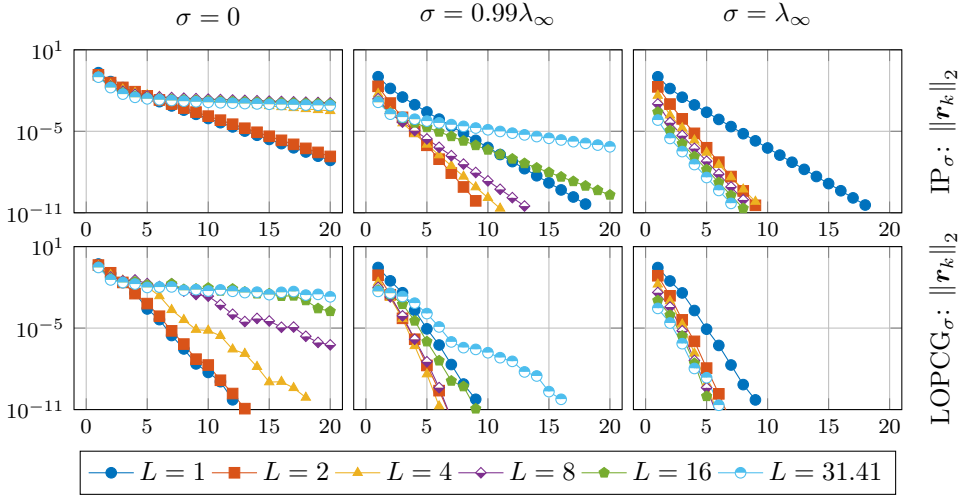

\subsection{Extension to Complex Domains: Barrier Principle and Defects in \texorpdfstring{\(\boldsymbol{x}\)}{x}-Direction}\label{ssec:experimentBarrierPotentialXDefects}
The initial box setup of \(\Omega_L = {(0, L)}^{p} \times {(0, \ell)}^{q}\) from \cref{sec:introduction} is very suitable for the mathematical analysis performed in \cref{sec:factorization}. However, we need to generalize the theory to more realistic domains for practical applications. Luckily, this can be quite intuitively done with some simple considerations.
\par
\begin{figure}[t]
  \centering
  \subfloat[
    The geometry of \(\tilde{\Omega}\).
  ]{



\begin{tikzpicture}[scale=1.105, rotate=0]

  \IfFileExists{./style.tex}%
  {\input{./style.tex}}
  {\input{plots/sketches/style.tex}}

  \pgfmathsetmacro{\solidThickness}{0.10};
  \pgfmathsetmacro{\R}{1.0};
  \pgfmathsetmacro{\N}{3};
  \pgfmathsetmacro{\overlap}{0.1};
  \pgfmathsetmacro{\centerdist}{2*(\R-\overlap)};
  \pgfmathsetmacro{\angleTopR}{acos(0.9)};
  \pgfmathsetmacro{\angleBotR}{360-acos(0.9)};
  \pgfmathsetmacro{\angleTopL}{acos(-0.9)};
  \pgfmathsetmacro{\angleBotL}{360-acos(-0.9)};
  \pgfmathsetmacro{\yOffset}{sin(\angleBotR)};

  \draw [fRegion] (0,-\R) rectangle (\N*\centerdist+2*\overlap,\R);
  \draw [bRegion] (0,-\R-\solidThickness) rectangle (\N*\centerdist+2*\overlap,-\R);
  \draw [bRegion] (0,+\R+\solidThickness) rectangle (\N*\centerdist+2*\overlap,+\R);
  \draw [bRegion] (0-\solidThickness,-\R-\solidThickness) rectangle (0,\R+\solidThickness);
  \draw [bRegion] (\N*\centerdist+2*\overlap,-\R-\solidThickness) rectangle (\N*\centerdist+2*\overlap+\solidThickness,\R+\solidThickness);

  \foreach \i in {1,...,\N}{
    \draw[point] (\R + \i * \centerdist - \centerdist,0) circle (0.05) node [below right] {\(\boldsymbol{c}_\i\)};
    \draw[point] (\R + \i * \centerdist - \centerdist,0) circle (0.05);
    \draw node at (\R + \i * \centerdist - \centerdist,-\R) [above] {\(\tilde{\Omega}_\i\)};
    \draw [fRegion]
    (\R + \i * \centerdist - \centerdist,0) ++(\angleTopL:\R) arc (\angleTopL:\angleTopR:\R);
    \draw [fRegion]
    (\R + \i * \centerdist - \centerdist,0) ++(\angleBotL:\R) arc (\angleBotL:\angleBotR:\R);
    \draw [fRegion,dashed] (\R + \i * \centerdist - 3*\centerdist/2,-\R) -- (\R + \i * \centerdist - 3*\centerdist/2,\R);
    \draw [fRegion,dashed] (\R + \i * \centerdist - \centerdist/2,-\R) -- (\R + \i * \centerdist - \centerdist/2,\R);
    \draw [fRegion] (\R + \i * \centerdist - 3*\centerdist/2,-\yOffset) -- (\R + \i * \centerdist - 3*\centerdist/2,\yOffset);
    \draw [fRegion] (\R + \i * \centerdist - \centerdist/2,-\yOffset) -- (\R + \i * \centerdist - \centerdist/2,\yOffset);
  }

  \draw [fRegion, fill=red, fill opacity=0.25]
  (\R + 1 * \centerdist - \centerdist,0) ++(\angleTopL:\R) arc (\angleTopL:\angleBotL:\R);
  \draw node [above right, red] at (\overlap,0) {\(\tilde{\Omega}_{\text{left}}\)};

  \draw [fRegion, fill=red, fill opacity=0.25]
  (\R + \N * \centerdist - \centerdist,0) ++(\angleBotR:\R) arc (\angleBotR-360:\angleTopR:\R);
  \draw node [above left, red] at (\R + \N * \centerdist - \centerdist + \R - \overlap,0) {\(\tilde{\Omega}_{\text{right}}\)};

  \draw[dim] (\R,0) -- ++(0,\R);
  \draw node [right] at (\R,0.5*\R) {\(R\)};
  \draw[dim] (\R,0) -- (\R+\R-\overlap,0);
  \draw node [above] at ({\R+0.5*(\R-\overlap)},0) {\(r\)};

\end{tikzpicture}

}
  \subfloat[
    A schematic discretization of \(\tilde{\Omega}\).\label{sfig:unionOfDiskWithDefectsMESH}
  ]{\includegraphics[width=0.5\linewidth]{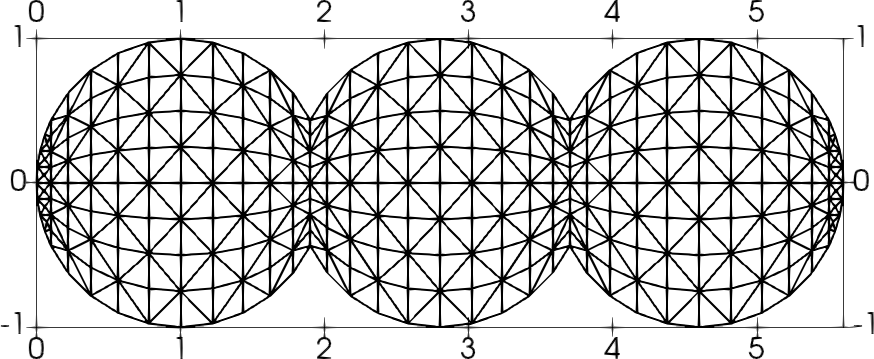}}
  \caption{
    A union of three disks (\(R=1\)) domain with defects in the \(x\)-direction and overlap of \(d=0.1\): \(\tilde{\Omega}\) comprises three identical unit cells \(\tilde{\Omega}_i\) and two domain defects \(\tilde{\Omega}_{\text{left}},\tilde{\Omega}_{\text{right}}\).
  }\label{fig:unionOfDiskWithDefects}
\end{figure}
\par
Consider, for example, the setup of \cref{fig:unionOfDiskWithDefects}, in which one aims to simulate a union of three disks \(\tilde{\Omega} = \bigcup_{i=1}^{3} B_{R}({(R+2(i-1)r,0)}^T) = \tilde{\Omega}_{\text{left}} \cup \left( \bigcup_{i=1}^{3} \tilde{\Omega}_i \right) \cup \tilde{\Omega}_{\text{right}}\)
where \(B_{R}(\boldsymbol{p})\) denotes a disk with radius \(R\) centered at \(\boldsymbol{p}\) and \(r = R - d\) with the overlap \(d\). These disks are all aligned along the \(x\)-axis and have a fixed overlap. We define the rectangular unit cell as the box with side lengths \(\{2r,2R\}\), where one disk is contained entirely. Inside this unit cell, we assume the potential as directional periodic. Furthermore, we have domain defects \(\tilde{\Omega}_{\text{left}}\) and \(\tilde{\Omega}_{\text{right}}\) that are not part of any unit cell on the left and the right side. In this setup, two problems arise \textendash\ the simulation of non-box-shaped domains and the handling of domain defects.

\subsubsection{Barrier Principle for an Optical Lattice Potential}
We could simulate the whole domain \(\Omega_{L=6r+2d}\) to overcome the first issue. However, we are only interested in the union-of-disks domain \(\tilde{\Omega}\), and a prescription of Dirichlet values on \(\partial \tilde{\Omega}\) might be problematic since it is inside the domain. It is well known that we can simply modify the potential \(V\) to achieve this setting. To avoid nontrivial values of \(\phi_L\) in certain regions, we can apply a significant penalty term to \(V\). We call this strategy the \textit{barrier principle}, which extends a given potential \(V\) to the barrier potential \(\tilde{V}(\boldsymbol{z};V,a) = V(\boldsymbol{z}) + a \chi_{\tilde{\Omega}^{\text{c}}}(\boldsymbol{z})\) where \(a \ge 0 \) is a penalty term. \(\chi_{\tilde{\Omega}^{\text{c}}}\) is the indicator function for the complement of \(\tilde{\Omega}\). For an increasing value of \(a \to \infty\), we can still apply our theory for any finite value of \(a\). In the limit case, the eigenvalue problem on the box-shaped domain \(\Omega_{6r+2d}\) is equivalent to an eigenvalue problem, purely posed on the subdomain \(\tilde{\Omega} \subset \Omega_{6r+2d}\).
\begin{figure}[t]
  \centering
  \includegraphics[width=0.49\linewidth]{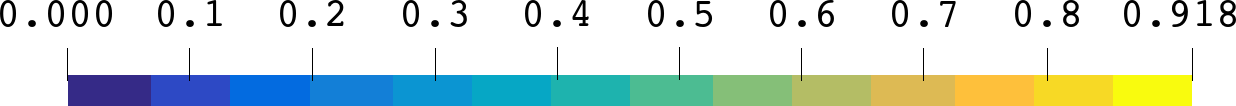}
  \includegraphics[width=0.49\linewidth]{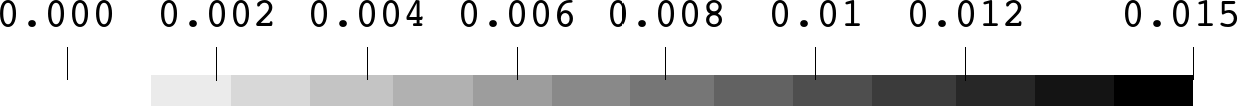}
  \\[-1ex]
  \subfloat[
    \(a=2^{0}\).\label{sfig:barrierPhi0}
  ]{\includegraphics[width=0.33\linewidth]{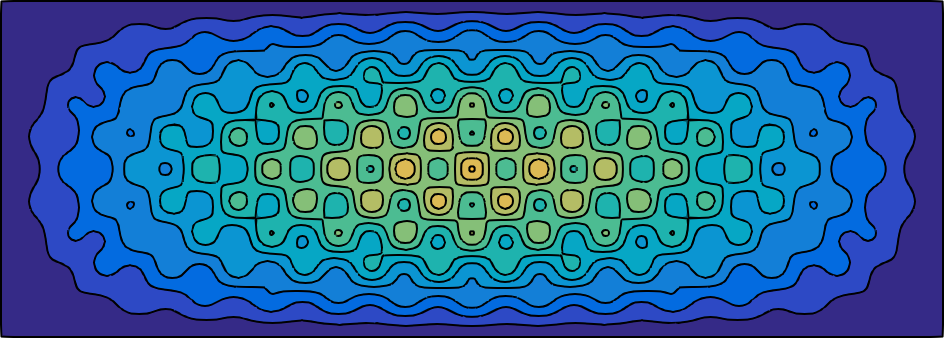}}
  \hfill
  \subfloat[
    \(a=2^{5}\).\label{sfig:barrierPhi1}
  ]{\includegraphics[width=0.33\linewidth]{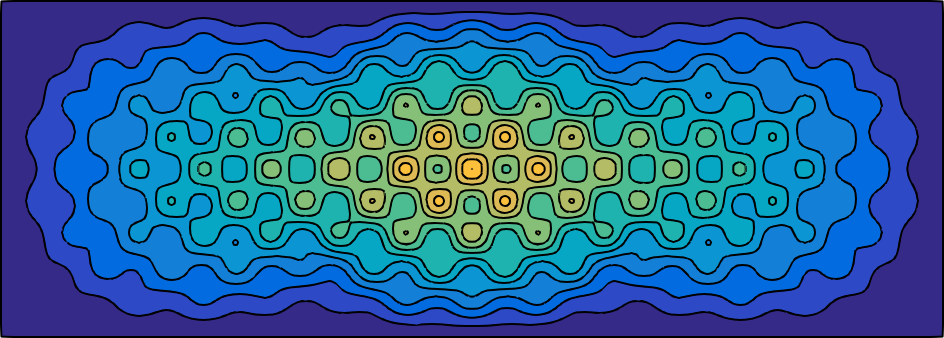}}
  \hfill
  \subfloat[
    \(a=2^{10}\).\label{sfig:barrierPhi2}
  ]{\includegraphics[width=0.33\linewidth]{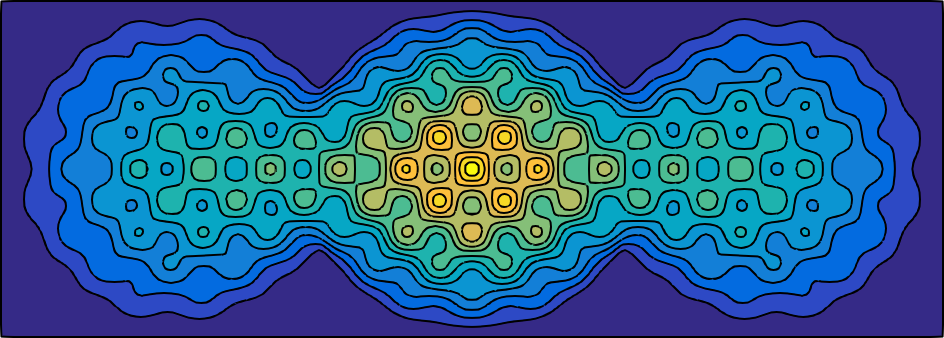}}
  \\[-1.5ex]
  \subfloat[
    \(a=2^{15}\).\label{sfig:barrierPhi3}
  ]{\includegraphics[width=0.33\linewidth]{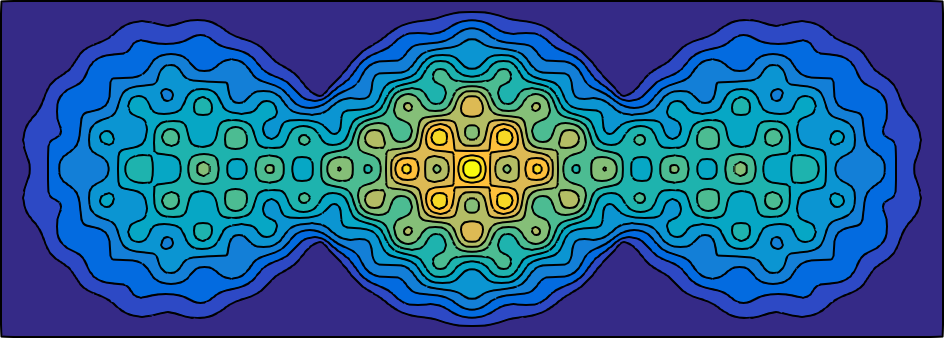}}
  \hfill
  \subfloat[
    \(a=2^{20}\).\label{sfig:barrierPhi4}
  ]{\includegraphics[width=0.33\linewidth]{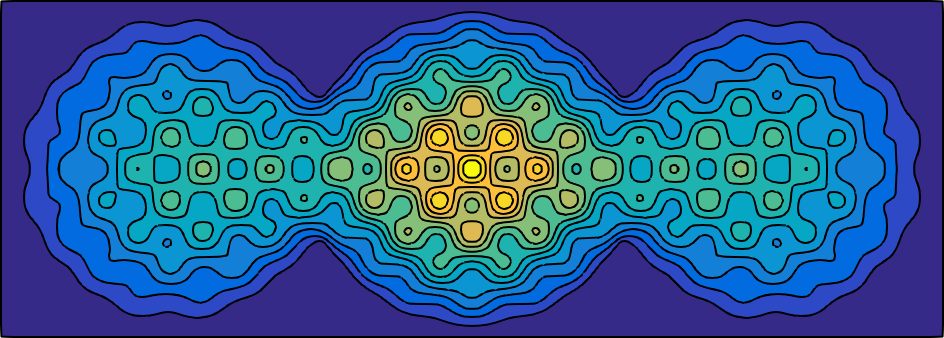}}
  \hfill
  \subfloat[
    \(\left|( \phi|_{a=2^{25}} - \phi|_{a=2^{20}} )\right|\).\label{sfig:barrierError}
  ]{\includegraphics[width=0.33\linewidth]{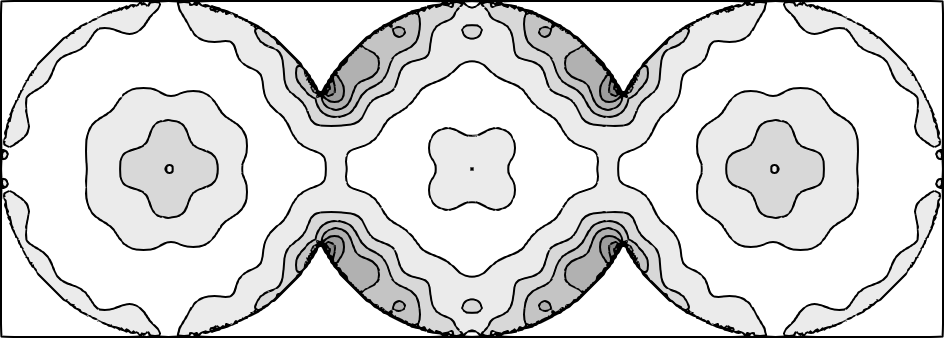}}
  \caption{
    Effect of Barrier Potential \(\tilde{V}(\boldsymbol{z};0,a)\) for varying penalty parameters \(a\) in the union-of-disks domain \(\tilde{\Omega}\) of \cref{fig:unionOfDiskWithDefects}. With increasing \(a\), the resulting problem statement reduces to the eigenproblem formulated in \(\tilde{\Omega}\). When comparing the change between \(a=2^{15}\) and \(a=2^{20}\) in \cref{sfig:barrierError}, the solution's overall change is small and focused on the connection points. Also, we see an interpolation error at the disk boundary since the underlying mesh is not boundary-aligned.
  }\label{fig:barrierTrickDomains}
\end{figure}
\par
To demonstrate the barrier effect of \(\tilde{V}(\boldsymbol{z};V,a)\), we inspect the union of three disks case from \cref{fig:unionOfDiskWithDefects} in combination with the optical lattice potential~\cite{henningSobolevGradientFlow2020}
\begin{equation}
  V(x,y) = 100 \left(
    1  -
    \sin{
      \tfrac{\omega \pi (x - d)}{2 (R-d)}
    }
    \sin{
      \tfrac{\omega \pi (y - (R - d))}{2 (R-d)}
    }
  \right)
  ,
\end{equation}
where \(\omega=9,R=1,d=0.1\) for various penalty terms \(a \in \{2^{0},2^{5},\dots,2^{20}\}\). \cref{fig:barrierTrickDomains} shows the first eigenfunction, which is calculated with the \(\text{LOPCG}_{\sigma=0}\) method for \(\texttt{TOL}=10^{-10}\) on a structured \(\mathbb{Q}_1\)-mesh with \(h=1/100\). It can be observed that with \(a\) increasing, the eigenfunction outside of \(\tilde{\Omega}\) approaches zero. However, these above considerations are purely theoretical. In practice, we directly exclude the regions \(\Omega_{6r+2d} \setminus \tilde{\Omega}\) and purely solve and mesh on \(\tilde{\Omega}\) as displayed in \cref{sfig:unionOfDiskWithDefectsMESH}.

\subsubsection{Principle of Defect Invariance}
When it comes to simulating the domain of \cref{ssec:experimentBarrierPotentialXDefects} with the quasi-optimally preconditioned eigensolver, we also need to calculate the shift \(\sigma = \lambda_\infty\). For domains without any domain defect that perfectly match the potential's period, this is done by simulating one unit cell \(\Omega_1\) with Dirichlet zero in \(\boldsymbol{y}\)- and periodic boundary conditions in \(\boldsymbol{x}\)-direction as theoretically derived in \cref{thm:factorization,thm:asymptoticBehaviorOfXDirectionNew}.
\par
For the case of defects located at the extremities of the expanding \(\boldsymbol{x}\)-direction as a subset of an imaginary unit cell \(\Omega_1\), we can do the same (if the potential acts as usual in the defect regions). The limit eigenvalue does not change since we can prove:
\begin{theorem}[Principle of Defect Invariance]\label{thm:defectPrinciple}
  Let \(\Omega_{L+2\delta}\) with \(L\) denoting the \(\boldsymbol{x}\)-period of the potential \(V\) and \(\delta < L\) be given. For the linear Schrödinger eigenvalue problem \cref{eq:schroedingerEquation} posed on \(\Omega_{L+2\delta}\), it still holds that \(\lim_{L \to \infty} \lambda_L^{(m)} = \lambda_{\varphi_y}\).
\end{theorem}
\begin{proof}
  Consider \(\Omega_L \subset \Omega_{L+2\delta} \subset \Omega_{L+2}\). Then, by the inclusion principle~\cite[p13]{henrotExtremumProblemsEigenvalues2006} for elliptic operators with Dirichlet boundary conditions, we have
  \begin{align}
    \lambda_{\mathcal{B}_d,\mathcal{B}_d,0,V}^{(m)}(\Omega_{L})
    \le
    \lambda_{\mathcal{B}_d,\mathcal{B}_d,0,V}^{(m)}(\Omega_{L+2\delta})
    \le
    \lambda_{\mathcal{B}_d,\mathcal{B}_d,0,V}^{(m)}(\Omega_{L+2})
  \end{align}
  Using the factorizations of \cref{thm:factorization}, this is equivalent to
  \begin{align}
    \lambda_{\varphi_y}^{(1)}(\Omega_{L}) + \lambda_{u_{y,2}}^{(m)}(\Omega_{L})
    \le
    \lambda_{\mathcal{B}_d,\mathcal{B}_d,0,V}^{(m)}(\Omega_{L+2\delta})
    \le
    \lambda_{\varphi_y}^{(1)}(\Omega_{L+2}) + \lambda_{u_{y,2}}^{(m)}(\Omega_{L+2})
    .
  \end{align}
  Since \(\lambda_{\varphi_y}^{(1)}(\Omega_{L}) = \lambda_{\varphi_y}^{(1)}(\Omega_{L+2}) = \lambda_{\varphi_y}^{(1)}(\Omega_1)\) and \(\lambda_{u_{y,2}}^{(m)}(\Omega_{L}), \lambda_{u_{y,2}}^{(m)}(\Omega_{L+2}) \in \mathcal{O}(1/L^2)\) by \cref{thm:asymptoticBehaviorOfXDirectionNew}, we conclude with the Sandwich Lemma.
\end{proof}
We are now prepared to solve the union-of-disks geometry in the next \cref{ssec:experiment1dChain}.

\subsection{Chain Model with Truncated Coulomb Potential in Two Dimensions}\label{ssec:experiment1dChain}
Consider the domain \(\tilde{\Omega}_N = \bigcup_{i=1}^{N} B_{R}\left({(R+2(i-1)r,0)}^T\right)\) with the parameters \(R=1, r=0.9\). Chain-like molecules in the context of molecular simulations inspire this model. For real applications, the potential is a Coulomb potential. Since a singularity of \(V\) violates the assumption (A2), we use a truncated Coulomb potential as
\begin{equation}\label{eq:truncatedCoulombPotential}
  V_{\text{C,lim}}(\boldsymbol{z}; b)
  =
  \begin{cases}
    - \tfrac{Z}{{\|\boldsymbol{z}\|}_2} & \text{for } {\|\boldsymbol{z}\|}_2 \ge b
    \\
    - \tfrac{Z}{b} & \text{for } {\|\boldsymbol{z}\|}_2 < b
  \end{cases}
  ,
\end{equation}
to mimic, e.g., the electrostatic potential with charge \(Z>0\). We also have to neglect long-range interaction to fulfill the periodicity assumption on \(V\). Consider the \(N\) centers \({\{\boldsymbol{c}_i\}}_{i=1}^N\) with \(\boldsymbol{c}_i = (R + (i-1) 2r,0)\). We prescribe the compound periodic potential \(V(\boldsymbol{z}) = \sum_{i \in \{i : |\boldsymbol{z}-\boldsymbol{\tilde{c}}_i|<R\} } V_{\text{C,lim}}(\boldsymbol{z} - \boldsymbol{\tilde{c}}_i; b)\) where \(\boldsymbol{\tilde{c}}_i \in {\{\boldsymbol{c}_i\}}_{i=1}^N \cup \{\boldsymbol{c}_1 - (2r,0)\} \cup \{\boldsymbol{c}_N + (2r,0)\} \) include ghost centers to fulfill the periodicity assumption (A1) also in the defect regions. We also note that the semi-positivity assumption (A3) is violated. However, it turned out that the resulting spectrum is still positive, the operator, thus, elliptic, and our theory is applicable.
Nevertheless, it would also be possible to fulfill the assumption (A3) by adding a positive constant to the \(L^\infty\)-potential without changing the resulting eigenfunctions.
\par
A series of computations for \(N=1,2,4,\dots,32\) is performed using the \(\text{LOPCG}_\sigma\) method for the potential parameters \(Z=1,b=10^{-4}\) and the tolerance \(\texttt{TOL} = 10^{-10}\). As shown in \cref{sfig:unionOfDiskWithDefectsMESH}, the spatial discretization uses unstructured but symmetrically meshed \(\mathbb{P}_1\) elements.
The quasi-optimal shift calculation uses the unit cell \(\tilde{\Omega}_0 = \tilde{\Omega}_1 \cap [R-r,R-r] \times [-R,R]\) with periodic boundary conditions in \(x\)-direction and zero boundary conditions on the rest. Thus, \(\sigma = \lambda_{\varphi_y}(\tilde{\Omega}_0) \approx 1.08784\) can be computed a priori in \(\mathcal{O}(1)\) since the unit cell is independent of \(N\).
\par
In \cref{fig:simulationUnionOfSpeheres}, the resulting ground state eigenfunctions \(\phi_h^{(1)}\) are presented with the solution to the base problem in \cref{sfig:simulationUnionOfSpeheresLIMIT}. We can observe for increasing \(N\) that the solution inside a single disk approaches the shape of the solution to the unit cell problem. This convergence is the expected behavior, as shown theoretically in \cref{sec:factorization}. In \cref{tab:unionOfDisksSummary}, we observe that, due to the quasi-optimal preconditioning, the number of iterations needed to meet the required residual tolerance is of order \(\mathcal{O}(1)\). Furthermore, the method shows a linearly scaling behavior since the ratio of calculation time to the disk amount \(t_{\text{eig}}/N\) seems to be independent of \(N\).
\begin{figure}[t]
  \centering
  \includegraphics[width=0.5\linewidth]{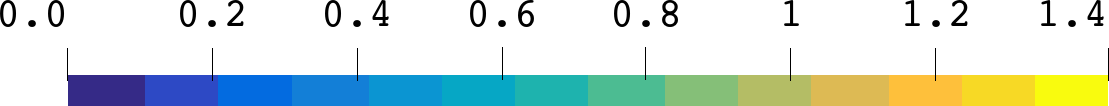}
  \\[-1ex]
  \subfloat[
    Limit.\label{sfig:simulationUnionOfSpeheresLIMIT}
  ]{\includegraphics[height=0.13\linewidth]{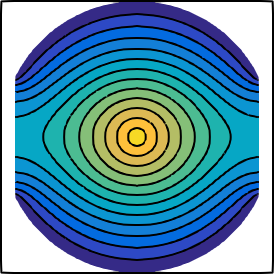}}
  \hfill
  \subfloat[
    \(N=1\).\label{sfig:simulationUnionOfSpeheresN1}
  ]{\includegraphics[height=0.13\linewidth]{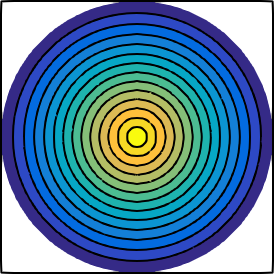}}
  \hfill
  \subfloat[
    \(N=2\).\label{sfig:simulationUnionOfSpeheresN2}
  ]{\includegraphics[height=0.13\linewidth]{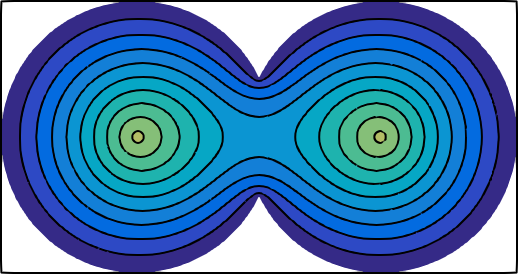}}
  \hfill
  \subfloat[
    \(N=4\).\label{sfig:simulationUnionOfSpeheresN4}
  ]{\includegraphics[height=0.13\linewidth]{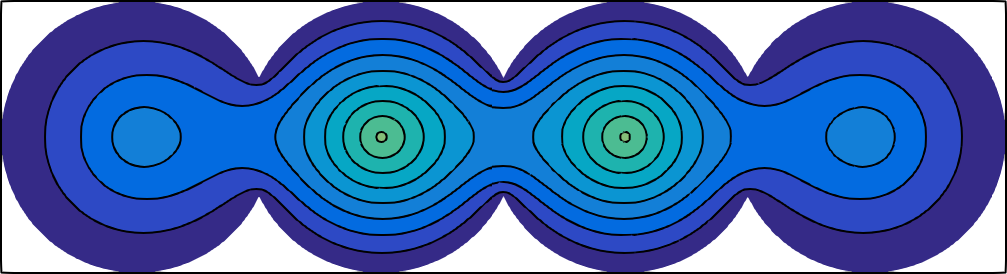}}
  \\
  \subfloat[
    \(N=8\).\label{sfig:simulationUnionOfSpeheresN8}
  ]{\includegraphics[height=0.13\linewidth]{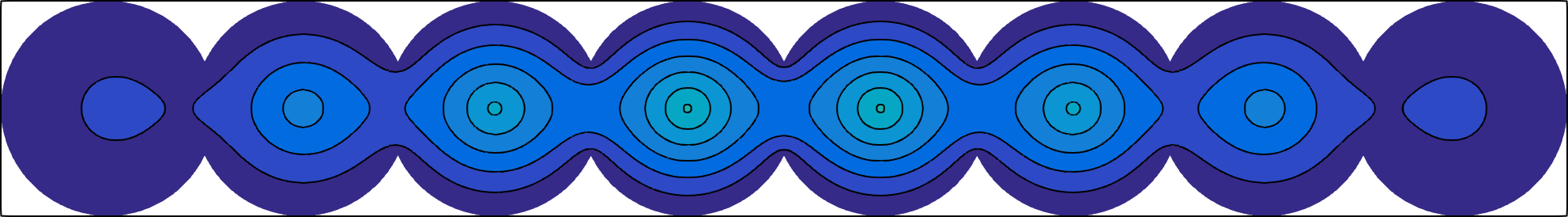}}
  \caption{
    Contours of the first eigenfunction for the union of \(N\) disks using the truncated Coulomb potential without long-range interactions: \cref{sfig:simulationUnionOfSpeheresLIMIT} shows the asymptotic limit eigenfunction with periodic boundary conditions in the \(x\)-direction.
  }\label{fig:simulationUnionOfSpeheres}
\end{figure}
\begin{table}[t]
  \small
  \centering
  \caption{
    The summary of computations for the union of \(N\) disks with the truncated Coulomb potential. Due to the same discretization density, the number of nodes \(n_{\text{nodes}}\) for each mesh is approximately proportional to the number of disks \(N\) (up to the defects). The wall times are measured on an \text{Intel X7542} CPU using one core.
  }\label{tab:unionOfDisksSummary}
  \begin{filecontents*}{./tables/union-of-spheres.csv}
N,nn,lambda1,maxPhi1,kit,cond,t,tnrat
1,89869,1.96221975928055,1.4372256850032725,5,482709.22099912376,1.832185788,1.832185788
2,169007,1.4691224644626166,0.9606357715969978,5,1.2515768423790515e6,3.945159607,1.9725798035
4,327337,1.2001283795670163,0.8499356757995072,5,4.254264253381052e6,7.695727942,1.9239319855
8,643961,1.1176784020502912,0.6384751316147151,5,1.1361436870426428e7,17.025766558,2.12822081975
16,1277191,1.095501416965797,0.4597093658340914,5,2.5521502145545986e7,33.056710913,2.0660444320625
32,2543705,1.0897810831476469,0.3270150222065116,5,5.33084832837862e7,64.960355982,2.0300111244375
\end{filecontents*}
\pgfplotstabletypeset[
    col sep = comma,
    every head row/.style={before row=\toprule,after row=\midrule},
    every last row/.style={after row=\bottomrule},
    columns/N/.style={column name=\(N \propto L\)},
    columns/nn/.style={column name=\(n_{\text{nodes}}\)},
    columns/lambda1/.style={precision=5, column name=\(\lambda_h^{(1)}\)},
    columns/maxPhi1/.style={column name=\(\max \phi_h^{(1)}\)},
    columns/kit/.style={column name=\(k_{\text{it}}\)},
    columns/cond/.style={column name=\(\tfrac{\lambda_{\max}(\boldsymbol{A} - \sigma \boldsymbol{B})}{\lambda_{\min}(\boldsymbol{A} - \sigma \boldsymbol{B})}\)},
    columns/t/.style={column name=\(t_{\text{eig}}\) \([s]\)},
    columns/tnrat/.style={column name=\(\tfrac{t_{\text{eig}}}{N}\) \([s]\)},
  ]
  {./tables/union-of-spheres.csv}
\end{table}

\subsection{Plane Model with Kronig--Penney Potential in Three Dimensions}\label{ssec:experiment2dPlane}
Finally, we also show the preconditioner's quasi-optimality for a three-dimensional case with two expanding directions (\(p=2\), \(q=1\)). We use a three-dimensional Kronig--Penney potential~\cite{vargaComputationalNanoscienceApplications2011}, defined by \(V(\boldsymbol{z}) = 0\) for \({\|\boldsymbol{z} \bmod 1 - \boldsymbol{1}/2 \|}_1 < 1/4\) and \(V(\boldsymbol{z}) = 100\) otherwise,
where \(\boldsymbol{1}\) denotes the vector of ones in \(d\) dimensions. This potential represents cubic wells with sidelength \(0.5\) centered in the unit cubes, which form the plane-like expanding domain \(\Omega_L={(0,L)}^2 \times {(0,1)}\) with \(L \in \mathbb{N}\). We again calculate the quasi-optimal shift \(\sigma = \lambda_{\mathcal{B}_\#,\mathcal{B}_d,1,V}^{(1)}(\Omega_1)\) on the unit cube and use it to precondition the \(\Omega_L\)-problem. Finally, both problems are discretized using a uniform mesh size of \(h=1/10\) and \(\mathbb{Q}_1\) elements resulting in \(\sigma \approx 57.60485\). We use the \(\text{LOPCG}_{\sigma}\) method with \(\texttt{TOL}=10^{-10}\) and solve for the ground state solution.
The simulations are performed on a series of domains \(\Omega_L\) with \(L \in \{1,2,4,\dots,32\}\). In \cref{tab:2dKronigPenney}, we observe that the number of eigensolver iterations \(k_{\text{it}}\) does not increase for \(L \to \infty\), confirming our theory. However, in contrast to \cref{tab:unionOfDisksSummary}, a slight increase in solution time per number of unit cells (\(t_{\text{eig}}/L^2\)) can be observed. This increase is the expected behavior of using a direct solver for sparse matrices with increased bandwidth for \(L \to \infty\) for our case of \(p=2\) expanding directions.
\par
\begin{table}[t]
  \small
  \centering
  \caption{
    The summary of computations for the plane-like expanding domain in three directions with the Kronig--Penney potential. The number of unit cells \(N\) now scales quadratically with \(L\).
  }\label{tab:2dKronigPenney}
  \begin{filecontents*}{./tables/2d-kronig-penney.csv}
L,nn,lambda1,maxPhi1,kit,t,tperdofs
1,1331,58.999149326997326,4.709667863808572,5,0.164487483,0.164487483
2,4851,58.30880813989642,2.3102930080204964,6,0.519474064,0.129868516
4,18491,57.8118627559417,1.9451361362402826,6,3.009973672,0.1881233545
8,72171,57.65869612835613,1.0914770155525597,7,15.281724323,0.238776942546875
16,285131,57.618445957705575,0.561262735756492,7,73.329060343,0.28644164196484373
32,1133451,57.6082586779154,0.2825927647388176,6,320.636602697,0.31312168232128906
\end{filecontents*}
\pgfplotstabletypeset[
    col sep = comma,
    every head row/.style={before row=\toprule,after row=\midrule},
    every last row/.style={after row=\bottomrule},
    columns/L/.style={column name=\(L\)},
    columns/nn/.style={column name=\(n_{\text{nodes}}\)},
    columns/lambda1/.style={precision=5, column name=\(\lambda_h^{(1)}\)},
    columns/maxPhi1/.style={column name=\(\max \phi_h^{(1)}\)},
    columns/kit/.style={column name=\(k_{\text{it}}\)},
    columns/t/.style={column name=\(t_{\text{eig}}\) \([s]\)},
    columns/tperdofs/.style={column name=\(\tfrac{t_{\text{eig}}}{L^2}\) \([s]\)},
  ]
  {./tables/2d-kronig-penney.csv}
\end{table}

\section{Conclusion}\label{sec:conlusion}
This work presented a quasi-optimal shift-and-invert preconditioner to solve the linear periodic Schrödinger eigenvalue problem in a constant number of eigensolver iterations for domains expanding periodically in a subset of directions. First, we analyzed and proved the quasi-optimality of the method using factorization and homogenization techniques. The analysis revealed powerful insights into the behavior of the eigenfunctions and eigenvalues. Significantly, the representation of the searched eigenfunction as the product of easy-to-calculate functions leads to a decisive result \textendash\ the corresponding eigenvalues can be expressed as the sum of other eigenvalues, which can be much easier computed in practice than solving the whole system. This realization makes the proposed method very practical since calculating the quasi-optimal shift can be done in \(\mathcal{O}(1)\). We then extended the results to complex and defect domain shapes to allow for a broader range of geometrical applications. Finally, in our experiments, we showed the practical usability of the method for chain-like and plane-like expanding domains.
\par
Limitations of our method include the assumptions on the potential \(V\) to be essentially bounded and periodic. Also, we observed that using the perfect shift in the eigensolver algorithms leads naturally to an ill-conditioned system matrix \((\boldsymbol{A} - \sigma \boldsymbol{B})\). Thus, future work could weaken the periodicity assumptions on \(V\) by, e.g., allowing for a perturbation of \(\delta V\) that vanishes in the limit \(L \to \infty\). Also, efficient solvers for the linear system involving \((\boldsymbol{A} - \sigma \boldsymbol{B})\) must be constructed when the system size requires iterative linear solvers.



\bibliographystyle{siamplain}
\bibliography{refs}
\end{document}